\documentclass[11pt,a4paper]{amsart}
\usepackage[T1]{fontenc}
\usepackage[utf8]{inputenc}
\usepackage{lmodern}
\usepackage[a4paper,margin=28mm]{geometry}
\usepackage{amsmath,amssymb,amsfonts,mathtools}
\usepackage{enumitem,booktabs,array,longtable,microtype,needspace}
\usepackage{hyperref}
\hypersetup{hidelinks,pdftitle={Quantitative homogenization and large-scale regularity for nondivergence-form equations under a critical ellipticity moment},pdfauthor={Jizu Huang and Yong Ma}}
\theoremstyle{plain}
\newtheorem{theorem}{Theorem}[section]
\newtheorem{lemma}[theorem]{Lemma}
\newtheorem{proposition}[theorem]{Proposition}
\newtheorem{corollary}[theorem]{Corollary}
\theoremstyle{definition}
\newtheorem{definition}[theorem]{Definition}
\newtheorem{assumption}[theorem]{Assumption}
\theoremstyle{remark}
\newtheorem{remark}[theorem]{Remark}

\newcommand{\fint}{\mathop{\mathchoice{\displaystyle\int}{\textstyle\int}{\scriptstyle\int}{\scriptscriptstyle\int}}\mkern-14mu-\mkern8mu}
\numberwithin{equation}{section}
\allowdisplaybreaks[2]
\setlist{itemsep=.2em,topsep=.3em}
\title[Critical-moment nondivergence homogenization]{Quantitative homogenization and large-scale regularity for nondivergence-form equations under a critical ellipticity moment}
\author[J. Huang]{Jizu Huang}
\author[Y. Ma]{Yong Ma}
\address[Jizu Huang and Yong Ma]{SKLMS, Academy of Mathematics and Systems Science, Chinese Academy of Sciences, Beijing 100190, P.R. China}
\address[Jizu Huang and Yong Ma]{School of Mathematical Sciences, University of Chinese Academy of Sciences, Beijing 100190, P.R. China}
\email[Yong Ma]{mayong@amss.ac.cn}
\thanks{Corresponding author: Yong Ma (\texttt{mayong@amss.ac.cn}).}
\date{September 30, 2026}

\begin{document}
\begin{abstract}
We prove quantitative homogenization estimates for linear elliptic equations
in nondivergence form with stationary, symmetric coefficients, finite range
of dependence, and a deterministic upper ellipticity bound. No deterministic
positive lower bound is imposed. We assume that the reciprocal of the
infimum of the smallest eigenvalue on a unit ball has a finite moment of
order $d$, together with the common continuity condition of Armstrong and
Smart. Under these assumptions, we obtain algebraic probability bounds for
finite-cell errors and for Dirichlet homogenization errors with nonzero
sources. We construct quadratic correctors on the whole space, modulo
 affine functions, and prove first- and second-order large-scale regularity
and the corresponding Liouville theorems. We also identify the effective
matrix through the invariant density and quantify smooth spatial averages
of the density and the weighted coefficients. The proof separates a
unit-trace diffusion from its physical clock. A reverse H\"older estimate
for the Green function of a stopped coarse process yields the integrability
gain needed to control the clock at the critical moment. Applications
include weighted gradient convergence and a finite-domain approximation of
the effective matrix.
\end{abstract}
\maketitle
\tableofcontents
\medskip

\section{Introduction}\label{sec:introduction}

We consider linear elliptic equations in nondivergence form,
\begin{equation}\label{eq:model}
 -A(x,\omega):D^2u=0\qquad\text{in }\mathbb R^d,
\end{equation}
with stationary, symmetric coefficients satisfying $0<A\le I$. The
ellipticity lower bound may approach zero, and the eigendirections may vary
with position. We assume finite range of dependence and the critical
moment condition
\begin{equation}\label{intro:critical-moment}
 \mathbb E\left[\left(\inf_{x\in B_1}\lambda_{\min}A(x)\right)^{-d}\right]
 <\infty.
\end{equation}
Our aim is to quantify homogenization under this condition and to describe
the large-scale structure of solutions. The main results are algebraic
homogenization errors, whole-space quadratic correctors, first- and
second-order excess decay, and quantitative spatial averages of the
invariant density.

Qualitative homogenization for linear uniformly elliptic equations in
stationary ergodic media was established by Papanicolaou and
Varadhan~\cite{PV82} and Yurinski\u\i~\cite{Yur82}. The associated diffusion
is a martingale, and an invariant measure for the environment seen from
the particle identifies its limiting covariance. For fully nonlinear
equations, where this linear invariant-measure construction is
unavailable, Caffarelli, Souganidis and Wang~\cite{CSW05} introduced an
approach based on an obstacle problem and a subadditive quantity.

The quantitative question requires additional control of the dependence
of the environment. Caffarelli and Souganidis~\cite{CS10} obtained
subalgebraic error estimates under quantitative mixing assumptions.
Armstrong and Smart~\cite{ASquant} subsequently obtained algebraic rates
under finite range of dependence, using a monotone quantity defined by
convex envelopes. Both results assume deterministic uniform ellipticity.

Armstrong and Smart~\cite{ASdeg} removed the deterministic lower
ellipticity bound in the qualitative theory. Their theorem applies to
fully nonlinear equations under the critical moment
\eqref{intro:critical-moment} and supplies a uniformly elliptic effective
operator. Its regularity estimates control microscopic degeneracy through
spatial averages. The question addressed here is whether finite range of
dependence yields an algebraic rate at this same moment threshold, and
whether it yields quadratic correctors, large-scale regularity, and
quantitative estimates for the invariant density.

The principal estimate is a finite-cell error bound, obtained by separating
the directional diffusion from its physical clock. We first state the
results, then describe the estimates and their order of proof.

\subsection{Assumptions and notation}

Let $d\ge2$, let $\mathbb S^d$ be the real symmetric $d\times d$ matrices,
and use the Frobenius norm and pairing $M:N=\operatorname{tr}(M^{\mathsf T}N)$.
We work on the canonical space of continuous coefficient fields, equipped
with the locally uniform topology and its Borel sigma-algebra. Translation
is denoted by $\tau_z$, so that $A(x,\tau_z\omega)=A(x+z,\omega)$.
The probability measure is denoted by $\mathbb P$ and expectation by
$\mathbb E$. We impose the following assumptions.

\begin{assumption}\label{ass:model}
\begin{enumerate}[label=\textnormal{(H\arabic*)},leftmargin=3.7em]
\item Almost surely, for every $x\in\mathbb R^d$,
$A(x)=A(x)^{\mathsf T}$ and $0<A(x)\le I$.
\item The translations $\tau_z$ preserve $\mathbb P$.
\item There is $\ell<\infty$ such that $\sigma(A|_U)$ and $\sigma(A|_V)$
are independent whenever $U,V$ are deterministic Borel sets with
$\operatorname{dist}(U,V)>\ell$.
\item For $F(M,x,\omega)=-A(x,\omega):M$, the family
$\{F(\cdot,\cdot,\omega)\}_\omega$ is uniformly equicontinuous on
$\{|M|\le K\}\times\mathbb R^d$ for every $K<\infty$. There are a
deterministic continuous modulus $\rho$ and $\sigma>1/2$ such that
\[
 |F(M,x,\omega)-F(M,y,\omega)|
 \le\rho\big((1+|M|)|x-y|^\sigma\big).
\]
\item Set $\lambda(x)=\lambda_{\min}A(x)$ and
$\lambda_*(z)=\inf_{B_1(z)}\lambda$. Then
$\mathbb E[\lambda_*(0)^{-d}]<\infty$.
\end{enumerate}
\end{assumption}

All constants and exponents are deterministic and may depend on the fixed
coefficient law, as well as on parameters displayed in the statement. We
write $B_R(z)$ for the open ball, $B_R=B_R(0)$, $\mathcal P_1$ for the affine
functions, and $\|\cdot\|_V$ for the uniform norm on $V$. A probability
estimate at $z$ is understood for each fixed deterministic center. The
exceptional event may depend on $z$; the simultaneous quantifiers over
solutions and radii are specified in each theorem.

Solutions are understood in the viscosity sense; we use the comparison
and stability framework of Crandall, Ishii and Lions~\cite{CIL92}.
Lemma~\ref{lem:local-regularity}
shows that H4 gives a deterministic H\"older modulus for $A$, and hence
classical local regularity in each realization. H3 implies ergodicity.
The qualitative theorem of Armstrong and Smart~\cite[Theorem~1]{ASdeg}
therefore supplies a deterministic matrix $0<\bar A\le I$, whose
invariant-measure representation is proved in Section~\ref{sec:density}.

\subsection{Main results}

For $M\in\mathbb S^d$, define the zero-boundary cell solution by
\begin{equation}\label{intro:cell-equation}
 -A:D^2w_{M,R}=(A-\bar A):M\quad\text{in }B_R,
 \qquad w_{M,R}=0\quad\text{on }\partial B_R,
\end{equation}
and set $X_R=R^{-2}\sup_{|M|\le1}\|w_{M,R}\|_{B_R}$.

\begin{theorem}[Algebraic homogenization errors]
\label{thm:intro:homogenization}\label{thm:intro:dirichlet}
Under H1--H5 there exist $\delta_2\in(0,1)$, $\gamma_2>0$ and
$C,R_0<\infty$ such that
\begin{equation}\label{intro:cell-rate}
 \mathbb P(X_R>CR^{-\delta_2})\le CR^{-\gamma_2}
 \qquad(R\ge R_0).
\end{equation}
Fix $\eta\in(0,1)$, a bounded $C^{2,\eta}$ domain $U$, and deterministic
$f\in C^\eta(\overline U)$, $g\in C^{2,\eta}(\overline U)$. Let
\[
 \begin{cases}-A(x/\varepsilon):D^2u_\varepsilon=f&\text{in }U,\\
 u_\varepsilon=g&\text{on }\partial U,\end{cases}
 \qquad
 \begin{cases}-\bar A:D^2u_0=f&\text{in }U,\\
 u_0=g&\text{on }\partial U.\end{cases}
\]
There exist $\beta_4,\gamma_4>0$ and $C<\infty$ such that, for all
sufficiently small $\varepsilon$,
\begin{equation}\label{intro:dirichlet-rate}
 \mathbb P\big(\|u_\varepsilon-u_0\|_U>
 C N(f,g)\varepsilon^{\beta_4}\big)\le C\varepsilon^{\gamma_4},
 \quad N(f,g)=\|f\|_{C^\eta}+\|g\|_{C^{2,\eta}}.
\end{equation}
The latter estimate is for each fixed pair of deterministic data.
\end{theorem}

The finite-cell estimate is proved in Theorem~\ref{thm:cell}. The passage
to general data, including the source and boundary error, is proved in
Theorem~\ref{down:dirichlet-main}, where explicit choices of
$\beta_4,\gamma_4$ in terms of $\delta_2,\gamma_2,d,\eta$ are given.

\begin{theorem}[Whole-space quadratic correctors]
\label{thm:intro:correctors}
There is a measurable family $\phi_M$, linear in $M\in\mathbb S^d$,
such that, almost surely and simultaneously for all $M$,
\begin{equation}\label{intro:corrector-equation}
 -A:D^2\phi_M=(A-\bar A):M\quad\text{in }\mathbb R^d.
\end{equation}
For each fixed $z$, almost surely and simultaneously for all $M$,
\[
 \phi_M(\cdot+z,\omega)-\phi_M(\cdot,\tau_z\omega)\in\mathcal P_1.
\]
Set
\[
 Y_R(z)=R^{-2}\sup_{|M|\le1}\inf_{l\in\mathcal P_1}
             \|\phi_M-l\|_{B_R(z)},\qquad
 \gamma_1=\frac d{4d-2},\quad\gamma_\phi=\min\{\gamma_1,\gamma_2\}.
\]
For every fixed $z$ and all sufficiently large $R$,
\begin{equation}\label{intro:corrector-rate}
 \mathbb P(Y_R(z)>CR^{-\delta_2})\le CR^{-\gamma_\phi}.
\end{equation}
\end{theorem}

The corrector is an equivalence class modulo affine functions.
Section~\ref{down:sec-corrector} constructs a representative by a fixed
affine projection and a convergent sum of cell increments.

Define
\[
 \mathcal H_2^A=\left\{a+b\cdot x+\tfrac12x\cdot Mx+\phi_M(x):
       a\in\mathbb R,\ b\in\mathbb R^d,\ M\in\mathbb S^d,
       \ \bar A:M=0\right\}.
\]
Every member of this space is exactly $A$-harmonic. For a function $u$, let
\begin{equation}\label{intro:excesses}
 \mathcal E(u;r,z)=r^{-1}\inf_{l\in\mathcal P_1}\|u-l\|_{B_r(z)},
 \qquad E_2(u;r,z)=r^{-2}\inf_{H\in\mathcal H_2^A}\|u-H\|_{B_r(z)}.
\end{equation}

\begin{theorem}[Large-scale regularity]
\label{thm:intro:regularity}
For each prescribed $\alpha\in(0,1)$ there are translation-covariant
random scales $r_{1,\alpha},r_{2,\alpha}\ge1$ satisfying, at each fixed $z$,
\[
 \mathbb P(r_{1,\alpha}(z)>t)\le C_\alpha t^{-\gamma_1},\qquad
 \mathbb P(r_{2,\alpha}(z)>t)\le C_\alpha t^{-\gamma_\phi}
 \quad(t\ge1).
\]
For each fixed deterministic $z$, almost surely, simultaneously for all
$A$-harmonic $u$ in $B_R(z)$ and all the indicated radii,
\begin{align}
 \mathcal E(u;r,z)&\le C_\alpha(r/R)^\alpha\mathcal E(u;R,z),
       &&r_{1,\alpha}(z)\le r\le R/2,\label{intro:c1}\\
 E_2(u;r,z)&\le C_\alpha(r/R)^\alpha E_2(u;R,z),
       &&r_{2,\alpha}(z)\le r\le R/2.\label{intro:c2}
\end{align}
\end{theorem}

The affine estimate is proved before the cell theorem and is used in the
corrector construction. The second-order estimate follows after that
construction; see Theorems~\ref{thm:reg:q1} and
\ref{thm:second-regularity}.

\begin{corollary}[Liouville classification]\label{intro:liouville}
Almost surely, every entire $A$-harmonic function with
$\|u\|_{B_R}=O(R^{2-\varepsilon})$ for some $\varepsilon>0$ is affine.
Every entire $A$-harmonic function with
$\|u\|_{B_R}=O(R^{3-\varepsilon})$ for some $\varepsilon>0$ belongs to
$\mathcal H_2^A$. Moreover,
$\dim\mathcal H_2^A=d(d+3)/2$.
\end{corollary}

The first conclusion follows by sending $R$ to infinity in
\eqref{intro:c1} with $\alpha>1-\varepsilon$. The second, including the
dimension count and the identification of local approximants without
unique continuation, is proved in Corollary~\ref{cor:second-liouville}.
The strict power gap in the growth assumption is used in this passage.

\begin{theorem}[Invariant density and quantitative averages]
\label{thm:intro:density}
There is a positive $m\in L^1_{\mathrm{loc}}$, covariant almost surely
under each fixed translation, with $\mathbb E[m(0)]=1$ such that
\begin{equation}\label{intro:density-identification}
 \int_{\mathbb R^d}mA:D^2\zeta=0\quad(\zeta\in C_c^\infty),
 \qquad\bar A=\mathbb E[m(0)A(0)].
\end{equation}
For fixed $\psi\in C_c^\infty(\mathbb R^d)$ let
$\psi_R(x)=R^{-d}\psi(x/R)$ and
\[
 Z_R(\psi)=\left|\int\psi_R(m-1)\right|
        +\left|\int\psi_R(mA-\bar A)\right|.
\]
Then, for all sufficiently large $R$,
\begin{equation}\label{intro:density-rate}
 \mathbb P(Z_R(\psi)>C_\psi R^{-\beta_3})
 \le C_\psi R^{-\gamma_3},\qquad
 \beta_3=\delta_2/4,\quad\gamma_3=\min\{\gamma_2,\delta_2/4\}.
\end{equation}
Also $\mathbb E Z_R(\psi)\le
C_\psi R^{-\min\{\beta_3,\gamma_3\}}$.
\end{theorem}

The qualitative statement, including the identification of $\bar A$, is
proved in Section~\ref{sec:density}. The quantitative statement is proved
in Section~\ref{down:sec-density-quantitative}. We obtain the average
estimate from cell errors through an adjoint energy identity and a
constant-coefficient heat semigroup. This argument uses only the first
moment of the nonlocal field $m$.

Section~\ref{sec:applications} gives two consequences. The Dirichlet
error converges in the gradient norm weighted by $m_\varepsilon A_\varepsilon$,
with error exponent $\beta_4/2$. A finite-domain approximation of $\bar A$,
defined without knowledge of $\bar A$, inherits the cell rate through the
deterministic bound $|\widehat A_R-\bar A|\le2dX_R$.

\subsection{The main ideas}

Write $a=\operatorname{tr}A$ and $B=A/a$. The generator $B:D^2$ has unit
trace, so its mean exit time from a ball is a deterministic quadratic
function. The original diffusion is recovered by the time change with
speed $a^{-1}$. This separates two effects of degeneracy which require
different estimates.

The first advantage of normalization is the bound
\[
 (\det B)^{-1}\le C_d\lambda_{\min}(A)^{-(d-1)}.
\]
Consequently, $(\det B)^{-1}$ has a moment of order $d/(d-1)>1$.
This is the gain used in the convex-envelope contraction for bounded
directional sources. A Harnack estimate is needed to make the contraction
quantitative. We observe the diffusion after it exits the sparse
connected regions where ellipticity is small. These observations are
made at stopping times, so the sampled positions retain the martingale
property. Conditional covariance, jump, and localization
estimates then give a coarse ABP inequality and a multiscale coupling.
The same construction gives an approximation event common to all harmonic
functions, from which affine excess decay follows.

For the physical clock the critical moment supplies only an integrable
endpoint weight. Integrability alone gives no power rate for its truncated
mean. The additional estimate is a reverse H\"older inequality for the
coarse Green function. Its self-improvement yields an exponent $p<d$;
this strict inequality makes the clock-truncation error algebraically
small using exactly \eqref{intro:critical-moment}. A comparison on an
event of positive probability then controls the deterministic truncated
means. This completes the passage from the normalized equation to the
finite-cell estimate for $A$.

The centering constants in these arguments are occupation means. We
construct the invariant probability and identify its covariance with
$\bar A$ before beginning the quantitative estimates.

The remaining arguments use the cell theorem as a common input. To glue
finite-volume correctors modulo affine functions, we choose the affine
regularity exponent so that $\alpha+\delta_2>1$; this is exactly the
condition making the normalized increment series converge. Corrected
quadratic approximation gives second-order excess decay and the Liouville
classification. The density averages are obtained by adjoint duality and
heat-semigroup smoothing. Finally, local cell comparisons and a torsion
barrier handle the source and boundary in the Dirichlet problem.

\subsection{Relation to earlier work and organization}

The directional contraction in Section~\ref{sec:renormalization} follows
the convex-envelope framework of Armstrong and Smart~\cite{ASquant}.
We use the corrected version,
\href{https://arxiv.org/abs/1306.5340v5}{arXiv:1306.5340v5}.
We adapt their method to H1--H5 by controlling the determinant weight,
the degenerate region, and the physical clock. The weighted estimates
required in this setting are proved in the corresponding lemmas.

The discrete probabilistic counterparts are balanced random walks.
Guo and Zeitouni~\cite[Theorem~1]{GZ12} proved a quenched invariance
principle under an inverse moment of order $p>d$ for the geometric mean
of the coordinate transition probabilities in an ergodic environment.
Berger and Deuschel~\cite{BD14} obtained an invariance principle for
non-elliptic iid balanced environments under a genuine dimensionality
condition. Berger, Cohen, Deuschel and Guo~\cite{BCDG} developed the
associated elliptic Harnack theory. Their multiscale coupling is a
precursor of Section~\ref{sec:regularity}; our coarse kernel instead
records the continuous diffusion at genuine stopping times, and its
covariance, jump bounds and dependence regions are established separately.
Quantitative results under finite-range dependence are already available
for balanced finite-difference equations in Guo, Peterson and
Tran~\cite[Theorems~1.2 and~1.5]{GPT22}, under uniform ellipticity of the normalized lattice
coefficients. Guo and Tran~\cite{GuoTran25} obtain optimal or nearly
optimal rates for iid balanced lattice environments with uniformly
elliptic normalized coefficients. Their quantitative theorems do not
cover the rotating eigendirections and degenerate normalized diffusion
matrices permitted by H1--H5.
Guo, Sprekeler and Tran~\cite{GST25} obtain sharper Dirichlet error
rates for balanced difference operators in iid environments in
$d\ge3$, under uniform ellipticity of the trace-normalized coefficients
and suitable smooth, well-prepared boundary data. Their coefficients
are diagonal lattice fields; H1--H5 allow continuous fields with varying
eigendirections and degenerate trace-normalized ellipticity.

The passage from homogenization to improved regularity has an important
periodic precedent in Avellaneda and Lin~\cite{AvLin89}, including
nondivergence-form equations. In the stochastic divergence-form setting,
Armstrong and Smart~\cite{ASconvex16} established quantitative
homogenization and large-scale regularity for uniformly convex integral
functionals under finite-range dependence; a systematic account is given
by Armstrong, Kuusi and Mourrat~\cite{AKM19}. For uniformly elliptic
nondivergence-form equations, Armstrong and Lin~\cite[Sections~2--3]{AL}
use an initial algebraic rate to prove large-scale $C^{1,1}$ estimates.
Their optimal corrector estimates use an additional product assumption.
Our affine estimate is proved before the finite-cell rate and is then used
to construct quadratic correctors. The subsequent corrected-polynomial
iteration is related to Fischer and Otto~\cite{FO}; their divergence-form
theory uses correctors and flux potentials, whereas here affine functions
are already exactly harmonic and the quadratic correction satisfies
\eqref{intro:corrector-equation}. Section~\ref{down:sec-corrector}
proves the required convergence modulo affine functions explicitly.

For invariant densities, the local adjoint theory includes
Bauman~\cite{Bauman84}. Armstrong, Fehrman and Lin~\cite[Theorem~1.4]{AFL}
prove quantitative estimates for averages of $m$ and $mA$ under
deterministic uniform ellipticity and finite-range dependence, using
parabolic Green-function estimates. Section~\ref{down:sec-density-quantitative}
derives smooth-test averages from the critical-moment cell bound and the
first moment of $m$. We use~\cite[Proposition~A.4]{AFL} for local
regularity: its ellipticity and H\"older assumptions hold on each
fixed compact set in each realization, by
Lemma~\ref{lem:local-regularity}. No uniform probabilistic bound for its
local regularity constant is inferred from this application.

A second local input is Bowman~\cite[Proposition~1.2]{Bowman26}, which gives
an oscillation estimate when the inverse ellipticity belongs to $L^p$
with $p>d-1$. We state the form used below and check its local hypotheses.
The large-scale Harnack inequality and its probability estimate require
the separate multiscale argument of Section~\ref{sec:regularity}.

Section~\ref{sec:density} constructs the invariant density and identifies
$\bar A$. Section~\ref{sec:regularity} establishes the coarse process,
affine regularity and Harnack estimates. Sections~\ref{sec:renormalization}
and~\ref{sec:green} treat the directional equation and the physical clock,
respectively, and culminate in the cell theorem. The remaining sections
construct the correctors, prove second-order regularity, and establish
the density and Dirichlet estimates and their applications.

\section{Invariant density and the effective matrix}\label{sec:density}

We construct the invariant density and identify its weighted coefficient
average with the homogenized matrix. The argument uses qualitative
homogenization and the critical moment H5. We first record the local
regularity and ergodicity consequences used throughout the paper.

\begin{lemma}[Local regularity and ergodicity]
\label{lem:local-regularity}\label{down:local-regularity}
For every $0<\theta<\min\{1,\sigma\}$, $A$ has a deterministic global
$C^{0,\theta}$ bound. In each realization, its minimum eigenvalue has a
positive lower bound on every compact set. Solutions with locally
$C^{0,\theta}$ right-hand side are locally $C^{2,\theta}$. The canonical
translation action is ergodic.
\end{lemma}
\begin{proof}
For $h=|x-y|\le1$ and $A(x)\ne A(y)$, apply H4 with
$M=h^{-\sigma}(A(x)-A(y))/|A(x)-A(y)|$. This gives
\[
 |A(x)-A(y)|\le h^\sigma\sup_{0\le t\le2}\rho(t).
\]
For $h>1$, use $|A|\le\sqrt d$. Continuity and strict positive
definiteness give the compact-set ellipticity bound. The local Schauder
theorem~\cite[Chapter~6]{GT} and uniqueness of viscosity solutions then give
the asserted regularity. These local Schauder constants may depend on the
realization; no moment bound for them is being asserted.

If $G,H$ are bounded functions of the field on bounded sets, H3 gives
$\mathbb E[G(H\circ\tau_z)]=\mathbb EG\,\mathbb EH$ for large $|z|$.
For arbitrary bounded $G,H$, approximate by conditional expectations on
restrictions to increasing balls. Their $L^2$ convergence and invariance of
$\mathbb P$ yield the same equality in the limit $|z|\to\infty$.
Applied to the indicator of a translation-invariant event, this implies
that its probability is zero or one.
\end{proof}

\begin{theorem}[Invariant density]\label{down:density-main}
Under H1--H5, there is a measurable function $m_0$ on the canonical
coefficient space for which $m(x,\omega)=m_0(\tau_x\omega)$ has a locally
H{\"o}lder continuous, everywhere positive version satisfying
\begin{equation}\label{down:density-identities}
 \begin{gathered}
 \mathbb E[m(0)]=1,\qquad \mathbb E[m(0)A(0)]=\bar A,\\
 \int_{\mathbb R^d}mA:D^2\zeta=0
 \qquad(\zeta\in C_c^\infty(\mathbb R^d)).
 \end{gathered}
\end{equation}
The adjoint equation holds simultaneously for all test functions almost
surely. For each fixed $z$, this version satisfies
$m(x,\tau_z\omega)=m(x+z,\omega)$ almost surely for every $x$. Moreover,
with $q=d/(d-1)$ and $\lambda(x)=\lambda_{\min}(A(x))$,
\begin{equation}\label{down:density-weighted-moment}
 \|\lambda(0)m(0)\|_{L^q(\mathbb P)}\le C_d.
\end{equation}
The measure $Q=m_0\mathbb P$ is equivalent to $\mathbb P$ and is an
invariant ergodic probability measure for the environment seen from the
particle. It is the unique invariant probability measure absolutely
continuous with respect to $\mathbb P$.
\end{theorem}

The proof has three parts. Stopped occupation measures and the ABP
estimate give uniform integrability at the critical moment. The relevant
estimate is written with the local ellipticity weight, so its dual bound
controls $\lambda m$; uniform integrability of $m$ itself uses H5.
The limit of the occupation measures yields the adjoint density.
Finally, the martingale limit identifies a candidate covariance, and
qualitative homogenization~\cite[Theorem~1]{ASdeg} identifies it with
$\bar A$. We compare both homogeneous and unit-source Dirichlet problems
to fix the physical time normalization.

\subsection{Occupation measures}

Lemma~\ref{lem:local-regularity} permits the construction of the diffusion
with generator $G_\omega=A(\cdot,\omega):D^2$ by the local martingale
problem~\cite{SV}. Indeed, on each ball extend $A$ to bounded uniformly
elliptic H{\"o}lder coefficients and stop the resulting diffusion at the
boundary. Local uniqueness makes these stopped laws consistent. If $T_N$
is the exit time from $B_N$, then
\[
 \mathbf P_x^\omega(T_N\le t)
 \le\frac{\mathbf E_x^\omega|X_{t\wedge T_N}-x|^2}{(N-|x|)^2}
 \le\frac{2dt}{(N-|x|)^2},
\]
so the patched diffusion is nonexplosive. Its killed transition density
$p_{B_N}^\omega(t,x,y)$ is continuous and strictly positive in the ball
for $t>0$. Increasing the killing domains gives the whole-space density
$p^\omega$. Uniqueness gives translation covariance and the strong
Markov property. We denote quenched expectation by $\mathbf E_x^\omega$;
$\mathbb E$ continues to denote expectation over environments.

\begin{lemma}[Stopped occupation estimate]\label{down:density-occupation}
For every fixed environment, $R>0$, and nonnegative Borel function $h$,
\begin{equation}\label{down:density-occupation-bound}
 \sup_{x\in B_R}\mathbf E_x^\omega\int_0^{T_R}h(X_s)\,ds
 \le C_dR\left(\int_{B_R}h^d\lambda^{-d}\right)^{1/d},
 \qquad T_R=\inf\{t\ge0:X_t\notin B_R\},
\end{equation}
allowing an infinite right-hand side. With
$a_R=\mathbb E\mathbf E_0^\omega T_R$,
\begin{equation}\label{down:density-exit-annealed}
 \frac{R^2}{2d}\le a_R
 \le C_dR^2\bigl(\mathbb E[\lambda(0)^{-d}]\bigr)^{1/d}<\infty.
\end{equation}
\end{lemma}

\begin{proof}
For smooth $h\ge0$, solve $-A:D^2v=h$ in $B_R$, with zero boundary
values. Local Schauder theory, comparison, and ABP
\cite[Theorem~9.1]{GT} give
$0\le v\le C_dR\|h/\lambda\|_{L^d(B_R)}$. Writing
$\lambda_R=\min_{\overline B_R}\lambda>0$, the quadratic barrier gives
$\sup_x\mathbf E_x^\omega T_R\le R^2/(2d\lambda_R)$. It{\^o}'s
formula, first stopped in smaller concentric balls, then passed to $T_R$
using this bound and boundary continuity, gives
\[
 v(x)=\mathbf E_x^\omega\int_0^{T_R}h(X_s)\,ds.
\]
For Borel data, the occupation measure
$\nu_x(E)=\mathbf E_x^\omega\int_0^{T_R}\mathbf1_E(X_s)ds$ has a
Lebesgue density by the killed kernel. The preceding inequality bounds
its integral functional on smooth functions in
$L^d(B_R,\lambda^{-d}dy)$. Density and duality extend the functional to
that space; its representing measure agrees with $\nu_x$ on smooth
compactly supported tests and hence on Borel sets. Truncation and
monotone convergence prove~\eqref{down:density-occupation-bound} for
all nonnegative $h$, uniformly in $x$.

Optional stopping of the martingale
$|X_t|^2-2\int_0^t\operatorname{tr}A(X_s)ds$, started at the origin,
gives
\[
 R^2=2\mathbf E_0^\omega\int_0^{T_R}\operatorname{tr}A(X_s)ds
 \le2d\mathbf E_0^\omega T_R.
\]
The preceding integrability bound justifies the stopping limit. For the
upper bound, apply~\eqref{down:density-occupation-bound} with $h=1$,
then Jensen's inequality and stationarity:
\[
 a_R\le C_dR\left(\mathbb E\int_{B_R}\lambda^{-d}\right)^{1/d}
 \le C_dR^2\bigl(\mathbb E[\lambda(0)^{-d}]\bigr)^{1/d}.
\]
H5 applies since $\lambda(0)\ge\lambda_*(0)$.
\end{proof}

Let $\eta_t=\tau_{X_t}\omega$ and
$P_tf(\omega)=\mathbf E_0^\omega f(\eta_t)$. Define the probability
measure
\begin{equation}\label{down:density-stopped-probability}
 Q_R(f)=\frac1{a_R}\mathbb E\mathbf E_0^\omega
       \int_0^{T_R}f(\eta_s)\,ds.
\end{equation}
For bounded Borel $f\ge0$, apply the quenched occupation estimate to
$h(x)=f(\tau_x\omega)$. Jensen's inequality, stationarity, and the
lower bound in~\eqref{down:density-exit-annealed} yield
\begin{equation}\label{down:density-functional-bound}
 Q_R(f)\le C_d
 \bigl(\mathbb E[|f|^d\lambda(0)^{-d}]\bigr)^{1/d}.
\end{equation}
Taking null-event indicators shows $Q_R\ll\mathbb P$. Write
$m_R=dQ_R/d\mathbb P$. With $f=\lambda(0)g$, duality gives
$\|\lambda(0)m_R\|_{L^q}\le C_d$. More importantly,
\begin{equation}\label{down:density-uniform-integrability}
 \sup_R\int_E m_R\,d\mathbb P
 \le C_d\bigl(\mathbb E[\lambda(0)^{-d}\mathbf1_E]\bigr)^{1/d}
 \qquad(E\in\mathcal F).
\end{equation}
The right-hand side tends uniformly to zero with $\mathbb P(E)$.
Since $\mathbb E m_R=1$, Dunford--Pettis and weak sequential compactness
give a subsequence $m_{R_n}\rightharpoonup m_0$ in $L^1$, where
$R_n\to\infty$. Testing with nonnegative bounded functions and with $1$
gives $m_0\ge0$ and $\mathbb E m_0=1$. Set $Q=m_0\mathbb P$.
The bound~\eqref{down:density-functional-bound} passes to $Q$ by weak
convergence, and yields~\eqref{down:density-weighted-moment}.

\begin{lemma}[Invariance]\label{down:density-invariance}
For every $t\ge0$ and bounded Borel $f$, $Q(P_tf)=Q(f)$.
\end{lemma}
\begin{proof}
Since $\{s<T_R\}$ is measurable at time $s$, the Markov property gives
\[
 \mathbb E\mathbf E_0^\omega\int_0^{T_R}P_tf(\eta_s)ds
 =\mathbb E\mathbf E_0^\omega\int_0^{T_R}f(\eta_{s+t})ds.
\]
Fubini is justified by $\|f\|_\infty T_R\in L^1$. Shifting the
integration intervals pathwise, also when $T_R<t$, gives
\begin{equation}\label{down:density-time-shift}
 Q_R(P_tf)-Q_R(f)=\frac1{a_R}\mathbb E\mathbf E_0^\omega
 \left[\int_{T_R}^{T_R+t}f(\eta_s)ds-\int_0^tf(\eta_s)ds\right].
\end{equation}
Its absolute value is at most $2t\|f\|_\infty/a_R\to0$.
Weak $L^1$ convergence applies to both bounded tests $f$ and $P_tf$.
\end{proof}

\subsection{Positivity and the adjoint equation}

For $\mathbb P(E)>0$, spatial ergodicity implies that
$\{x:\tau_x\omega\in E\}$ has positive Lebesgue measure in some ball
containing the origin, almost surely. Positivity of the killed kernel
therefore gives, for each fixed $t>0$,
\begin{equation}\label{down:density-positivity-improving}
 P_t\mathbf1_E>0\qquad\mathbb P\text{-almost surely}.
\end{equation}
If $E=\{m_0=0\}$ had positive probability, invariance would give
$0=Q(E)=Q(P_t\mathbf1_E)>0$, since $Q\ll\mathbb P$. Thus
$Q\sim\mathbb P$. If $0<Q(E)<1$, applying the same argument to $E$
and $E^c$ gives $0<P_t\mathbf1_E<1$ almost surely. Such an $E$ cannot
be invariant, so $Q$ is ergodic.

The argument applies to every invariant $Q'\ll\mathbb P$. Hence $Q'$
is also equivalent to $\mathbb P$ and ergodic. The stationary Markov
path laws with initial distributions $Q$ and $Q'$ are equivalent,
since their transition kernels agree. For each bounded $f$, their
respective almost-sure time-average limits $Q(f)$ and $Q'(f)$ must
therefore coincide. Taking indicators proves uniqueness.

Set $\widetilde m(x,\omega)=m_0(\tau_x\omega)$. For compact $K$,
$\mathbb E\int_K\widetilde m=|K|$, so $\widetilde m\in L^1_{\rm loc}$
almost surely. For bounded Borel $f$ and $\zeta\in C_c^\infty$, define
\[
 F_\zeta(\omega)=\int\zeta(x)f(\tau_x\omega)\,dx.
\]
The orbit function
$F_\zeta(\tau_y\omega)=\int\zeta(x-y)f(\tau_x\omega)dx$ is smooth,
with derivatives bounded by $\|D^k\zeta\|_1\|f\|_\infty$.
It{\^o}'s formula and invariance imply
$Q(G_{\rm env}F_\zeta)=0$, where
$G_{\rm env}F_\zeta=A(0):\int D^2\zeta(x)f(\tau_x\omega)dx$.
Changing variables by the measure-preserving translations gives
\begin{align}
 0&=\int\mathbb E[m_0A(0)f(\tau_x\omega)]:D^2\zeta(x)\,dx\notag\\
 &=\mathbb E\left[f(\omega)\int
 \widetilde m(y,\omega)A(y,\omega):D^2\zeta(-y)\,dy\right].
 \label{down:density-quenched-duality}
\end{align}
The inner integral belongs to $L^1(\mathbb P)$ by $A\le I$ and
$\mathbb E\widetilde m(y)=1$. Since $f$ is arbitrary, duality makes it
zero almost surely. A countable $C^2$-dense family of tests on each
compact ball, using $\widetilde m A\in L^1_{\rm loc}$, gives the
adjoint equation for all tests on one probability-one event.

The local theory of nonnegative adjoint solutions includes the integral
estimates of Bauman~\cite[Section~3]{Bauman84}, for uniformly elliptic
operators with continuous principal coefficients. The precise regularity
result used here is Armstrong, Fehrman and Lin~\cite[Proposition~A.4]{AFL}:
an $L^1$ weak adjoint solution has a locally H{\"o}lder representative
when the coefficients are H{\"o}lder and uniformly elliptic on the
local domain. Lemma~\ref{lem:local-regularity} verifies these hypotheses
on every fixed compact set in each realization. The constants may
depend on its minimum eigenvalue; this local estimate gives no moment
bound for the random H{\"o}lder norm. Define the
representative $m$ by limits of ball averages. This preserves
measurability and covariance. Spatial almost
everywhere equality and stationarity give
\[
 \mathbb P(m(0)\ne m_0)=\frac1{|B_1|}\mathbb E\int_{B_1}
 \mathbf1_{\{m(x)\ne\widetilde m(x)\}}\,dx=0.
\]
Thus normalization and~\eqref{down:density-weighted-moment} are preserved.

For strict positivity at every point, expand $Q(P_tf)=Q(f)$, translate
the environment in the integral, and use Tonelli and bounded-test
duality to obtain
\[
 m_0(\omega)=\int\widetilde m(y,\omega)p^\omega(t,y,0)\,dy
 \quad\text{almost surely}.
\]
Translation and Fubini then give, for almost every $x\in B_N$,
\[
 m(x)\ge\int_{B_N}m(y)p_{B_N}^\omega(t,y,x)\,dy.
\]
The right-hand side is continuous on compact subsets of $B_N$, by the
local bound for the killed kernel and $m\in L^1(B_N)$. It is strictly
positive since $m>0$ almost everywhere and the kernel is positive.
Continuity of $m$ extends the inequality to every $x\in B_N$; countably
many balls exhaust the space. The ball-average construction and
continuity also give covariance for all $x$, almost surely for each
fixed translation.

\subsection{Identification in physical time}

Write
\begin{equation}\label{down:density-matrix-B}
 B=\mathbb E_Q[A(0)]=\mathbb E[m_0A(0)].
\end{equation}
Then $0<B\le I$: finiteness follows from $\mathbb E m_0=1$, and strict
positivity from $m_0>0$ and $A>0$. Ergodicity of the stationary
environment process gives
$T^{-1}\int_0^T A(X_s)ds\to B$ almost surely. By Fubini and
$Q\sim\mathbb P$, this holds $\mathbf P_0^\omega$-almost surely for
$\mathbb P$-almost every environment. For $X_t^{(R)}=R^{-1}X_{R^2t}$,
\begin{equation}\label{down:density-bracket-limit}
 \langle X_i^{(R)},X_j^{(R)}\rangle_t
 =\frac2{R^2}\int_0^{R^2t}A_{ij}(X_s)ds\longrightarrow2tB_{ij}.
\end{equation}
This convergence is locally uniform in $t$: the centered integral
$H(u)=\int_0^u(A(X_s)-B)ds$ satisfies $H(u)/u\to0$, while its part on
$[0,K]$, divided by $R^2$, vanishes. Since these are continuous
square-integrable martingales with continuous brackets,
\cite[Theorem~2.1(ii)]{Whitt} yields
\begin{equation}\label{down:density-qip}
 X^{(R)}\Longrightarrow W^B
 \quad\text{in }C_{\rm loc}([0,\infty);\mathbb R^d),
 \qquad \mathbb E[W_t^B\otimes W_t^B]=2tB.
\end{equation}

To fix the time normalization, we also need convergence of mean exit
times. The occupation estimate and spatial ergodicity give, almost
surely for all sufficiently large $R$,
\begin{equation}\label{down:density-exit-uniform-bound}
 \sup_{x\in B_R}\mathbf E_x^\omega T_R
 \le C_dR^2\left(\fint_{B_R}\lambda^{-d}\right)^{1/d}
 \le C(\omega)R^2.
\end{equation}
The strong Markov property and Tonelli imply
\[
 \mathbf E_0^\omega T_R^2
 =2\mathbf E_0^\omega\int_0^{T_R}\mathbf E_{X_s}^\omega T_R\,ds
 \le2\sup_{x\in B_R}\mathbf E_x^\omega T_R\,
          \mathbf E_0^\omega T_R\le C(\omega)R^4.
\]
Thus $T_R/R^2$ is uniformly integrable in each such environment.
The unit-ball exit time and exit position are continuous functionals
at almost every nondegenerate Brownian path: after reaching the
boundary, the projection onto the outward normal is a one-dimensional
Brownian motion with positive variance and enters the exterior
immediately. Consequently~\eqref{down:density-qip} gives
\begin{equation}\label{down:density-exit-mean-limit}
 R^{-2}\mathbf E_0^\omega T_R\longrightarrow
 \mathbf E_0T_{B_1}^{W^B}=\frac1{2\operatorname{tr}B}.
\end{equation}
The last equality follows by applying $-B:D^2$ to
$(1-|x|^2)/(2\operatorname{tr}B)$.

The same rescaled mean exit time is the central value of
\[
 -A(x/\varepsilon):D^2u^\varepsilon=1\quad\text{in }B_1,
 \qquad u^\varepsilon=0\quad\text{on }\partial B_1,
 \qquad \varepsilon=R^{-1}.
\]
Armstrong--Smart's qualitative theorem and the linear and
constant-shift properties of the effective operator
\cite[Theorem~1 and Lemma~4.2]{ASdeg} give
$u^\varepsilon(0)\to(2\operatorname{tr}\bar A)^{-1}$. Hence
\begin{equation}\label{down:density-trace-identification}
 \operatorname{tr}B=\operatorname{tr}\bar A.
\end{equation}
Next prescribe $g_M(x)=\tfrac12x\cdot Mx$ on $\partial B_1$ and zero
source. Exit-position convergence identifies the limiting central value
with that of the $B$-harmonic function
\[
 v_B(x)=\frac12x\cdot Mx+
       \frac{B:M}{2\operatorname{tr}B}(1-|x|^2).
\]
Qualitative homogenization identifies it instead with
$(\bar A:M)/(2\operatorname{tr}\bar A)$. Applying this to a basis of
$\mathbb S^d$ and using~\eqref{down:density-trace-identification} proves
$B=\bar A$, completing Theorem~\ref{down:density-main}. The source
problem is essential here: harmonic measures alone determine the
matrix only up to a scalar factor.

\subsection{The invariant law under trace normalization}

The quantitative estimates use the diffusion with coefficient matrix
$B=A/a$, where $a(x,\omega)=\operatorname{tr}A(x,\omega)$. The following
lemma identifies its invariant law and the mean of its physical clock.

\begin{lemma}[Invariant law under trace normalization]
\label{down:density-normalized-law}
Let $Q=m_0\mathbb P$ be the probability in
Theorem~\ref{down:density-main}, and set
\[
 c=Q(a(0))=\operatorname{tr}\bar A,\qquad
 \pi(d\omega)=\frac{a(0,\omega)}c\,Q(d\omega).
\]
The martingale problem with spatial generator $B:D^2$ is well posed and
conservative. The measure $\pi$ is equivalent to $\mathbb P$ and is
invariant and ergodic for the associated environment process. For each
fixed bounded Borel function $f$, its stationary time average converges
almost surely to $\pi(f)$. Moreover,
\begin{equation}\label{down:density-normalized-means}
 \pi(B(0))=\frac{\bar A}{\operatorname{tr}\bar A},\qquad
 \pi(a(0)^{-1})=\frac1{\operatorname{tr}\bar A}.
\end{equation}
The spatial density
$m_B(x,\omega)=c^{-1}a(x,\omega)m(x,\omega)$ is positive, locally
integrable, covariant under each fixed translation, and satisfies
\begin{equation}\label{down:density-normalized-adjoint}
 \mathbb E[m_B(0)]=1,\qquad
 \int_{\mathbb R^d}m_BB:D^2\zeta=0
 \quad(\zeta\in C_c^\infty(\mathbb R^d)).
\end{equation}
\end{lemma}

\begin{proof}
On each compact set, $a$ has a positive minimum, so $B$ is H\"older
continuous and uniformly elliptic there. Also $0<B\le I$ and
$\operatorname{tr}B=1$. The local martingale construction used for $A$
therefore applies to $B$. If $\sigma_N$ is its exit time from $B_N$,
the corresponding stopped diffusion $Z$ satisfies
\[
 \mathbf P_x^\omega(\sigma_N\le t)
 \le\frac{\mathbf E_x^\omega|Z_{t\wedge\sigma_N}-x|^2}
          {(N-|x|)^2}
 =\frac{2\mathbf E_x^\omega(t\wedge\sigma_N)}{(N-|x|)^2}
 \le\frac{2t}{(N-|x|)^2}.
\]
Thus the local laws define a unique conservative strong Markov diffusion.
Uniqueness gives translation covariance, as for $A$. Its environment
semigroup is denoted by $\widetilde P_s$.

Start the physical diffusion $X$ at the origin with initial environment
law $Q$, and write $\eta_t=\tau_{X_t}\omega$. This process is stationary
and ergodic by Theorem~\ref{down:density-main}. Define
\[
 S(t)=\int_0^t a(0,\eta_r)\,dr,
 \qquad T(s)=\inf\{t\ge0:S(t)>s\}.
\]
The sample paths of $S$ are continuous and strictly increasing. Since
$0<a\le d$ and $c=Q(a(0))>0$, the ergodic theorem gives
$S(t)/t\to c$ almost surely. Hence $S(t)\to\infty$, and $T(s)$ is finite
for every finite $s$, with $S(T(s))=s$. Each $T(s)$ is a stopping time
for the usual right-continuous filtration, since
$\{T(s)<t\}=\{S(t)>s\}$. Moreover, $T(s)\ge s/d$, so
$T(s)\to\infty$ as $s\to\infty$.

The time-changed path $X_{T(s)}$ solves the martingale problem for $B$.
Indeed, stop in a compact ball and at a bounded physical time. Optional
sampling of the physical martingale for $g\in C_c^2(\mathbb R^d)$ and
the substitution $v=S(r)$ give the local martingale
\[
 g(X_{T(s)})-g(X_0)
 -\int_0^s B(X_{T(v)},\omega):D^2g(X_{T(v)})\,dv.
\]
The substitution is justified by the positive minimum of $a$ on the
stopped ball; in particular, $dr=dv/a(X_{T(v)},\omega)$. Exhausting
the balls and physical times proves the assertion. Uniqueness of the
normalized martingale problem identifies this path law with that of $Z$.
Equivalently, additivity of $S$ and the strong Markov property of $X$
at $T(s)$ identify the future path with the same change of clock started
at its current environment. Thus
$\zeta_s=\eta_{T(s)}$ has semigroup $\widetilde P_s$.
This construction holds for $Q$-almost every initial environment;
the local construction above defines the normalized kernel elsewhere.

Fix a bounded Borel function $f$. Apply the physical ergodic theorem
to $a(0)$ and $a(0)f$. Changing variables pathwise gives
\begin{equation}\label{down:density-timechange-averages}
 \frac1s\int_0^s f(\zeta_v)\,dv
 =\frac{\displaystyle\int_0^{T(s)}a(0,\eta_r)f(\eta_r)\,dr}
        {\displaystyle\int_0^{T(s)}a(0,\eta_r)\,dr}
 \longrightarrow\frac{Q(a(0)f)}c=\pi(f)
 \quad\text{almost surely}.
\end{equation}
This conclusion uses the initial law $Q$. To prove invariance of $\pi$,
let $\mu_L$ be the time average over $[0,L]$ of the laws of $\zeta$
with that initial law:
\[
 \mu_L(f)=\frac1L\mathbf E_Q\int_0^L f(\zeta_v)\,dv.
\]
Here $\mathbf E_Q$ averages over both the initial environment and the
normalized diffusion. Dominated convergence in
\eqref{down:density-timechange-averages} gives
$\mu_L(f)\to\pi(f)$ for every bounded Borel $f$. For fixed $t\ge0$,
the Markov property and Fubini's theorem imply
\begin{align*}
 \mu_L(\widetilde P_tf)-\mu_L(f)
 &=\frac1L\mathbf E_Q\left[
       \int_L^{L+t}f(\zeta_v)\,dv-\int_0^t f(\zeta_v)\,dv\right],\\
 \big|\mu_L(\widetilde P_tf)-\mu_L(f)\big|
 &\le\frac{2t\|f\|_\infty}{L}.
\end{align*}
The function $\widetilde P_tf$ is bounded and Borel. Passing to the
limit for both tests proves
$\pi(\widetilde P_tf)=\pi(f)$ for all $t\ge0$ and bounded Borel $f$.

Since $d\pi/dQ=a(0)/c$ is finite and strictly positive,
$\pi\sim Q\sim\mathbb P$. The normalized path laws with initial
distributions $\pi$ and $Q$ are also equivalent: they mix the same
conditional path kernels, and their Radon--Nikodym derivative is
$a(0,\zeta_0)/c$. Thus~\eqref{down:density-timechange-averages} holds
under the initial law $\pi$ for each fixed bounded $f$.

To verify ergodicity, let $E$ be invariant for $\widetilde P_t$, so that
$\widetilde P_t\mathbf1_E=\mathbf1_E$ almost surely under $\pi$ for
every $t\ge0$. For fixed $t$, the probabilities of a change of membership
in $E$ are
\[
 \int_E(1-\widetilde P_t\mathbf1_E)\,d\pi=0,
 \qquad \int_{E^c}\widetilde P_t\mathbf1_E\,d\pi=0.
\]
Tonelli's theorem gives
$\mathbf1_E(\zeta_t)=\mathbf1_E(\zeta_0)$ for almost every time on
almost every path. Applying~\eqref{down:density-timechange-averages}
to $\mathbf1_E$ yields $\mathbf1_E(\zeta_0)=\pi(E)$ almost surely.
Consequently $\pi(E)\in\{0,1\}$.

Finally, Theorem~\ref{down:density-main} and $aB=A$ give
\[
 c=Q(a(0))=\operatorname{tr}Q(A(0))=\operatorname{tr}\bar A,
 \qquad \pi(B(0))=c^{-1}Q(A(0))=c^{-1}\bar A.
\]
For the nonnegative function $a(0)^{-1}$, the change of measure gives
$\pi(a(0)^{-1})=c^{-1}Q(1)=c^{-1}$, proving its integrability.
The properties of $m_B$ follow from those of $m$ and the identity
$m_BB=c^{-1}mA$. This also proves
\eqref{down:density-normalized-adjoint}.
\end{proof}

\subsection{Spatial averages}

\begin{lemma}[Weighted ergodic limits]\label{down:density-ergodic}
On one event of probability one, for every $\psi\in C_c(\mathbb R^d)$,
with $\psi_R(x)=R^{-d}\psi(x/R)$,
\begin{equation}\label{down:density-ergodic-limits}
 \int\psi_R(m-1)\longrightarrow0,\qquad
 \int\psi_R(mA-\bar A)\longrightarrow0.
\end{equation}
Also $\fint_{B_R}m\to1$.
\end{lemma}
\begin{proof}
The fields $m$ and $mA$ are integrable factors of the ergodic translation
action. Intersect the spatial ergodic events for rational rectangles,
for each matrix component, and for balls centered at the origin.
For a continuous $\psi$ supported in $B_K$, approximate uniformly by
step functions $s_n$ on rational grids, supported in $B_{K+1}$. The
weighted error is bounded by
\[
 \left|\int R^{-d}(\psi-s_n)(x/R)m(x)dx\right|
 \le\|\psi-s_n\|_\infty R^{-d}\int_{B_{(K+1)R}}m.
\]
The last factor is eventually bounded. The same estimate controls $mA$
since $|A|\le\sqrt d$. Let $R\to\infty$ and then $n\to\infty$.
This deterministic approximation applies to every $\psi$ on the chosen
countable intersection, establishing the asserted quantifiers.
\end{proof}

\begin{remark}\label{down:density-moment-scope}
The bounds $\lambda m\in L^{d/(d-1)}(\mathbb P)$ and
$\lambda^{-1}\in L^d(\mathbb P)$ give exactly $m\in L^1$ by
H{\"o}lder's inequality. They supply neither a higher moment of $m$ nor
finite-range dependence of this nonlocal field.
Lemma~\ref{down:density-ergodic} is qualitative. The algebraic estimates
in Section~\ref{down:sec-density-quantitative} will follow from the
cell bound, the adjoint equation, and this first moment.
\end{remark}

\section{Coarse processes and large-scale regularity}
\label{sec:regularity}

Throughout this section, H1--H5 hold and the coefficient law is fixed.
Constants may depend on this law. Set
\[
 \lambda=\lambda_{\min}A,\qquad a=\operatorname{tr}A,\qquad B=A/a.
\]
Then \(\operatorname{tr}B=1\), and \(A\) and \(B\) have the same harmonic
functions. The effective matrix \(\bar A\) is identified in
Section~\ref{sec:density}. Environmental probability is denoted
by \(\mathbb P\); \(\mathbf P_x^\omega,\mathbf E_x^\omega\) refer to the diffusion
in a fixed environment. By Lemma~\ref{lem:local-regularity}, each realization
is uniformly elliptic on compact sets, and local Schauder estimates apply.

For the space \(\mathcal P_1\) of affine functions, define
\[
 \mathcal E(u;r,z)=r^{-1}\inf_{\ell\in\mathcal P_1}
       \|u-\ell\|_{L^\infty(B_r(z))}.
\]
Our first objective is the following estimate.

\begin{theorem}[Large-scale \(C^{1,\alpha}\) estimate and random scale]
\label{thm:reg:q1}
Under H1--H5, for every prescribed \(\alpha\in(0,1)\), there are
a deterministic \(C_\alpha<\infty\) and a translation-covariant
measurable random scale \(r_{*,\alpha}(z)\ge1\) such that
\begin{equation}
 \mathbb P(r_{*,\alpha}(z)>t)\le C_\alpha t^{-\gamma_1},
 \qquad
 \gamma_1=\frac{d}{4d-2},\quad t\ge1.
 \label{reg:eq:random-scale-tail}
\end{equation}
For each fixed deterministic center \(z\), almost surely,
simultaneously for all \(R\ge2r_{*,\alpha}(z)\), all functions
\(u\) that are \(A\)-harmonic in \(B_R(z)\), and all
\(r_{*,\alpha}(z)\le r\le R/2\),
\begin{equation}
 \mathcal E(u;r,z)
 \le C_\alpha(r/R)^\alpha\mathcal E(u;R,z).
 \label{reg:eq:q1}
\end{equation}
\end{theorem}

The proof separates geometry from elapsed time. A stopping-time kernel for
\(B:D^2\) has zero mean, nondegenerate covariance, and jump lengths with
exponential tails. A contact-set argument then converts qualitative homogenization of
quadratic data into a fixed-tolerance approximation on a common event.
Mesoscopic oscillation estimates allow the boundary data to depend on the
solution, and a Campanato iteration completes the proof.
The passage from homogenized approximation to affine excess decay follows
the large-scale regularity strategy of Armstrong--Lin~\cite[Section~3]{AL}.
Their approximation estimates use uniform ellipticity. Here the inputs
to the iteration are the common approximation event constructed below
and the mesoscopic modulus in Lemma~\ref{lem:reg:coarse-holder}; the
probability estimate for the resulting minimal scale is derived from
these inputs.

\subsection{A coarse kernel for the normalized diffusion}
\label{reg:subsec:coarse}

\begin{lemma}[Ball-exit identities for the normalized diffusion]
\label{lem:reg:normalized-diffusion}
Under H1--H5, for almost every environment, the continuous-path martingale problem
with generator \(B:D^2\) has a unique nonexplosive solution, forming a strong Markov
family \(X\). Its solution kernel can be chosen jointly measurable in the environment
and the starting point. If \(x\in B_L(c)\) and
\(\tau=\inf\{t:X_t\notin B_L(c)\}\), then
\begin{equation}
 \mathbf E_x^\omega\tau=\frac{L^2-|x-c|^2}{2},
 \qquad
 \sup_{x\in B_L(c)}\mathbf E_x^\omega\tau^m
 \le m!\left(\frac{L^2}{2}\right)^m,\quad m\in\mathbb N.
 \label{reg:eq:exit-moments}
\end{equation}
\end{lemma}

\begin{proof}
The martingale problem for continuous, locally uniformly positive definite
coefficients is well posed \cite{SV}. Localization and uniqueness give a strong
Markov family with
\[
 \langle X^i,X^j\rangle_t=2\int_0^t B_{ij}(X_s)\,ds.
\]
Since \(\operatorname{tr}B=1\), the second-moment maximal inequality gives
\(\mathbf P_x^\omega(\tau_{B_n(x)}\le T)\le2T/n^2\), proving nonexplosion.
The localization construction and uniqueness on the canonical coefficient
space give joint measurability of the solution kernel.

For \(\tau=\tau_{B_L(c)}\), the stopped square-distance martingale yields
\[
 2\mathbf E_x^\omega(t\wedge\tau)
 =\mathbf E_x^\omega|X_{t\wedge\tau}-c|^2-|x-c|^2
 \le L^2-|x-c|^2.
\]
Thus \(\tau<\infty\) almost surely. Path continuity and bounded convergence
now give the first identity in \eqref{reg:eq:exit-moments}. By the strong
Markov property and Tonelli's theorem,
\[
 \mathbf E_x^\omega\tau^m
 =m\mathbf E_x^\omega\int_0^\tau
       \mathbf E_{X_s}^\omega\tau^{m-1}\,ds.
\]
Induction proves the moment bound; interpolation also controls noninteger moments.
\end{proof}
The physical diffusion is recovered through the inverse of the clock
\begin{equation}
 t(s)=\int_0^s a(X_v)^{-1}\,dv.
 \label{reg:eq:physical-time}
\end{equation}
This clock is finite up to bounded-domain exits, since \(a\) is bounded below
on compact sets, and \(t(s)\ge s/d\). The time-changed generator is \(A:D^2\),
so exit locations are unchanged. Section~\ref{sec:green} controls the physical
clock; here we work in normalized time.

\begin{lemma}[Sparse bad clusters]
\label{lem:reg:bad-clusters}
Under H1--H5, one can choose a deterministic \(\kappa>0\) for which the following
properties hold. Set
\[
 Q_k=k+[-1/2,1/2]^d,\qquad
 \widehat Q_k=k+[-3,3]^d,\qquad
 I_k=\mathbf1_{\{\inf_{\widehat Q_k}\lambda<\kappa\}},
 \quad k\in\mathbb Z^d.
\]
Declare lattice points adjacent when their \(\ell^\infty\) distance is one.
Let \(\mathcal C(k)\) be the bad cluster containing \(k\), with the cluster of a good
point defined to be empty. Then
\begin{equation}
 \mathbb P(|\mathcal C(k)|\ge n)\le Ce^{-cn}.
 \label{reg:eq:cluster-tail}
\end{equation}
Almost surely, all bad clusters are simultaneously finite. Every connected
component of the open set
\[
 W=\{x:\lambda_2(x)<\kappa\},\qquad
 \lambda_2(x)=\inf_{\overline B_2(x)}\lambda
\]
is also bounded; this set is covariant under all translations.
If one of its components \(C\) intersects \(Q_k\), then
\begin{equation}
 \operatorname{diam}C\le C_d(1+|\mathcal C(k)|).
 \label{reg:eq:component-diameter}
\end{equation}
\end{lemma}

\begin{proof}
A finite cover by unit balls, H5 and stationarity give
\[
 \mathbb E[(\inf_{\widehat Q_k}\lambda)^{-d}]
 \le N_d\mathbb E[\lambda_*(0)^{-d}]=:M_0,
 \qquad p:=\mathbb P(I_k=1)\le M_0\kappa^d.
\]
Choose an integer \(K_0>\ell+6\). The cubes associated with distinct points
of a residue class modulo \(K_0\) are separated by more than \(\ell\).
H3, applied successively to one cube and the union of the others, gives joint
independence within each class.

There are at most \(C_d^n\) connected lattice animals of size \(n\) containing
a fixed anchor: encode a spanning tree by a depth-first walk of length at
most \(2n\), with at most \(3^d-1\) choices per step. Each animal contains
at least \(n/K_0^d\) vertices of one color, whence
\[
 \mathbb P(|\mathcal C(k)|\ge n)\le C_d^n p^{n/K_0^d}.
\]
A sufficiently small \(\kappa\) proves \eqref{reg:eq:cluster-tail} and,
by a countable union over anchors, finiteness of all clusters.

Continuity of \(\lambda_2\) makes \(W\) open. If \(x\in Q_k\cap W\), then
\(\overline B_2(x)\subset\widehat Q_k\), so \(I_k=1\). The cells intersecting
a component of \(W\) therefore form a connected bad animal. They belong to
one finite cluster, proving boundedness and \eqref{reg:eq:component-diameter}.
\end{proof}

\begin{proposition}[The coarse kernel]
\label{prop:reg:coarse-process}
Under H1--H5, take \(\kappa\) from Lemma~\ref{lem:reg:bad-clusters}.
Let \(G=\mathbb R^d\setminus W\). For \(x\in W\), let \(C_x\) be the component
of \(W\) containing the starting point, and define
\[
 S_0=\inf\{t:X_t\notin C_x\};
\]
for \(x\in G\), set \(S_0=0\). Further set
\[
 Y=X_{S_0},\qquad
 S=\inf\{t\ge S_0:|X_t-Y|=1\},\qquad
 K_\omega(x,E)=\mathbf P_x^\omega(X_S\in E).
\]
Let \(D(x)=\operatorname{diam}C_x\) on \(W\), and set it to zero on \(G\).
Then \(K\) is a jointly measurable probability kernel, covariant under all
translations, and
\begin{align}
 &\sup_{0\le t\le S}|X_t-x|\le D(x)+1,
 \qquad
 \mathbf E_x^\omega S^m\le m!\left(\frac{(D(x)+2)^2}{2}\right)^m,
 \label{reg:eq:coarse-path}\\
 &\int(y-x)K_\omega(x,dy)=0,\qquad
 \int(y-x)(y-x)^T K_\omega(x,dy)\ge \frac{\kappa}{d}I.
 \label{reg:eq:coarse-moments}
\end{align}
Iterating this stopping-time construction gives a coarse-grained process whose
stopping times do not accumulate.
\end{proposition}

\begin{proof}
Boundedness of \(C_x\) and Lemma~\ref{lem:reg:normalized-diffusion} give
\(S_0<\infty\). A boundary point of an open component lies outside \(W\):
a connected neighborhood contained in \(W\) would otherwise join that point
to the component. Hence path continuity gives \(Y\in G\), and
\(B\ge(\kappa/d)I\) on \(B_2(Y)\). By the strong Markov property, the
second segment is a unit-ball exit and \(S\) is a finite stopping time.

The path stays within distance \(D(x)+1\) of \(x\), so \(S\) precedes exit
from \(B_{D(x)+2}(x)\). Equation~\eqref{reg:eq:exit-moments} proves
\eqref{reg:eq:coarse-path}. The bounded martingale \(X_{t\wedge S}\) has
terminal mean \(x\). For the second increment \(V=X_S-Y\),
\[
 \mathbf E^\omega[V\mid Y]=0,\qquad
 \mathbf E^\omega[VV^T\mid Y]
 =2\mathbf E_Y^\omega\int_0^{\tau_{B_1(Y)}}B(X_s)\,ds
 \ge (\kappa/d)I,
\]
since \(\mathbf E_Y^\omega\tau_{B_1(Y)}=1/2\). The cross terms with
\(Y-x\) vanish, giving \eqref{reg:eq:coarse-moments}.

For measurability, enumerate rational balls whose closures lie in \(W\).
This condition is \(\max\lambda_2<\kappa\) on the closed ball and is
measurable by continuity. The union of balls linked by finite intersecting
chains to a ball containing \(x\) is precisely \(C_x\). Thus
\((x,\omega,y)\mapsto\mathbf1_{\{y\in C_x\}}\) is measurable; combined with
the measurable diffusion family, this proves measurability of \(K\).
The construction commutes with translations:
\[
 K_{\tau_z\omega}(x,E)=K_\omega(x+z,E+z).
\]

For each second-segment duration \(V_t\), the conditional maximal inequality
implies \(\mathbf P(V_t\le s\mid\text{past})\le2s\). The conditional
probability that \(n\) successive segments all last less than \(1/4\) is
therefore at most \(2^{-n}\). A countable union over starting indices shows
that infinitely many segments last at least \(1/4\); the stopping times
cannot accumulate. Equation~\eqref{reg:eq:physical-time} gives the same
conclusion in physical time.
Finally, every compact set meets only finitely many finite bad clusters.
Inductively, the points reachable after any fixed number of steps lie in
an environment-dependent compact set. The iterated process is consequently
an integrable quenched martingale at each finite step.
\end{proof}

\begin{lemma}[A lower bound in convex order]
\label{lem:reg:convex-order}
Let \(\nu\) be a mean-zero probability measure supported on \(\overline B_1\),
with covariance \(C_\nu=\int zz^T\,d\nu\ge a_0I\), where \(a_0>0\).
Then, for every finite convex function \(f\),
\[
 \int f\,d\nu\ge\fint_{B_{a_0}}f.
\]
In particular, the kernel of Proposition~\ref{prop:reg:coarse-process} satisfies
\begin{equation}
 K_\omega f(x)\ge\fint_{B_{\kappa/d}(x)}f(y)\,dy
 \label{reg:eq:convex-order}
\end{equation}
for every convex function \(f\) defined on a convex domain containing the relevant
paths and averaging balls.
\end{lemma}

\begin{proof}
For \(|y|\le a_0\), the measure
\[
 d\nu_y(z)=(1+y\cdot C_\nu^{-1}z)\,d\nu(z)
\]
is nonnegative, has mass one and barycenter \(y\). Jensen's inequality gives
\(f(y)\le\int f(z)(1+y\cdot C_\nu^{-1}z)\,d\nu(z)\).
Average over \(B_{a_0}\) to obtain the first assertion.
Apply it to the conditional law of the second coarse increment, and then
apply Jensen to the first increment, whose mean is zero:
\[
 K_\omega f(x)\ge
 \mathbf E_x^\omega\left[\fint_{B_{\kappa/d}}f(Y+z)\,dz\right]
 \ge\fint_{B_{\kappa/d}}f(x+z)\,dz.
\]
The specified convex domain contains every averaging point.
\end{proof}

\begin{remark}
The comparison \eqref{reg:eq:convex-order} is restricted to convex functions;
it is not a minorization of the kernel by Lebesgue measure.
\end{remark}

\begin{proposition}[Localization and concentration]
\label{prop:reg:geometry-concentration}
Under H1--H5, there exist a deterministic fine grid and random quantities
\(J_Q\ge1\) that simultaneously control the coarse-grained path radius and
normalized time for every starting point in the cell \(Q\):
\[
 \sup_{x\in Q}\sup_{t\le S_x}|X_t-x|\le CJ_Q,\qquad
 \sup_{x\in Q}\mathbf E_x^\omega S_x\le CJ_Q^2.
\]
Each \(J_Q\) has an exponential tail with deterministic constants, and
the truncation \(J_Q\wedge n\) depends only on the coefficients in the deterministic
\(Cn\)-neighborhood of \(Q\).
For every fixed \(m>0\), there exist \(C_m,c_m,\vartheta_m>0\) such that
\begin{equation}
 \mathbb P\left(R^{-d}\sum_{Q\cap B_{2R}(z)\ne\varnothing}J_Q^m>C_m\right)
 \le C_me^{-c_mR^{\vartheta_m}}
 \label{reg:eq:geometry-concentration}
\end{equation}
for every fixed \(z\) and \(R\ge1\).
There is also a local kernel \(K^L_\omega(x,\cdot)\), with deterministic
coefficient-dependence radius \(L\), satisfying
\begin{equation}
 \mathbb P(K^L_\omega(x,\cdot)\ne K_\omega(x,\cdot))\le Ce^{-cL}.
 \label{reg:eq:kernel-locality}
\end{equation}
\end{proposition}

\begin{proof}
Set \(a_0=\kappa/d\), and choose an even integer \(N_0\) with
\(\sqrt d/N_0\le a_0/8\). Use the half-open cubes of \(N_0^{-1}\mathbb Z^d\).
For \(H_k=1+|\mathcal C(k)|\), let \(J_Q\) be a fixed multiple of the
maximum of \(1+H_k\) over the unit cells meeting \(Q\) and a fixed finite
neighborhood. Equation~\eqref{reg:eq:component-diameter} and
Proposition~\ref{prop:reg:coarse-process} give the path and time bounds,
and \eqref{reg:eq:cluster-tail} gives exponential tails.

Explore an anchored bad cluster until either \(n\) bad vertices have been
found or all its neighboring vertices have been checked and are good.
Each exploration step moves one lattice unit and inspects a fixed-size cube.
Thus \(H_k\wedge n\), and hence \(J_Q\wedge n\), are determined within
radius \(Cn\).
For \(n=R^b\), the variables \((J_Q\wedge n)^m\) are bounded by \(n^m\)
and can be partitioned into \(Cn^d\) independent residue classes. Each
class contains \(\asymp R^dn^{-d}\) cells in the ball when \(R/n\) is large,
as follows by comparison with inscribed and circumscribed boxes.
Hoeffding's inequality bounds a fixed-tolerance deviation of each class
average by
\[
 C\exp\{-cR^dn^{-d-2m}\}.
\]
A deviation of the full average forces a deviation in one class. The union
over classes costs a polynomial factor; the truncation mismatch probability
is at most \(CR^de^{-cn}\). The expectations are bounded uniformly by the
exponential tail. Choose \(b>0\) with \(d-b(d+2m)>0\); then any
\[
 0<\vartheta_m<\min\{b,d-b(d+2m)\}
\]
gives \eqref{reg:eq:geometry-concentration}.

For \(0<L<6\), set \(K^L_\omega(x,\cdot)=\delta_x\); enlarging \(C\)
in \eqref{reg:eq:kernel-locality} covers this bounded range of \(L\).
For \(L\ge6\), inspect \(A\) in \(B_L(x)\), which determines
\(W\cap B_{L-2}(x)\). Accept its component containing \(x\) only if its
closure lies in \(B_{L-3}(x)\); otherwise set \(K^L(x,\cdot)=\delta_x\).
On acceptance, its boundary lies in the inspected region and outside \(W\),
so every full-space extension of these coefficients has the same component
and the same two-segment kernel. For \(x\in G\), radius three suffices.
Failure requires a bad cluster of diameter at least \(cL\), and
\eqref{reg:eq:cluster-tail} proves \eqref{reg:eq:kernel-locality}.
\end{proof}

\begin{remark}
Finite-range independence applies to the localized kernels \(K^L\) at
\emph{deterministic} locations separated by more than \(2L+\ell\).
The full kernel and the locations selected by the path need not have this
independence property.
\end{remark}

\subsection{Contact sets and harmonic approximation}
\label{reg:subsec:fixed-tolerance}

We first prove a contact-plane estimate of
Alexandrov--Bakelman--Pucci type; see \cite{CC95} for the local elliptic
version. For the stopped kernel, convex order controls the volume of
contact slopes on each fine cell. The proof keeps the full coarse path
inside the comparison region, as required by the nonlocal kernel.

\begin{lemma}[A coarse ABP estimate]
\label{lem:reg:coarse-abp}
Let \(1\le k\le R\), and let \(h\) be continuous on
\(\overline B_{R+3k}\), nonpositive on \(B_{R+3k}\setminus B_{R+2k}\),
with every positive global upper contact point lying in \(\overline B_{R+1}\).
Assume that the coarse paths started at these points, and all points needed
for the convex-order averages, remain in \(B_{R+2k}\).
Let \(\Gamma\) be the least concave majorant of \(h^+\) on the larger ball,
and write
\[
 \mathcal C=\{x:h(x)>0,\ h(x)=\Gamma(x)\},\qquad
 f(x)=h(x)-K_\omega h(x)\quad(x\in\mathcal C).
\]
Then
\begin{equation}
 \max h\le CR
 \left(\sum_{Q\cap\mathcal C\ne\varnothing}
   \sup_{x\in Q\cap\mathcal C}f(x)^d\right)^{1/d},
 \label{reg:eq:coarse-abp}
\end{equation}
where the sum uses the fixed fine grid from
Proposition~\ref{prop:reg:geometry-concentration}, and \(C\) depends only on
\(d,\kappa\).
\end{lemma}

\begin{proof}
For \(x\in\mathcal C\), convex order applied to \(-\Gamma\) gives
\[
 0\le\Gamma(x)-\fint_{B_{a_0}(x)}\Gamma
 \le\Gamma(x)-K_\omega\Gamma(x)
 \le h(x)-K_\omega h(x)=f(x),\qquad a_0=\kappa/d.
\]
Given \(p\in\partial^+\Gamma(x)\), define
\(g(z)=\Gamma(x)+p\cdot z-\Gamma(x+z)\) on \(B_{a_0}\).
Then \(g\) is convex, nonnegative, vanishes at zero, and
\(\fint_{B_{a_0}}g\le f(x)\). Jensen's inequality implies
\[
 g(z)\le\fint_{B_{a_0/2}(z)}g\le2^df(x),\qquad |z|\le a_0/2.
\]
For \(|z|\le a_0/4\) and \(q\in\partial g(z)\), apply the subgradient
inequality at \(z\pm(a_0/4)q/|q|\) to obtain
\(|q|\le Cf(x)/a_0\). Hence
\[
 \partial^+\Gamma(B_{a_0/4}(x))\subset B_{Cf(x)/a_0}(p).
\]
Since each fine cell has diameter at most \(a_0/8\), choosing one contact
point \(x_Q\) per cell bounds the volume of its contact-slope image by
\(Cf(x_Q)^d\).

Let \(m=\max h>0\). For \(|p|<m/[4(R+3k)]\), a maximizer of
\(h(y)-p\cdot y\) is interior and has \(h(y)>0\). Its supporting plane
is bounded below on the larger ball by \(m-2|p|(R+3k)>m/2\), so it
supports \(h^+\) as well. Thus every such \(p\) belongs to the slope
image of \(\mathcal C\). Comparing volumes yields
\[
 c(m/R)^d\le C\sum_{Q\cap\mathcal C\ne\varnothing}
                 \sup_{x\in Q\cap\mathcal C} f(x)^d,
\]
which proves \eqref{reg:eq:coarse-abp}; the case \(m\le0\) is immediate.
\end{proof}

\begin{lemma}[Good blocks for quadratic data]
\label{lem:reg:quadratic-blocks}
Assume H1--H5, and let
\(\mathcal S_0=\{M\in\mathbb S^d:\bar A:M=0\}\).
For every \(\epsilon,p_0>0\), there is a deterministic \(L<\infty\)
such that the following event at each grid cell fails with probability
at most \(p_0\).
If \(q_Q\) is the center of the cell \(Q\), let \(v_{M,Q,L}\) be
\(A\)-harmonic in \(B_L(q_Q)\), with boundary values
\(P_{M,Q}(x)=\frac12(x-q_Q)\cdot M(x-q_Q)\).
The good event is
\begin{equation}
 \sup_{\substack{M\in\mathcal S_0\\|M|\le1}}
 \|v_{M,Q,L}-P_{M,Q}\|_{L^\infty(B_L(q_Q))}
 \le\epsilon L^2.
 \label{reg:eq:quadratic-good}
\end{equation}
Moreover, there are \(C,c>0\), depending on \(L,p_0\), such that
\begin{equation}
 \mathbb P\big(\#\{Q\cap B_R(z)\ne\varnothing:Q\text{ is bad}\}>
                  Cp_0R^d\big)\le Ce^{-cR^d}
 \label{reg:eq:bad-block-count}
\end{equation}
for every fixed center and all sufficiently large \(R\).
\end{lemma}

\begin{proof}
Choose an orthonormal basis of \(\mathcal S_0\). Qualitative Dirichlet
homogenization \cite{ASdeg} gives, for each basis matrix,
\[
 L^{-2}\|v_{M,Q,L}-P_{M,Q}\|_\infty\longrightarrow0
 \quad\text{almost surely}.
\]
Linearity makes this convergence uniform on the unit ball of
\(\mathcal S_0\). Stationarity therefore allows one deterministic \(L\)
for which \eqref{reg:eq:quadratic-good} fails with probability at most \(p_0\).
Each failure indicator depends on \(A|_{B_L(q_Q)}\). Coloring the fixed
fine grid into finitely many independent residue classes and applying a
Chernoff bound in each class proves \eqref{reg:eq:bad-block-count}.
\end{proof}

\begin{proposition}[Fixed-tolerance harmonic approximation]
\label{prop:reg:fixed-tolerance}
Under H1--H5, for every \(t>0\), there are deterministic
\(R_t,C_t,c_t,\beta_t>0\) and \(C_*>1\) such that, for each fixed
\(z\) and \(R\ge R_t\), there is an event
\(E_R(t,z)\in\sigma(A|_{B_{C_*R}(z)})\) satisfying
\[
 \mathbb P(E_R(t,z)^c)\le C_t e^{-c_tR^{\beta_t}}.
\]
On this event, the following holds simultaneously for every
\(F\in C^3(\overline B_1)\) with \(\bar A:D^2F=0\):
if \(u_F\) is \(A\)-harmonic in \(B_R(z)\), with boundary values
\(F((x-z)/R)\), then
\begin{equation}
 \|u_F-F((\cdot-z)/R)\|_{L^\infty(B_R(z))}
 \le t\|F\|_{C^3(\overline B_1)}.
 \label{reg:eq:fixed-tolerance}
\end{equation}
The event does not depend on \(F\).
\end{proposition}

\begin{proof}
By translation, take \(z=0\). Write \(K_F=\|F\|_{C^3}\), \(F_R(x)=F(x/R)\),
and extend \(F\) to a larger ball by a bounded \(C^3\) extension operator.
Put \(k=R^{1/2}\). Choose \(0\le\chi_R\le1\), equal to one on \(B_R\)
and supported in \(B_{R+1}\), and solve
\begin{equation}
 B:D^2H'=\chi_R B:D^2F_R\quad\text{in }B_{R+2k},
 \qquad H'=0\quad\text{on }\partial B_{R+2k}.
 \label{reg:eq:expanded-problem}
\end{equation}
Local ellipticity ensures that this boundary-value problem has a solution,
which can also be represented by the diffusion potential.
The source is bounded by \(CK_FR^{-2}\). The ball-exit identity gives, on
\(\partial B_R\),
\[
 |H'|\le CK_FR^{-2}((R+2k)^2-R^2)/2\le CK_Fk/R.
\]
Since \(H=F_R-u_F\) has the same equation in \(B_R\) and zero boundary data,
comparison yields
\begin{equation}
 \|H-H'\|_{L^\infty(B_R)}\le CK_Fk/R.
 \label{reg:eq:expanded-comparison}
\end{equation}

Fix \(\gamma=C_1\epsilon K_FR^{-2}\), with \(\epsilon\le C_1^{-1}\).
Extend \(H'\) by zero and set
\[
 h(x)=H'(x)+\gamma|x|^2-\gamma(R+3k)^2\quad\text{on }B_{R+3k}.
\]
The function is nonpositive on the outer shell and strictly negative on
\(\partial B_{R+2k}\). Moreover, \(B:D^2h=2\gamma>0\) outside
\(\overline B_{R+1}\), excluding upper affine contacts there. Thus every
positive contact lies in \(\overline B_{R+1}\).
On the event \(CJ_Q\le k/2\) for all relevant cells, coarse paths stay in
\(B_{R+2k}\), and Dynkin's formula gives
\begin{equation}
 0\le h(x)-K_\omega h(x)
 \le CK_FR^{-2}\mathbf E_x^\omega S_x
 \le CK_FR^{-2}J_Q^2,\qquad x\in Q\cap\mathcal C.
 \label{reg:eq:contact-defect}
\end{equation}

A good cell with \(B_L(q_Q)\subset B_R\) contains no such contact point.
Indeed, for \(x\in Q\) and \(\tau=\tau_{B_L(q_Q)}\), harmonicity of
\(H'-F_R\) and zero drift give
\begin{equation}
 \mathbf E_x^\omega[h(X_\tau)-h(x)]
 =\mathbf E_x^\omega[F_R(X_\tau)-F_R(x)]
       +\gamma\mathbf E_x^\omega|X_\tau-x|^2.
 \label{reg:eq:exclude-contact}
\end{equation}
Taylor expansion at \(x\) has a vanishing expected linear term and Hessian
\(D^2F_R(x)\in\mathcal S_0\). Recentring this quadratic at \(q_Q\) adds an
affine function. Equation~\eqref{reg:eq:quadratic-good} therefore implies
\[
 \big|\mathbf E_x^\omega[F_R(X_\tau)-F_R(x)]\big|
 \le C_d\epsilon K_FL^2R^{-2}+CK_FL^3R^{-3}.
\]
Meanwhile, \(\mathbf E_x^\omega|X_\tau-x|^2=L^2-|x-q_Q|^2\ge L^2/2\)
for \(L\) larger than the cell diameter. Choose \(C_1>4C_d\) and then
\(R\ge CL/\epsilon\). The right-hand side of \eqref{reg:eq:exclude-contact}
is positive, whereas an upper supporting plane and zero drift would make
it nonpositive.

Contacts thus occur only in bad cells or in a boundary layer of thickness
\(L+C\). Their number is at most \(C(p_0+L/R)R^d\) on the event of
\eqref{reg:eq:bad-block-count}. Equation~\eqref{reg:eq:geometry-concentration}
with \(m=4d\) and Cauchy--Schwarz give
\[
 R^{-d}\sum_{Q\cap\mathcal C\ne\varnothing}J_Q^{2d}
 \le\left(R^{-d}\#\{Q:Q\cap\mathcal C\ne\varnothing\}\right)^{1/2}
     \left(R^{-d}\sum_QJ_Q^{4d}\right)^{1/2}
 \le C(p_0+L/R)^{1/2}.
\]
Together with \eqref{reg:eq:contact-defect}, the coarse ABP estimate yields
\(\max h\le CK_F(p_0+L/R)^{1/(2d)}\). Restoring the quadratic, using
\eqref{reg:eq:expanded-comparison}, and applying the same argument to \(-F\),
we obtain
\begin{equation}
 \|F_R-u_F\|_\infty
 \le CK_F\left(\epsilon+p_0^{1/(2d)}
              +(L/R)^{1/(2d)}+R^{-1/2}\right).
 \label{reg:eq:fixed-tolerance-error}
\end{equation}
Choose first \(\epsilon,p_0\) so the first two terms contribute at most
\(t/2\), then \(L\) by Lemma~\ref{lem:reg:quadratic-blocks}, and finally
\(R_t\) to bound the remaining terms.

We finish by defining the common event explicitly. Let \(\mathcal Q_R\)
be the cells meeting \(B_{2R}\) and \(n_R=\lfloor cR^{1/2}\rfloor\), with
\(c\) small. Require
\[
 \max_{Q\in\mathcal Q_R}(J_Q\wedge n_R)<n_R,\qquad
 R^{-d}\sum_{Q\in\mathcal Q_R}(J_Q\wedge n_R)^{4d}\le C_{4d},
\]
and the preceding bad-cell count. The first condition certifies that the
truncations are exact and that \(CJ_Q<k/2\). These conditions are determined
in \(B_{2R+Cn_R+L+C}\subset B_{C_*R}\). Their failures are bounded by the
cluster tail, \eqref{reg:eq:geometry-concentration}, and
\eqref{reg:eq:bad-block-count}, respectively, with stretched-exponential
probability. Their intersection defines \(E_R(t,0)\) and is independent
of \(F\), as required.
\end{proof}

\subsection{Decay of the affine excess}
\label{reg:subsec:q1}

To use harmonic approximation for arbitrary solutions, we first control
oscillations down to a mesoscopic scale. The moment in H5 provides a
common event for all required balls.

\begin{lemma}[Mesoscopic ellipticity averages]
\label{lem:reg:ellipticity-averages}
Assume H1--H5, and choose \(d-1<p<d\), \(q=d/p\), and
\(a_0\in(1/q,1)\).
There are deterministic \(C_0,C<\infty\) such that, for each fixed
center \(z\) and \(R\ge1\), there is a coefficient event
\(\mathcal V_R(z)\) satisfying
\[
 \mathbb P(\mathcal V_R(z)^c)\le CR^{-\gamma_H},
 \qquad \gamma_H=d(a_0q-1)>0.
\]
On this event, simultaneously for all \(y\in B_R(z)\) and
\(R^{a_0}\le s\le R\),
\begin{equation}
 \fint_{B_s(y)}\lambda^{-p}\le C_0.
 \label{reg:eq:average-ellipticity}
\end{equation}
\end{lemma}

\begin{proof}
Partition space into small cubes \(Q_i\subset B_1(x_i)\) and put
\(Z_i=(\inf_{Q_i}\lambda)^{-p}\). Then
\(\mathbb E Z_i^q\le\mathbb E\lambda_*(0)^{-d}\), and the \(Z_i\) have
finite-range dependence. For independent centered variables and \(1<q<2\),
\[
 \mathbb E\left|\sum_i\zeta_i\right|^q\le C_q\sum_i\mathbb E|\zeta_i|^q.
\]
Indeed, iterate
\(|a+b|^q\le |a|^q+q\operatorname{sgn}(a)|a|^{q-1}b+C_q|b|^q\)
and take conditional expectations, which remove the linear terms.
Apply this inequality within deterministic color classes. Markov's inequality,
with \(C_1\) above a uniform expectation bound, gives
\begin{equation}
 \mathbb P\left(\fint_{B_s(y)}\lambda^{-p}>C_1\right)
 \le Cs^{-d(q-1)},\qquad s\ge1.
 \label{reg:eq:one-ball-average}
\end{equation}

For \(s_j=2^jR^{a_0}\), up to the first \(s_j\ge R\), take deterministic
nets in \(B_{2R}(z)\) of spacing \(s_j/(10\sqrt d)\). There are at most
\(C(R/s_j)^d\) net points. Require the average
bound underlying \eqref{reg:eq:one-ball-average} on every radius-\(4s_j\) ball centered at these points. Each requested
\(B_s(y)\) lies in such a ball of radius at most \(8s\); the volume ratio
gives \eqref{reg:eq:average-ellipticity}. The union bound is
\[
 C\sum_j(R/s_j)^d s_j^{-d(q-1)}
 \le CR^dR^{-a_0dq}=CR^{-d(a_0q-1)}.
\]
These finitely many conditions define \(\mathcal V_R(z)\).
\end{proof}

\begin{lemma}[A mesoscopic H\"older modulus]
\label{lem:reg:coarse-holder}
With the parameter choices of
Lemma~\ref{lem:reg:ellipticity-averages}, there are deterministic
\(\nu\in(0,1)\) and \(C<\infty\) such that, on \(\mathcal V_R(z)\),
every function that is \(A\)-harmonic in \(B_R(z)\) and satisfies
\(\|u\|_{L^\infty(B_R(z))}\le1\) obeys
\begin{equation}
 |u(x)-u(y)|
 \le C\left[\left(\frac{|x-y|}{R}\right)^\nu
              +R^{-(1-a_0)\nu}\right],
 \qquad x,y\in B_{3R/4}(z).
 \label{reg:eq:coarse-holder}
\end{equation}
\end{lemma}

\begin{proof}
\cite[Proposition~1.2]{Bowman26} states that,
for continuous positive \(\lambda\), \(0<A\le I\), \(p>d-1\) and
\(\fint_{B_{5s}(x)}\lambda^{-p}\le C_0\),
\[
 \operatorname{osc}_{B_s(x)}u
 \le(1-\theta_0)\operatorname{osc}_{B_{5s}(x)}u,
\]
where \(\theta_0=\theta_0(d,p,C_0)>0\). Homogeneous solutions satisfy the
required two Pucci inequalities. This oscillation proposition requires only
\(p>d-1\), rather than the stronger assumption of that paper's Harnack theorem.

If \(R^{1-a_0}<100\), the bound \(\operatorname{osc}u\le2\) proves
\eqref{reg:eq:coarse-holder} after increasing \(C\). We may therefore assume
\(R^{a_0}\le R/100\), so the following interval of iteration is nonempty.
For \(x\in B_{3R/4}(z)\), iterate from radius \(R/20\) by factors of five,
stopping at radius \(\asymp R^{a_0}\). All outer balls lie in \(B_R(z)\),
and \eqref{reg:eq:average-ellipticity} gives the same contraction at each step.
Consequently,
\[
 \operatorname{osc}_{B_s(x)}u\le C(s/R)^\nu,
 \qquad R^{a_0}\le s\le R/20,
\]
for \(0<\nu<\min\{1,-\log(1-\theta_0)/\log5\}\).
Use \(s\asymp|x-y|\) above the stopping scale and the last ball below it.
For \(|x-y|\ge R/40\), use \(\|u\|_\infty\le1\).
This proves \eqref{reg:eq:coarse-holder} on the same event for all solutions.
\end{proof}

\begin{proposition}[One-step improvement of the affine excess]
\label{prop:reg:one-step}
Assume H1--H5, and fix \(\alpha\in(0,1)\).
There are \(\theta_\alpha=2^{-m_\alpha}\in(0,1/4]\) and
\(R_\alpha,C_\alpha<\infty\) such that, for every fixed \(z\) and
\(R\ge R_\alpha\), there is a measurable coefficient event
\(\mathcal F_{\alpha,R}(z)\) satisfying
\[
 \mathbb P(\mathcal F_{\alpha,R}(z)^c)\le C_\alpha R^{-\gamma_H},
\]
on which, simultaneously for every function \(A\)-harmonic in
\(B_R(z)\),
\begin{equation}
 \mathcal E(u;\theta_\alpha R,z)
 \le\frac12\theta_\alpha^\alpha\mathcal E(u;R,z).
 \label{reg:eq:one-step}
\end{equation}
\end{proposition}

\begin{proof}
Subtract a best affine approximation, rescale, and divide by the unnormalized
excess to reduce to
\[
 A(z+Rx):D^2u=0\quad\text{in }B_1,\qquad \|u\|_{B_1}\le1.
\]
A best approximation exists because \(\mathcal P_1\) is finite dimensional.
The conclusion is immediate when the excess is zero or infinite.
On \(\mathcal V_R(z)\), convolve with
a smooth nonnegative mollifier of scale \(h<1/12\) to obtain \(g\) on
\(B_{2/3}\). Lemma~\ref{lem:reg:coarse-holder} gives
\[
 \|u-g\|_{B_{2/3}}\le C(h^\nu+R^{-(1-a_0)\nu}),\qquad
 \|g\|_{C^4(B_{2/3})}\le C_h,\qquad |g|\le1.
\]
Let \(v\) and \(u_g\) solve the constant- and variable-coefficient harmonic
problems in \(B_{1/2}\), both with boundary values \(g\). Comparison and
boundary Schauder estimates give
\[
 |v|\le1,\qquad \|v\|_{C^3(\overline B_{1/2})}\le C_h,\qquad
 \|u-u_g\|_{B_{1/2}}\le C(h^\nu+R^{-(1-a_0)\nu}).
\]
Proposition~\ref{prop:reg:fixed-tolerance}, applied at physical radius
\(R/2\) to \(F(y)=v(y/2)\), gives \(\|u_g-v\|_{B_{1/2}}\le tC_h\)
on \(E_{R/2}(t,z)\). Since the proposition holds for every \(F\) on the
same event, it applies to this environment-dependent choice of \(F\). Thus
\begin{equation}
 \|u-v\|_{B_{1/2}}
 \le C(h^\nu+R^{-(1-a_0)\nu})+tC_h.
 \label{reg:eq:harmonic-approximation}
\end{equation}
Interior estimates give \(\inf_{\ell\in\mathcal P_1}\|v-\ell\|_{B_\theta}
\le C\theta^2\). Choose a dyadic \(\theta\le1/4\) with
\(C\theta\le\theta^\alpha/4\), then choose \(h,t,R_\alpha\), in this
order, so that \eqref{reg:eq:harmonic-approximation} is at most
\(\theta^{1+\alpha}/4\). It follows that
\[
 \theta^{-1}\inf_{\ell}\|u-\ell\|_{B_\theta}
 \le\tfrac14\theta^\alpha+C\theta\le\tfrac12\theta^\alpha.
\]
Rescaling proves \eqref{reg:eq:one-step}. The event
\(\mathcal F_{\alpha,R}(z)=\mathcal V_R(z)\cap E_{R/2}(t,z)\) has the
claimed polynomial failure bound.
\end{proof}

\begin{proof}[Proof of Theorem~\ref{thm:reg:q1}]
Take
\[
 p=d-\tfrac12,\qquad q=d/(d-1/2),\qquad a_0=1-1/(4d).
\]
Then \(d-1<p<d\), \(1<q<2\), \(a_0>1/q\), and
\(\gamma_H=d(a_0q-1)=d/(4d-2)=\gamma_1\).
Choose \(n_{\rm det}\) with \(2^{n_{\rm det}}\ge R_\alpha\), and let
\(E_n=\mathcal F_{\alpha,2^n}(0)\). These measurable coefficient events
satisfy \(\mathbb P(E_n^c)\le C_\alpha2^{-n\gamma_1}\). Define
\[
 n_*=\max\{n_{\rm det},1+\sup\{n\ge n_{\rm det}:E_n^c\text{ occurs}\}\},
 \qquad r_{*,\alpha}(0)=2^{n_*},
\]
using \(n_{\rm det}-1\) for the empty supremum and \(+\infty\) for an
unbounded set. Borel--Cantelli gives finiteness almost surely, and
\[
 \mathbb P(r_{*,\alpha}(0)>t)
 \le C_\alpha\sum_{2^n\ge ct}2^{-n\gamma_1}
 \le C_\alpha t^{-\gamma_1}.
\]

On this event, let \(R\ge2r_*\) and \(r_*\le r\le R/2\).
Choose the largest dyadic \(S\le R\), so \(R/2<S\le R\), and the largest
\(j\ge0\) such that \(S_j=\theta_\alpha^jS\ge r\).
Every outer radius used in the iteration is dyadic and at least \(r_*\).
Equation~\eqref{reg:eq:one-step} gives
\(\mathcal E(u;S_j,0)\le(S_j/S)^\alpha\mathcal E(u;S,0)\).
Since \(r\le S_j<r/\theta_\alpha\) and the unnormalized excess is monotone,
\begin{align*}
 \mathcal E(u;r,0)
 &\le(S_j/r)\mathcal E(u;S_j,0)\\
 &\le\theta_\alpha^{-(1+\alpha)}(r/S)^\alpha\mathcal E(u;S,0)\\
 &\le2^{1+\alpha}\theta_\alpha^{-(1+\alpha)}
       (r/R)^\alpha\mathcal E(u;R,0).
\end{align*}
This includes \(j=0\) and \(r=r_*\). Define
\(r_{*,\alpha}(z,\omega)=r_{*,\alpha}(0,\tau_z\omega)\).
Stationarity gives the same tail and almost-sure estimate at each fixed
center. The common coefficient events and deterministic iteration give
simultaneity over all radii and all solutions.
\end{proof}

\begin{corollary}[Growth condition for the first-order Liouville property]
\label{cor:reg:liouville}
Assume H1--H5, and fix \(\alpha\in(0,1)\).
Almost surely, every entire \(A\)-harmonic function \(u\) satisfying
\begin{equation}
 \liminf_{R\to\infty}R^{-\alpha}\mathcal E(u;R,0)=0
 \label{reg:eq:liouville-growth}
\end{equation}
is affine.
In particular, on one event of full probability, for every fixed
\(\varepsilon>0\), all entire harmonic functions with growth
\(O(R^{2-\varepsilon})\) are affine.
\end{corollary}

\begin{proof}
For \(r\ge r_{*,\alpha}(0)\), apply \eqref{reg:eq:q1} along a sequence
realizing \eqref{reg:eq:liouville-growth}:
\[
 \mathcal E(u;r,0)\le C_\alpha r^\alpha
       R_k^{-\alpha}\mathcal E(u;R_k,0)\longrightarrow0.
\]
Closedness of \(\mathcal P_1\) in the uniform norm makes \(u\) affine on
each ball, and the affine representations agree on nested balls.
For growth \(O(R^{2-\varepsilon})\), choose \(\alpha>1-\varepsilon\).
A countable intersection of the events for \(\alpha_n\uparrow1\)
handles all \(\varepsilon>0\) simultaneously.
\end{proof}

\begin{remark}
The scales depend on \(\alpha\), and their tail estimates are for each
fixed center. The growth criterion excludes, for example,
\(R^2/\log R\); it gives no classification of all \(o(R^2)\) solutions.
\end{remark}

\subsection{Comparison of exit distributions}
\label{reg:subsec:exit-kernels}

Fixed-tolerance homogenization compares exit probabilities for finitely many
boundary pieces. It does not compare the full exit measures in total variation.
We combine this finite comparison at large scales with a common-mass estimate
at one fixed scale, and then couple the stopped paths across scales.

\begin{lemma}[Common exit mass at a fixed base scale]
\label{lem:reg:base-coupling}
Assume H1--H5 and fix \(M>2\), \(b\ge1\), and \(p_b>0\).
There exist deterministic constants \(\kappa_b,c_b>0\) and a local event
\[
 G_b(c)=\{\inf_{B_{(M+1)b}(c)}\lambda\ge\kappa_b\},
 \qquad \mathbb P(G_b(c)^c)\le p_b,
\]
such that, on this event, the exit distributions from \(B_{Mb}(c)\) of the normalized diffusion started at any
\(x,y\in\overline B_b(c)\) have common mass at least \(c_b\).
The two complete stopped paths admit a jointly measurable coupling whose exit points agree with probability at least \(c_b\).
\end{lemma}

\begin{proof}
Cover the deterministic ball \(B_{(M+1)b}(c)\) by \(N_b<\infty\) unit balls.
By H5, stationarity, and Markov's inequality,
\[
 \mathbb P(G_b(c)^c)
 \le N_b\mathbb E[\lambda_*(0)^{-d}]\kappa_b^d.
\]
Choose \(\kappa_b\) so that the right-hand side is at most \(p_b\).
On \(G_b(c)\), \((\kappa_b/d)I\le B\le I\).
Let \(\omega_x\) be the exit distribution from \(B_{Mb}(c)\).
For every nonnegative continuous boundary datum \(g\), its harmonic extension
\(u_g(x)=\int g\,d\omega_x\) satisfies the classical uniformly elliptic Harnack inequality
\cite{GT}. Thus there is a deterministic \(H_b<\infty\) such that
\[
 \int g\,d\omega_x\le H_b\int g\,d\omega_y
 \quad(x,y\in\overline B_b(c)).
\]
Approximation on open sets and outer regularity give
\(\omega_x\le H_b\omega_y\) as Borel measures.
Set \(\nu=\omega_c\) and \(c_b=(2H_b)^{-1}\). Then
\(\omega_x\ge c_b\nu\) simultaneously for all starting points in the inner ball.
Consequently,
\[
 \omega_x'=\frac{\omega_x-c_b\nu}{1-c_b},\qquad
 \mathsf C_{x,y}
 =c_b(\operatorname{Diag})_\#\nu+
                    (1-c_b)(\omega_x'\otimes\omega_y')
\]
define residual probabilities and an exit coupling with diagonal mass at least \(c_b\).

To lift the coupling to stopped paths, let \(\mathcal X=C([0,\infty),\mathbb R^d)\),
extending each stopped path constantly after its stopping time.
The path space is Polish, and the exit point is the Borel map
\(e(\gamma)=\lim_{n\to\infty}\gamma(n)\) on stopped paths, extended arbitrarily outside their support.
Parametrized disintegration into regular conditional probabilities gives jointly measurable kernels \(\Pi_x(z,d\gamma)\) such that
\[
 P_x^{\rm stop}(d\gamma)=\int\Pi_x(z,d\gamma)\omega_x(dz),
 \qquad \Pi_x(z,\{e=z\})=1\quad\omega_x\text{-almost everywhere}.
\]
The local coefficient restriction and starting point are standard Borel parameters.
Define
\[
 \mathsf Q_{x,y}(d\gamma_1,d\gamma_2)
 =\int\Pi_x(z_1,d\gamma_1)\Pi_y(z_2,d\gamma_2)
                                  \mathsf C_{x,y}(dz_1,dz_2).
\]
Its marginals are \(P_x^{\rm stop}\) and \(P_y^{\rm stop}\).
Since \(\nu\ll\omega_x,\omega_y\), the common-exit branch avoids the exceptional sets of the conditional kernels.
All kernels are local and jointly measurable.
\end{proof}

\begin{lemma}[Finite-partition comparison]
\label{lem:reg:exit-partitions}
Assume H1--H5. Fix \(0<\vartheta<1\), \(\epsilon>0\),
and a finite Borel partition \(P_1,\ldots,P_N\) of the unit sphere such that each piece has boundary of zero surface measure.
For each fixed center \(c\) and all sufficiently large \(R\), there is an event measurable with respect to the coefficients in \(B_{CR}(c)\),
with failure probability at most \(Ce^{-cR^\delta}\), on which
\begin{equation}
 \sup_{x\in\overline B_{\vartheta R}(c)}
 \|p_A(x)-p_{\bar A}(x)\|_{\rm TV}\le\epsilon.
 \label{reg:eq:partition-tv}
\end{equation}
Here \(p_A(x)\) and \(p_{\bar A}(x)\) are the distributions of the labels of the pieces \(c+RP_i\)
through which the respective diffusions exit \(B_R(c)\).
For finite distributions, total variation is one half of the \(\ell^1\) distance.
\end{lemma}

\begin{proof}
For each piece and a small fixed width \(\rho>0\), choose smooth functions on the sphere satisfying
\[
 0\le g_i^-\le\mathbf1_{P_i}\le g_i^+\le1,
\]
with \(g_i^+-g_i^-\) supported in the \(\rho\)-neighborhood of the boundary of the piece.
Set the lower function to zero and the upper function to one on the boundary;
the inequalities then hold even if the random exit measure charges it.
Let \(F_i^\pm\) be their \(\bar A\)-harmonic extensions.
The constant-coefficient Poisson kernel is bounded for starting points in
\(\overline B_\vartheta\). The boundary-null assumption permits a choice of \(\rho\) such that
\[
 \sup_{\overline B_\vartheta}(F_i^+-F_i^-)\le\epsilon/N
 \quad\text{for every }i.
\]
Boundary regularity gives
\(K_\rho=\max_{i,\pm}\|F_i^\pm\|_{C^3(\overline B_1)}\).
Next choose \(t\le\epsilon/(NK_\rho)\)
and use the same local coefficient event from Proposition~\ref{prop:reg:fixed-tolerance}.
Comparison gives
\[
 |\omega_x^A(c+RP_i)-\omega_x^{\bar A}(c+RP_i)|
 \le2\epsilon/N
\]
simultaneously for every \(x\in\overline B_{\vartheta R}(c)\) and every \(i\).
Summing over \(i\) proves \eqref{reg:eq:partition-tv}.
\end{proof}

\begin{lemma}[Two-point and relative comparison]
\label{lem:reg:relative-exit}
On the event of Lemma~\ref{lem:reg:exit-partitions},
\begin{equation}
 \|p_A(x)-p_A(y)\|_{\rm TV}
 \le2\epsilon+C_{\bar A,\vartheta}|x-y|/R,
 \qquad x,y\in\overline B_{\vartheta R}(c).
 \label{reg:eq:two-start-tv}
\end{equation}
Suppose, in addition, that every piece \(P_i\) has strictly positive surface measure.
There is a constant \(H_0\ge4\), depending only on \(\bar A,d\) and independent of the fineness of the partition,
such that, on another local event with stretched-exponential failure probability,
\begin{equation}
 \omega_x^{A,B_{2r}(c)}(c+2rP_i)
 \le\frac{H_0}{2}\,
       \omega_y^{A,B_{2r}(c)}(c+2rP_i)
 \label{reg:eq:relative-exit}
\end{equation}
simultaneously for all \(x,y\in\overline B_r(c)\) and all pieces.
The deterministic starting scale of this event may depend on the partition ultimately fixed.
\end{lemma}

\begin{proof}
For any Borel boundary function \(0\le g\le1\),
its constant-coefficient harmonic extension has gradient bounded by \(C/R\) on \(\overline B_{\vartheta R}\).
Taking the supremum over \(g\) yields
\[
 \|\omega_x^{\bar A}-\omega_y^{\bar A}\|_{\rm TV}
 \le C|x-y|/R.
\]
Pass to the label distributions and apply \eqref{reg:eq:partition-tv}
to obtain \eqref{reg:eq:two-start-tv}.

The constant-coefficient Harnack inequality at radius ratio \(1:2\) gives
\(\bar p_x(i)\le H_{\rm bar}\bar p_y(i)\).
Positivity and continuity of its Poisson kernel give
\[
 a_*=\min_i\inf_{x\in\overline B_1}
            \omega_x^{\bar A,B_2}(2P_i)>0.
\]
Choose the absolute error for each piece in Lemma~\ref{lem:reg:exit-partitions} to be \(a_*/4\). This gives
\[
 p_x(i)\le(H_{\rm bar}+1/4)\bar p_y(i),\qquad
 p_y(i)\ge(3/4)\bar p_y(i).
\]
Thus \(p_x(i)\le2H_{\rm bar}p_y(i)\).
Thus \(H_0=4H_{\rm bar}\) is partition-independent; only \(a_*\) and the starting scale depend on the partition.
\end{proof}

\begin{remark}
We will choose centers from paths by rounding to a deterministic grid.
The good event is imposed beforehand at every candidate center.
For radius \(R\), observation region \(B_H\), and grid spacing \(\zeta R\),
a union bound gives failure probability at most
\[
 C(1+H/(\zeta R))^d e^{-cR^\delta}.
\]
This yields a pathwise conclusion at adaptive centers without conditioning
the environmental probability estimate on the path.
\end{remark}

\subsection{A terminal-scale Harnack bound}
\label{reg:subsec:crude-harnack}

The eventual Harnack argument amplifies a large ratio until it contradicts a
rough bound at a smaller scale. We obtain that bound by detouring around bad
clusters. Filling their holes allows the maximum principle to move endpoints
to good cubes, where local Harnack inequalities apply.

\begin{definition}[Filled sets and outer boundaries]
For a finite *-connected set \(C\subset\mathbb Z^d\),
let \(E(C)\) be the unique infinite nearest-neighbor connected component of its complement, and define
\[
 F(C)=\mathbb Z^d\setminus E(C),\qquad
 \Gamma(C)=\{k\in E(C):\operatorname{dist}_\infty(k,C)=1\}.
\]
Thus \(F(C)\) fills the finite holes, and accessibility to \(\Gamma(C)\)
uses nearest-neighbor paths in the complement.
\end{definition}

\begin{lemma}[Geometry of filled clusters]
\label{lem:reg:filled-clusters}
Let \(C,D\) be two distinct finite bad *-clusters as in Lemma~\ref{lem:reg:bad-clusters}.
Then:
\begin{enumerate}
 \item \(F(C)\) is contained in the coordinate bounding box of \(C\), and
       \(\operatorname{diam}_\infty F(C)\le |C|-1\).
 \item \(F(C)\) and \(F(D)\) are disjoint, or one contains the other.
 \item \(\Gamma(C)\) consists entirely of good lattice sites, is nearest-neighbor connected,
       and any two of its sites can be joined by a boundary path of length at most \((3^d-1)|C|\).
 \item For two distinct maximal filled sets in a finite family of clusters,
       \(\Gamma(C)\cap F(D)=\varnothing\).
 \item If \(\mathcal O_C=\operatorname{int}\bigcup_{j\in F(C)}Q_j\),
       then \(\partial\mathcal O_C\subseteq\bigcup_{k\in\Gamma(C)}Q_k\).
\end{enumerate}
\end{lemma}

\begin{proof}
An outward coordinate ray from a site outside the bounding box avoids \(C\).
Such a site belongs to \(E(C)\); the coordinate span of a *-connected animal
is at most \(|C|-1\). This proves (1).

Replace any *-edge in \(D\) by a nearest-neighbor path inside its unit box, moving one coordinate at a time.
This path avoids the distinct bad cluster \(C\), since each site is
*-adjacent to an endpoint in \(D\). Hence \(D\) lies in one
nearest-neighbor component of \(\mathbb Z^d\setminus C\).
If \(D\) lies in a finite component, the infinite connected set \(E(C)\) avoids \(D\),
so \(E(C)\subseteq E(D)\), or equivalently \(F(D)\subseteq F(C)\).
Interchanging \(C,D\) treats the symmetric case. Otherwise
\(D\subset E(C)\) and \(C\subset E(D)\).
Every finite hole of \(C\) avoids \(D\) and has a nearest-neighbor edge to
\(C\subset E(D)\), so it also lies in \(E(D)\).
Thus \(F(C)\subset E(D)\), proving (2).

Nearest-neighbor connectivity of the outer boundary is precisely \cite[Theorem 4]{Timar}.
Its external *-boundary uses nearest-neighbor accessibility to infinity,
as does \(\Gamma(C)\).
Each boundary site is *-adjacent to a site of \(C\), so
\(|\Gamma(C)|\le(3^d-1)|C|\).
Simple boundary paths give the length bound. Every boundary site is good,
since a bad site there would belong to \(C\). This proves (3).

For (4), maximality and (2) give disjointness and \(C\subset E(D)\).
If \(k\in\Gamma(C)\), choose a *-adjacent site \(u\in C\).
The path connecting \(k\) to \(u\) one coordinate at a time inside their unit box avoids the bad cluster \(D\).
Thus \(k\) is connected to \(E(D)\), and \(k\notin F(D)\).

Finally, every exposed \((d-1)\)-dimensional face separates a site
\(j\in F(C)\) from a nearest neighbor \(k\in E(C)\).
If \(j\notin C\), that edge would connect \(j\) to the infinite component, a contradiction.
Hence \(j\in C\) and \(k\in\Gamma(C)\).
Every boundary point is a limit of exposed faces. Finiteness gives a
subsequence in one exterior closed cube \(Q_k\), proving (5).
\end{proof}

\begin{proposition}[Crude Harnack inequality]
\label{prop:reg:crude-harnack}
Assume H1--H5, and let \(\mathcal G(R,L)\) be the event that every bad cluster intersecting \(B_{3R}\) has size at most \(L\).
Then
\[
 \mathbb P(\mathcal G(R,L)^c)\le CR^de^{-cL}.
\]
There exist constants \(C_0,C_1\), depending only on \(d,\kappa\), such that, on this event, simultaneously for all
\[
 z\in B_R,\qquad C_0(L+1)\le r\le R
\]
and all nonnegative functions \(u\) that are \(A\)-harmonic in \(B_{2r}(z)\),
\begin{equation}
 \sup_{B_r(z)}u\le
       \exp\{C_1(r+1)(L+1)\}\inf_{B_r(z)}u.
 \label{reg:eq:crude-harnack}
\end{equation}
\end{proposition}

\begin{proof}
The probability estimate follows from \eqref{reg:eq:cluster-tail} by a
union bound over the \(CR^d\) lattice anchors in \(B_{3R+O(1)}\).

Fix \(x,y\in B_r(z)\), and choose closed unit cubes \(Q_i,Q_j\) containing them.
The cubes crossed by \([x,y]\), with simultaneous face crossings ordered
coordinate by coordinate, give a nearest-neighbor path \(P\) from \(i\) to \(j\) with
\[
 |P|\le C_d(r+1),\qquad P\subset B_{r+C_d}(z).
\]
Take the maximal filled sets of the bad clusters meeting \(P\).
They are disjoint, cover its bad sites, and have diameter at most \(C_dL\).

If \(i\in F(C)\), then \(x\in\overline{\mathcal O_C}\).
The containment \(\overline{\mathcal O_C}\Subset B_{2r}(z)\), verified below,
allows the maximum principle on its components; no boundary regularity is needed.
Choose \(x'\in\partial\mathcal O_C\) with \(u(x)\le u(x')\),
and then use Lemma~\ref{lem:reg:filled-clusters} to choose \(i'\in\Gamma(C)\) such that \(x'\in Q_{i'}\).
Otherwise set \(x'=x,i'=i\), which is a good site.
Similarly choose \(y'\in Q_{j'}\) with \(u(y')\le u(y)\) and \(j'\) good.
Use the original cube index to make this choice also for endpoints on cube boundaries.

Follow \(P\). At the first entrance into any maximal \(F(C)\),
delete the portion up to the last exit from that set and replace it by a simple nearest-neighbor path in \(\Gamma(C)\) joining the entrance and exit.
The exterior entrance and exit sites belong to \(\Gamma(C)\); use \(i',j'\)
when an endpoint lies in the filled set.
By Lemma~\ref{lem:reg:filled-clusters}, each detour avoids all other maximal
filled sets. Each replacement removes an untreated original site, so at most
\(|P|+1\) replacements and two endpoint connections produce a good-site chain of length
\[
 N'\le C_d(r+1)(L+1).
\]
If the two endpoints belong to the same filled set, connect them directly along its outer boundary.

Every detour and moved endpoint lies in \(B_{r+C_d(L+1)}(z)\).
On a good site, \(\kappa I\le A\le I\) in \(\widehat Q_k\);
local Harnack compares values in \(k+[-3/2,3/2]^d\) by a fixed factor
\(H_\kappa\). This contains \(Q_k\) and the neighboring cube centers.
Comparison along the chain gives
\[
 u(x)\le u(x')\le H_\kappa^{N'+2}u(y')\le H_\kappa^{N'+2}u(y).
\]
All the sets \(\widehat Q_k\) and closures of filled regions used lie in
\(B_{r+C_d(L+1)+3\sqrt d}(z)\).
For sufficiently large \(C_0\), this ball is compactly contained in
\(B_{2r}(z)\). Also \(z\in B_R\) and \(r\le R\) place the original
bad sites in \(B_{3R}\). Thus \(\mathcal G(R,L)\) controls the whole
construction. Taking the supremum over \(x\) and infimum over \(y\) proves
\eqref{reg:eq:crude-harnack}, including when \(u\) vanishes.
\end{proof}

\begin{remark}
We use \(L=R^{b_{\rm geom}}\), rather than logarithmic \(L\), to retain
a stretched-exponential failure bound. At a terminal scale \(r\asymp R^a\),
the resulting Harnack factor is \(\exp(CR^{a+b_{\rm geom}})\),
with \(0<b_{\rm geom}<a\).
\end{remark}

\subsection{Multiscale coupling and oscillation decay}
\label{reg:subsec:multiscale}

We now turn finite exit comparisons into an oscillation estimate uniform over
all harmonic functions. Successful couplings reduce the separation scale;
a biased random walk controls unsuccessful steps. A second scale hierarchy
ensures that the coupling reaches a good base ball despite sparse bad candidates.
This multiscale coupling follows the strategy of
Berger--Cohen--Deuschel--Guo~\cite[Section~4]{BCDG}, developed there for
nearest-neighbor balanced walks in an iid environment. We implement it
for continuous stopped paths: the local common-mass estimate replaces
the terminal discrete coupling, and we specify the coefficient region on
which each event depends before invoking finite-range independence. The auxiliary biased-walk
regeneration estimate is stated separately below.

Fix \(\Lambda=\mathbb Z^d\), translating it with the observation center.
We use this lattice to index the candidate balls and to choose their centers.
The diffusion paths themselves remain continuous.
Choose a sufficiently large integer \(M>100\), and then a finite partition \(\mathcal P\) of the sphere
whose pieces have diameter at most \(1/(100M^2)\) and boundaries of zero surface measure.
For a base scale \(b=M^J\ge100\sqrt d\), use only radii in
\(\mathcal R_b=bM^{\mathbb N_0}\).
A ball \((c,r)\), with \(c\in\Lambda,r>b\), is called good if, for every starting point
\(x\in\overline B_r(c)\), the distribution of the exit-piece label from \(B_{Mr}(c)\)
has total variation distance at most the fixed \(\epsilon\) from its constant-coefficient counterpart for \(\bar A\).
A base-scale ball \((c,b)\) is good if the local uniform ellipticity event in Lemma~\ref{lem:reg:base-coupling} holds.

Lemma~\ref{lem:reg:exit-partitions} gives, after fixing \(M,\mathcal P,\epsilon\),
an exponent \(\delta\in(0,1]\) and a starting scale such that
\begin{equation}
 \mathbb P((c,r)\text{ is bad})\le e^{-r^\delta},\qquad r>b,
 \label{reg:eq:good-ball-tail}
\end{equation}
and the good event is determined by the coefficients in \(B_{C_Mr}(c)\).
We have decreased \(\delta\) and increased the threshold to absorb constants.
At the fixed scale \(b\), choose \(\kappa_b\) in
Lemma~\ref{lem:reg:base-coupling} so that the bad probability is at most
\(e^{-2b^\delta}\). Denote the resulting deterministic common-exit mass by \(c_0>0\).

\begin{lemma}[Coupling driven by independent signs]
\label{lem:reg:coupling-driver}
The parameters \(M,\epsilon\) may be chosen so that, for every good ball \((c,r)\), \(r>b\),
and every pair of starting points \(y,z\in\overline B_r(c)\), there is a coupling of the complete stopped paths with the following properties:
\begin{enumerate}
 \item Both coordinate marginals are the original diffusion stopped on exiting \(B_{Mr}(c)\).
 \item With probability at least \(3/4\), the two exits belong to the same piece.
       On this success event, both exits lie within distance \(r/M\) of a new lattice center \(c'\).
 \item The same construction can use iid signs \(\epsilon_n\in\{-1,1\}\),
       with \(\mathbf P(\epsilon_n=-1)=2/3\),
       such that the radius becomes \(r/M\) on success and \(Mr\) on failure, and, before the first bad ball or the first visit to the base scale,
       \begin{equation}
        \log_M(r_{n+1}/r_n)\le\epsilon_n.
        \label{reg:eq:driver-domination}
       \end{equation}
\end{enumerate}
Concatenate the segments, continue with the original process after a bad ball is encountered,
and continue after the base-scale coupling. Each unconditional marginal of the resulting pair of complete paths is still the original diffusion.
\end{lemma}

\begin{proof}
Maximally couple the piece labels, then sample the corresponding conditional
exit and stopped-path laws.
Lemma~\ref{lem:reg:exit-partitions} and the constant-coefficient interior gradient estimate give
\[
 \mathbf P(\text{same piece})\ge1-C_{\bar A}/M-2\epsilon.
\]
Choose \(M\) large and then \(\epsilon\) small so that the right-hand side is at least \(3/4\).
On success, the two exits are at distance at most
\(Mr/(100M^2)=r/(100M)\).
Round the first exit to a nearest lattice point \(c'\), using a fixed rule
for ties. Since the rounding error is at most \(\sqrt d\) and
\(r/M\ge b\ge100\sqrt d\), both exits lie in \(\overline B_{r/M}(c')\).
On failure, retain \(c\) and use inner radius \(Mr\). The exits then lie
strictly inside the next outer ball \(B_{M^2r}(c)\).
Each center displacement is at most \((M+1)r\), and each complete segment
stays in its selected closed outer ball.

Write \(P_s\) for the current basic joint segment law, \(D\) for its success event,
and \(q_s=P_s(D)\ge3/4\). Sample iid signs in advance.
Conditional on the current sign being down, use
\[
 P_s^-=P_s(\,\cdot\mid D),
\]
and conditional on it being up, use
\begin{equation}
 P_s^+=\frac{P_s-(2/3)P_s^-}{1/3}.
 \label{reg:eq:conditional-driver-kernel}
\end{equation}
Since \((2/3)/q_s\le1\), this is a probability kernel, and
\((2/3)P_s^-+(1/3)P_s^+=P_s\)
recovers the basic law. A down sign forces success, proving
\eqref{reg:eq:driver-domination}. Both conditional kernels are dominated by
\(3P_s\) and are jointly measurable in all parameters.

To verify the marginal laws, reveal the completed segments and previous signs
at the start of each new segment, leaving its sign independent and unrevealed.
The mixture identity then gives each coordinate its original stopped-diffusion
law. During that coordinate's segment reveal only its own path, revealing the
other segment and the current sign at the end. Concatenation preserves its
martingale problem in this filtration. This argument is coordinatewise; it
does not give a common continuous-time martingale filtration for the pair.
Formally, truncate the number of segments, verify the martingale identity
for compactly supported smooth test functions, and remove the truncation on
fixed time intervals by boundedness of the coefficients and dominated convergence.

By \eqref{reg:eq:driver-domination}, the number of segments before the first bad ball or the first visit to the base scale
is bounded by the hitting time of the corresponding lower level for a downward-biased simple random walk.
This number is almost surely finite, so concatenation cannot accumulate.
After a bad ball or a base-scale exit, continue by the original diffusion.
Uniqueness of the martingale problem identifies the complete unconditional marginals.
\end{proof}

\begin{lemma}[Regeneration count]
\label{lem:reg:regeneration-count}
Let \(\xi_0=0\), and let \(\xi_{n+1}-\xi_n\) be iid signs,
with down probability \(2/3\) and up probability \(1/3\).
For sufficiently large \(M\), call \(n\) a regeneration time if
\begin{equation}
 \xi_n>\xi_{n+1}>\xi_{n+2},\qquad
 \sum_{j\ge n+2}M^{\xi_j-\xi_{n+1}}<1.
 \label{reg:eq:regeneration}
\end{equation}
There are deterministic constants \(c_2,c_3,C>0\) such that
\[
 \mathbf P\big(\#\{0\le n<N:n\text{ is a regeneration time}\}<c_2N\big)
 \le Ce^{-c_3N}.
\]
Moreover, \(\mathbf P(0\text{ is a regeneration time})=:p_*>0\).
\end{lemma}

\begin{proof}
We use the descending-record renewal decomposition in
\cite[Claim~4.11 and Appendix~D]{BCDG} (final submitted manuscript).
A renewal is an integer \(n\ge0\) satisfying both
\[
 \xi_n<\min_{0\le m<n}\xi_m,
 \qquad
 \sup_{m>n+1}\xi_m<\xi_{n+1}<\xi_n,
\]
where the first condition is vacuous for \(n=0\). Let
\(\tau_1<\tau_2<\cdots\) enumerate all such integers.
For this selection rule, Appendix~D of \cite{BCDG} gives that the spacings
\(V_k=\tau_{k+1}-\tau_k\), \(k\ge1\), are iid, while both
\(\tau_1\) and \(V_1\) have exponential tails. The past-record condition
is part of the selection rule; the future inequality alone is not the
renewal decomposition used here. We give the counting argument, including
the effect of spacings beyond the observation window. Heights decrease by at least one per renewal,
and after \(\tau_k+2\) the walk stays below \(\xi_{\tau_k+1}\). Hence
\begin{equation}
 \sum_{m\ge\tau_k+2}M^{\xi_m-\xi_{\tau_k+1}}
 \le V_k/M+\sum_{i\ge1}V_{k+i}M^{-i}.
 \label{reg:eq:renewal-sum}
\end{equation}
If \(V_k<M/4\), and if \(V_{k+i}<(M/2)^i/4\) for every \(i\ge1\),
the right-hand side is strictly less than \(1/2\),
so \(\tau_k\) satisfies \eqref{reg:eq:regeneration}.

The number of candidates spoiled by a spacing \(V_j\) is at most
\[
 I_j=\mathbf1_{\{V_j\ge M/4\}}
       +\sum_{i\ge1}\mathbf1_{\{V_j\ge(M/2)^i/4\}},
\]
The \(I_j\) are iid. For \(M\ge8\),
\(I_j\le C\log(2+V_j)\), and \(I_j\to0\) as \(M\to\infty\).
The exponential tail permits \(M\) large enough that
\(\mathbf E e^{I_j}\le e^{1/8}\).
Exponential Markov's inequality then gives
\[
 \mathbf P\left(\sum_{j=1}^JI_j>J/4\right)\le e^{-J/8}.
\]
Restrict to \(k\le\lfloor3J/4\rfloor\). A spacing with \(j>J\) can
spoil such a candidate only if
\(V_j\ge(M/2)^{j-\lfloor3J/4\rfloor}/4\).
The probability of any such future effect is at most
\[
 \sum_{j>J}C\exp\{-c(M/2)^{j-\lfloor3J/4\rfloor}\}
 \le C\exp\{-c(M/2)^{J/4}\}.
\]
Thus, with probability at least \(1-Ce^{-cJ}\), at least \(J/2-2\)
of the first \(\lfloor3J/4\rfloor\) renewals are genuine regenerations.
For \(J=\lfloor cN\rfloor\) with \(c>0\) small, exponential moments give
\(\mathbf P(\tau_J>N)\le Ce^{-c'N}\), proving the asserted count.

Finally, negative drift implies \(\sum_{j\ge0}M^{\xi_j}<\infty\) almost surely.
Choose \(L_0<\infty\) such that this sum is at most \(L_0\) with positive
probability. Require first a sufficiently long string of down steps, and then
the independent remaining sum to be at most \(L_0\). The latter contribution
is multiplied by \(M^{-n}\), so these positive-probability requirements imply
\eqref{reg:eq:regeneration} at zero.
\end{proof}

\begin{lemma}[Descent and landing probabilities]
\label{lem:reg:descent}
Run the algorithm of Lemma~\ref{lem:reg:coupling-driver}.
Let \(U\) be the index of the first segment with a bad ball,
and let \(T_s\) be the first segment index at which the inner radius is at most \(s\in\mathcal R_b\).
There are constants \(c_1,c_5,C>0\), depending only on the coupling parameters already fixed, such that:
\begin{enumerate}
 \item Starting from \(r\le S=s^K\), if all required candidate balls of radius at least \(s\) in the parent ball are good,
       then the probability of encountering a bad ball or of a complete path leaving \(B_{S^2}\) before descending to \(s\)
       is at most \(Cs^{-c_1K}\).
 \item Starting from \(S=s^K\), for every deterministic \(w\in\mathbb R^d\),
       \begin{equation}
        \mathbf P_{\rm path}
        (c_{T_s}\in B_s(w),\ T_s<U)\le Cs^{-c_5K}.
        \label{reg:eq:anticoncentration}
       \end{equation}
\end{enumerate}
Here \(K\) is a sufficiently large integer; fixed constants can be absorbed by increasing the base scale.
\end{lemma}

\begin{proof}
For the auxiliary walk, the probability of ascending \(j\) levels is
\(2^{-j}\), and the mean time to descend \(m\) levels is \(3m\).
By \eqref{reg:eq:driver-domination}, before \(T_s\wedge U\) the radius
reaches \(S^{3/2}\) with probability at most \(CS^{-c_M}\),
where \(c_M=(\log2)/(2\log M)>0\).
Stop at this radius. The expected number \(N_{\rm stop}\) of segments is at most
\(3\log_M(S/s)\).
The center displacements and complete segments are contained within distance
\[
 C_MS^{3/2}(1+N_{\rm stop}).
\]
Markov's inequality bounds the probability of leaving \(B_{S^2}\) by
\[
 C_MS^{-1/2}(1+\log(S/s))+CS^{-c_M}
 \le CS^{-c_1}.
\]
Goodness of the candidate balls excludes a bad ball before escape.
Substituting \(S=s^K\) proves (1).

For (2), keep the environment fixed and put
\(N=\lfloor\log_M(S/s)\rfloor\). Since each step descends at most one
level, \(T_s\ge N\). At a regeneration \(n<T_s\), domination gives
\(r_j\le r_nM^{\xi_j-\xi_n}\); hence, before stopping,
\[
 |c_m-c_{n+1}|
 \le(M+1)\sum_{j=n+1}^{m-1}r_j
 \le4r_n.
\]
If \(c_{T_s}\in B_s(w)\), then \(c_{n+1}\in B_{5r_n}(w)\),
and the first exit of segment \(n\) must lie in
\[
 B_{6r_n}(w)\cap\partial B_{Mr_n}(c_n).
\]
Cover this cap by intersecting partition pieces, increasing its angular
radius by at most \(1/(100M^2)\). For \(d\ge2\) its relative surface
area is at most \(C/M\). The bounded constant-coefficient Poisson kernel
and finite-partition comparison bound its basic-law probability by
\(C_{\bar A}/M+\epsilon\). Choose \(M,\epsilon\) so that
\[
 \eta_0:=3(C_{\bar A}/M+\epsilon)<e^{-2}.
\]

Condition now on the entire sign sequence \(\xi\), which fixes the
regeneration indices. Given the preceding segments, the current conditional
kernel is dominated by \(3P_s\); use fresh independent sampling randomness
for each segment. Each cap event therefore has conditional
probability at most \(\eta_0\). Declare it false if the algorithm has already
stopped or met a bad ball. Successive conditioning at the regeneration indices
among the first \(N\) steps yields
\[
 \mathbf P_{\rm path}(c_{T_s}\in B_s(w),T_s<U\mid\xi)
 \le\eta_0^{\#\{0\le n<N:n\text{ is a regeneration time}\}}.
\]
Integrating and applying Lemma~\ref{lem:reg:regeneration-count} gives
\[
 \mathbf P_{\rm path}(c_{T_s}\in B_s(w),T_s<U)
 \le\eta_0^{c_2N}+Ce^{-c_3N}
 \le C(S/s)^{-c_4}\le Cs^{-c_5K}.
\]
The last step uses \(K\ge2\), decreasing \(c_5\) to absorb integer parts.
Conditioning on \(\xi\) changes the diffusion marginals: this argument uses
kernel domination, not optional stopping under the conditional law.
\end{proof}

\begin{definition}[Admissible parent balls]
\label{def:reg:admissible}
Choose a sufficiently large integer \(K\), and set
\[
 R_0=b=M^J,\qquad R_{k+1}=R_k^K,\qquad
 \sigma=\frac{\delta}{4K},\qquad \eta=\frac{2d+1}{K}.
\]
All these \(R_k\) belong to \(\mathcal R_b\).
For \(c\in\Lambda\), define the event \(A_k(c)\) by the following conditions:
\begin{enumerate}
 \item If \(k=0\), every base ball \((q,b)\) with \(q\in\Lambda\cap B_{2b^2}(c)\) is good.
 \item If \(k\ge1\), every ball \((q,r)\) with \(q\in\Lambda\cap B_{2R_k^2}(c)\) and
       \(r\in\mathcal R_b\cap[R_{k-1},2R_k^2]\) is good;
       moreover, among the same set of centers, at most \(R_k^\eta\) fail to satisfy \(A_{k-1}(q)\).
\end{enumerate}
\end{definition}

\begin{lemma}[Locality and probability of admissibility]
\label{lem:reg:admissible-probability}
With the parameters above, fix \(D_0>4+4C_M\) and then choose \(b\) sufficiently large.
For every \(k\ge0\) and every deterministic \(c\in\Lambda\),
\[
 A_k(c)\in\sigma(A|_{B_{D_0R_k^2}(c)}),\qquad
 \mathbb P(A_k(c)^c)\le e^{-R_k^\sigma}.
\]
\end{lemma}

\begin{proof}
The first family has at most \(CR_k^{2d}\log R_k\le CR_k^{2d+1}\)
candidates and depends only on the coefficients in
\(B_{(2+2C_M)R_k^2}(c)\). By induction, the second family depends only on
the coefficients in \(B_{2R_k^2+D_0R_{k-1}^2}(c)\). For large \(b\),
both determining regions lie in
\(B_{D_0R_k^2}(c)\).

For \(k=0\), the probability of a bad base ball is at most \(e^{-2b^\delta}\).
A union bound over \(Cb^{2d}\) centers gives \(e^{-b^\sigma}\).
Suppose the conclusion holds at level \(k-1\).
By \eqref{reg:eq:good-ball-tail} and the base-scale bound, the probability of a large bad ball is at most
\[
 CR_k^{2d+1}e^{-R_{k-1}^\delta}
 \le\tfrac12 e^{-R_k^\sigma},
\]
since \(R_{k-1}^\delta=R_k^{\delta/K}\) and \(\delta/K=4\sigma\).

The dependence radius of \(A_{k-1}(q)^c\) is \(D_0R_{k-1}^2\).
Use at most \(CR_{k-1}^{2d}=CR_k^{2d/K}\) deterministic residue classes
whose dependence balls have pairwise distance greater than \(\ell\) in each class.
H3, applied successively
to a ball and the union of the preceding balls, gives joint independence there.
If more than \(R_k^\eta\) indices are bad, one class contains at least
\(cR_k^{1/K}\) bad indices.
Each color has sample size \(N_k\le CR_k^{2d}\), and each individual bad probability is at most
\(p_k\le e^{-R_k^{\sigma/K}}\).
Let \(m_k=\lceil cR_k^{1/K}\rceil\). For the number \(Y\) of bad indices in that color,
\[
 \mathbb P(Y\ge m_k)
 \le \binom{N_k}{m_k}p_k^{m_k}
 \le(eN_kp_k/m_k)^{m_k}
 \le e^{-cR_k^{(1+\sigma)/K}}.
\]
Only independence and the common upper bound on probabilities are used.
Since \((1+\sigma)/K>\sigma\), a union bound over classes gives at most
\(\tfrac12e^{-R_k^\sigma}\) for large \(b\), completing the induction.
\end{proof}

\begin{lemma}[Descent between admissible levels]
\label{lem:reg:admissible-transition}
The integer \(K\) may first be chosen sufficiently large that there exists \(s_0>2\) with the following property.
If \(A_{k+1}(c)\) holds, start the coupling at inner radius \(R_{k+1}\),
with two starting points in \(\overline B_{R_{k+1}}(c)\).
The total path probability of encountering a bad ball before descending to \(R_k\),
leaving the parent domain, or landing at a center \(q\) that does not satisfy \(A_k(q)\),
is at most \(CR_k^{-s_0}\).
\end{lemma}

\begin{proof}
While the radius is below \(R_{k+1}^{3/2}\) and the complete paths stay in
\(B_{R_{k+1}^2}(c)\), all balls used belong to the deterministic candidate
family. Lemma~\ref{lem:reg:descent} bounds escape from this family by
\(CR_k^{-c_1K}\). Admissibility bounds the number of bad landing centers by
\[
 R_{k+1}^\eta=R_k^{2d+1}.
\]
For each such center \(w\), \eqref{reg:eq:anticoncentration} bounds landing
at \(w\), a subevent of landing in \(B_{R_k}(w)\), by \(CR_k^{-c_5K}\).
A union bound gives
\[
 CR_k^{2d+1-c_5K}.
\]
Choose \(K\) first so that
\(\min(c_1,c_5)K>4d+10\).
Both errors can then be bounded by \(CR_k^{-s_0}\) with \(s_0>2\).
\end{proof}

\begin{proposition}[Large-scale oscillation decay]
\label{prop:reg:oscillation}
Assume H1--H5.
There exist deterministic constants
\(\Psi>1\), \(\upsilon\in(0,1)\), \(\gamma_{\rm osc}>0\), and \(C,c,R_0<\infty\)
such that, for every fixed \(z\) and every \(R\ge R_0\), there is a coefficient event \(\mathcal O_R(z)\) satisfying
\[
 \mathbb P(\mathcal O_R(z)^c)\le Ce^{-cR^{\gamma_{\rm osc}}},
\]
on which, simultaneously for all functions that are \(A\)-harmonic in \(B_{\Psi R}(z)\),
\begin{equation}
 \operatorname{osc}_{B_R(z)}u
 \le\upsilon\operatorname{osc}_{B_{\Psi R}(z)}u.
 \label{reg:eq:oscillation-contraction}
\end{equation}
\end{proposition}

\begin{proof}
First take \(z=0\) and \(R\in\mathcal R_b\), and choose \(k\) so that
\(R_k<R\le R_{k+1}\). Set
\[
 \Omega_R=A_{k+1}(0)\cap
          \bigcap_{q\in\Lambda\cap B_{2R_{k+1}^2}}A_k(q).
\]
Lemma~\ref{lem:reg:admissible-probability} and a finite union bound give
\begin{equation}
 \mathbb P(\Omega_R^c)
 \le Ce^{-cR_k^\sigma}
 \le Ce^{-cR^{\sigma/K}}.
 \label{reg:eq:omega-probability}
\end{equation}
The polynomial number of centers is absorbed by the exponential.

Start the coupling from \(y,z\in\overline B_R\). During descent to
\(R_k\), Lemma~\ref{lem:reg:descent} controls escape, and \(\Omega_R\)
ensures admissibility at the landing center. Apply
Lemma~\ref{lem:reg:admissible-transition} at each subsequent level to obtain
\[
 \mathbf P_{\rm path}(\text{descends to }b\text{ without encountering a bad ball})
 \ge1-C\sum_{j\ge0}R_j^{-s_0}
 =1-\epsilon_*.
\]
After \(M,\epsilon,K\) are fixed, increase \(b=M^J\) so that
\(\epsilon_*<p_*/2\), where \(p_*\) is independent of \(b\).
By the union bound, with probability at least \(p_*/2\), descent succeeds
and time zero is a regeneration. Independence of these events is unnecessary.

On the event that time zero is a regeneration time,
\eqref{reg:eq:driver-domination} and \eqref{reg:eq:regeneration} give
\[
 \sum_{n<T_b}(M+1)r_n\le C_MR.
\]
Thus all complete segments, including the final base ball of radius
\(Mb\le MR\), lie in \(B_{\Psi_0R}\) for a deterministic \(\Psi_0\).
At that base ball use fresh randomness, independent of all signs and preceding
segments, for Lemma~\ref{lem:reg:base-coupling}. Its conditional common-exit
probability is at least \(c_0\), even on the future-dependent regeneration
event. Coincident exits therefore occur with probability at least
\(p_*c_0/2=:\varrho>0\). From there use one shared original diffusion.
Both coordinates exit \(B_{\Psi_0R}\) together with probability at least \(\varrho\).

The unconditional marginals are the original diffusions by
Lemma~\ref{lem:reg:coupling-driver}. For \(u\) harmonic in \(B_{2\Psi_0R}\),
the exit representation in the compactly contained ball \(B_{\Psi_0R}\) gives
\[
 |u(y)-u(z)|\le(1-\varrho)\operatorname{osc}_{B_{\Psi_0R}}u.
\]
Take the supremum over \(y,z\). The coefficient event and coupling bounds
are independent of \(u\), giving the simultaneous assertion.

For general \(R\), round up to \(\mathcal R_b\), multiplying the radius by
at most \(M\). Set \(\Psi=2M\Psi_0\), \(\upsilon=1-\varrho\), and
\(\gamma_{\rm osc}=\sigma/K\). Translate \(\Lambda\) for a fixed nonzero center.
\end{proof}

\subsection{The large-scale Harnack inequality}
\label{reg:subsec:strong-harnack}

Oscillation decay is additive, whereas Harnack compares ratios. Relative
exit probabilities bridge the two: a large interior ratio forces a large
ratio on a small exit piece. Repeated oscillation decay then doubles that
ratio at the next scale. The terminal estimate rules out continued growth.

\begin{theorem}[Large-scale Harnack inequality for positive solutions]
\label{thm:reg:harnack}
Assume H1--H5.
There exist deterministic constants \(H,C,c,R_H<\infty\) and \(\gamma_{\rm Har}>0\)
such that, for every fixed center \(z\) and every \(R\ge R_H\), there is a measurable coefficient event
\(\mathcal H_R(z)\) satisfying
\begin{equation}
 \mathbb P(\mathcal H_R(z)^c)\le Ce^{-cR^{\gamma_{\rm Har}}}.
 \label{reg:eq:harnack-probability}
\end{equation}
On this event, simultaneously for all nonnegative functions that are \(A\)-harmonic in \(B_{4R}(z)\),
\begin{equation}
 \sup_{B_R(z)}u\le H\inf_{B_R(z)}u.
 \label{reg:eq:harnack}
\end{equation}
The same conclusion holds for \(B\)-harmonic functions.
\end{theorem}

\begin{proof}
It suffices to prove the assertion at the origin.
Take the partition-independent \(H_0=4H_{\rm bar}\) from
Lemma~\ref{lem:reg:relative-exit}, and choose \(N\) so that
\(\upsilon^{-N}\ge4H_0\).
Choose \(\epsilon_0>0\) so that the center displacements after the first
step sum to at most \(R/4\).
Set
\[
 r_0=R,\qquad r_j=\lfloor\epsilon_0R/j^2\rfloor,\quad j\ge1.
\]
Set \(a=1/8\), \(b_{\rm geom}=1/16\), and let \(k\) be the largest
index with \(r_k\ge C_0R^a\), for \(C_0\) to be fixed. For large \(R\),
\begin{equation}
 r_{j+1}/r_j\ge c_*\quad(0\le j<k),\qquad
 k\ge cR^{(1-a)/2},\qquad C_0R^a\le r_k\le CR^a.
 \label{reg:eq:growth-radii}
\end{equation}
Indeed, \((j/(j+1))^2\ge1/4\), the first-step ratio is at least
\(\epsilon_0/2\), and \(r_j=\epsilon_0R/j^2+O(1)\).

Choose \(\sigma_0>0\) with \(\Psi^N\sigma_0<c_*/2\).
Fix a second partition with positive-measure pieces, surface-null boundaries,
and diameters less than \(\sigma_0/4\). Rounding scaled representatives
to \(\Lambda\) places each exit piece at scale \(2r_j\) in a ball
\(\overline B_{\sigma_0r_j}\) with lattice center, once the smallest radius
absorbs \(\sqrt d\). This partition is used only for relative exit comparison;
the earlier \(b,c_0,\upsilon\) remain fixed. Take \(C_0\) with
\(\sigma_0C_0\ge1\).

Define \(\mathcal H_R(0)\) as the intersection of the following finite families of events:
\begin{enumerate}
 \item For every \(c\in\Lambda\cap B_{4R}\) and
       \(r\in\{r_0,\ldots,r_k\}\), \eqref{reg:eq:relative-exit} holds.
 \item For every such center, \(0\le j<k\), and \(0\le l<N\),
       \eqref{reg:eq:oscillation-contraction} holds at radius \(\sigma_0\Psi^lr_j\).
 \item The event \(\mathcal G(4R,R^{b_{\rm geom}})\) of Proposition~\ref{prop:reg:crude-harnack} holds.
\end{enumerate}
These deterministic families contain at most \(CR^{d+1}\) candidates,
and the smallest oscillation radius is at least \(R^a\).
A union bound, including \(CR^de^{-cR^{b_{\rm geom}}}\) for the third family,
gives \eqref{reg:eq:harnack-probability} with
\[
 \gamma_{\rm Har}
 =\min\{b_{\rm geom},a\delta_{\rm rel},a\gamma_{\rm osc}\}>0,
\]
where \(\delta_{\rm rel}\) is the exponent for the second partition;
increase the starting scale to absorb polynomial factors.

Fix an environment in this event. The local strong maximum principle reduces
the proof to \(u>0\).
Suppose the assertion fails for \(H=2H_0\), and choose
\(x_0,y_0\in\overline B_R\) such that
\[
 T_0=\frac{u(x_0)}{u(y_0)}>2H_0.
\]
Write \(c_j\) for the iteration centers, starting with \(c_0=0\).
Suppose recursively that \(x_j,y_j\in\overline B_{r_j}(c_j)\) and
\(T_j=u(x_j)/u(y_j)>2H_0\).
Let \(\mu_x\) be the exit measure from \(B_{2r_j}(c_j)\).
If every exit piece \(P\) satisfied
\[
 \sup_Pu\le(T_j/H_0)\inf_Pu,
\]
then the exit representation and \eqref{reg:eq:relative-exit} would give
\begin{align*}
 u(x_j)
 &\le\sum_P\mu_{x_j}(P)\sup_Pu\\
 &\le\frac{T_j}{2}\sum_P\mu_{y_j}(P)\inf_Pu
 \le\frac{T_j}{2}u(y_j)=\frac12u(x_j),
\end{align*}
a contradiction.
Hence some piece satisfies \(\sup_Pu/\inf_Pu>T_j/H_0\).
Choose a ball \(\overline B_{\sigma_0r_j}(c_{j+1})\), centered at a lattice point, that contains this piece.

Apply oscillation contraction \(N\) times at the same center.
Since \(\Psi^N\sigma_0r_j\le r_{j+1}\), monotonicity of oscillation under inclusion gives
\[
 \operatorname{osc}_{\overline B_{r_{j+1}}(c_{j+1})}u
 \ge4H_0\operatorname{osc}_{\overline B_{\sigma_0r_j}(c_{j+1})}u.
\]
Let \(m,M'\) be the minimum and maximum on the smaller ball, and \(a',b'\) those on the larger ball.
The balls are compactly contained in the solution domain, as checked below.
Their extrema are attained and their minima positive, with
\(M'/m>T_j/H_0\) and \(a'\le m\).
Taking \(x_{j+1},y_{j+1}\) at the larger-ball extrema gives
\begin{equation}
 T_{j+1}
 =1+\frac{b'-a'}{a'}
 \ge1+4H_0\frac{M'-m}{m}
 >1+4H_0(T_j/H_0-1)\ge2T_j.
 \label{reg:eq:ratio-growth}
\end{equation}
This preserves \(T_{j+1}>2H_0\).

For containment, \(|c_1|\le2R+\sqrt d\) and subsequently
\(|c_{j+1}-c_j|\le2r_j+\sqrt d\le3r_j\). Since
\[
 \sum_{j\ge1}r_j\le\epsilon_0R\sum_{j\ge1}j^{-2},
\]
the choice of \(\epsilon_0\) puts every center in \(B_{3R}\).
The first exit ball is \(B_{2R}\); all subsequent exit and oscillation balls,
including \(B_{2r_k}(c_k)\), are compactly contained in \(B_{4R}\).

By \eqref{reg:eq:ratio-growth} and \eqref{reg:eq:growth-radii},
\[
 \log T_k\ge ck\ge cR^{(1-a)/2}=cR^{7/16}.
\]
On the other hand, \(r_k\asymp R^a\), and \(b_{\rm geom}<a\) ensures, for sufficiently large \(R\), that
\(r_k\ge C_0(R^{b_{\rm geom}}+1)\).
On the common geometric event, the terminal crude Harnack inequality gives
\[
 \log T_k\le C(r_k+1)(R^{b_{\rm geom}}+1)
 \le CR^{a+b_{\rm geom}}=CR^{3/16},
\]
contradicting the preceding bound.
This proves \eqref{reg:eq:harnack} on the same event for every solution.
Translation and equality of the \(A\)- and \(B\)-harmonic classes finish the proof.
\end{proof}

\begin{corollary}[Harnack inequality on nested cubes]
\label{cor:reg:cube-harnack}
Assume H1--H5, and let
\[
 Q_m=(-3^m/2,3^m/2)^d.
\]
There exist deterministic constants \(H_{\rm cube}<\infty\) and \(m_H\)
such that, for every fixed translation \(z\) and every \(m\ge m_H\),
there is a coefficient event of probability at least \(1/2\) on which, simultaneously for all nonnegative functions that are \(B\)-harmonic in \(z+Q_{m+1}\),
\[
 \sup_{z+Q_m}u\le H_{\rm cube}\inf_{z+Q_m}u.
\]
\end{corollary}

\begin{proof}
Cover \(\overline Q_m\) by balls of radius \(r_m=\rho_d3^m\), for small
\(\rho_d>0\), with centers on a grid of spacing \(r_m/(4\sqrt d)\).
Their closed fourfold dilations lie in \(Q_{m+1}\), and their intersection
graph is connected. The number of balls and maximal chain length depend only on \(d\).
Apply Theorem~\ref{thm:reg:harnack} at these deterministic centers.
The probability that the common event fails is at most \(C_de^{-cr_m^{\gamma_{\rm Har}}}\).
Choose \(m_H\) large enough that this is at most \(1/2\).
On this event, comparison along the fixed covering chains gives
\(H_{\rm cube}=H^{N_d}\), using the equation only in \(Q_{m+1}\).
Translation preserves the probability.
\end{proof}

\begin{remark}[Dependencies]
The parameter order is \(M\), the first partition and \(\epsilon\), then
\(c_1,c_5,K,\sigma,\eta\), followed by \(b=M^J\) and \(\kappa_b\).
Only then are \(c_0\) and \(\upsilon\) fixed. For Harnack, choose
\(H_0,N,\epsilon_0,\sigma_0\), then the second partition and its error
tolerance, and finally the starting scale. This second partition does not
alter the earlier coupling parameters. The argument uses fixed-tolerance
homogenization, independently of the algebraic cell estimate
in Theorem~\ref{thm:cell}.
\end{remark}

\section{Quantitative estimates for directional sources}
\label{sec:renormalization}

We consider the trace-normalized operator
\begin{equation}
 a:=\operatorname{tr}A,\qquad B:=a^{-1}A,\qquad \operatorname{tr}B=1.
 \label{ren:normalization}
\end{equation}
The critical ellipticity moment gives $(\det B)^{-1}\in L^{d/(d-1)}$.
We use this gain to adapt the monotone-quantity argument of
Armstrong--Smart~\cite[Sections~2--5]{ASquant}, originally proved under
uniform ellipticity:
an $L^q$ concentration estimate replaces the second-moment estimate, while
an integral bound on the degenerate region and large-scale Harnack replace
the uses of uniform ellipticity in the geometric contraction.

Throughout, constants may depend on the fixed coefficient law, $d$, and
the common dependence range of the sources. All probability estimates
are for a source fixed in advance; their constants and starting scales
are uniform over the source class.

\subsection{The source class and the main results}

Fix $r_f<\infty$. Our source class consists of fields $f$ such that
$(A,f)$ is jointly stationary, $|f|\le1$, $f(x)$ is measurable with respect
to $A|_{B_{r_f}(x)}$, and $f$ is locally H\"older continuous in each
realization. The H\"older exponent and constants may depend on the source
and realization. The joint field $(B,f)$ has dependence range
\begin{equation}
 \ell_f:=\ell+2r_f.
 \label{ren:joint-range}
\end{equation}
The contraction argument also applies to any jointly stationary $(B,f)$
with this common dependence range. The centering constant is the mean of
the source with respect to the invariant probability on the corresponding
environment space, provided that this probability is ergodic and equivalent
to the coefficient law.

Write $\pi$ for the invariant probability of the directional environment
process constructed in Lemma~\ref{down:density-normalized-law}, and set
\begin{equation}
 \langle f\rangle_\pi:=\mathbb E_\pi[f(0)],\qquad
 \overline B:=\mathbb E_\pi[B(0)].
 \label{ren:pi-average}
\end{equation}
We recall that $\pi\sim\mathbb P$.

\begin{theorem}[Bounded local directional sources]
\label{ren:bounded-source}
Assume \textup{(H1)--(H5)}. There exist $\delta_B,\gamma_B>0$ and
$C,R_B<\infty$ such that, for each fixed source $f$, deterministic center
$z$, and $R\ge R_B$, the solution of
\begin{equation}
 \begin{cases}
 -B:D^2v_{f,R,z}=f-\langle f\rangle_\pi
       &\text{in }B_R(z),\\
 v_{f,R,z}=0&\text{on }\partial B_R(z)
 \end{cases}
 \label{ren:source-cell}
\end{equation}
satisfies
\begin{equation}
 \mathbb P\!\left(
 R^{-2}\|v_{f,R,z}\|_{L^\infty(B_R(z))}>CR^{-\delta_B}
 \right)\le CR^{-\gamma_B}.
 \label{ren:source-estimate}
\end{equation}
The constants depend only on the fixed coefficient law, $d$ and $\ell_f$,
uniformly over the source class.
\end{theorem}

\begin{corollary}[Linear dependence on the source amplitude]
\label{ren:source-amplitude}
For the same source class with $|f|\le T$, $0<T<\infty$,
with the constants of Theorem~\ref{ren:bounded-source}, one has
\begin{equation}
 \mathbb P\!\left(
 R^{-2}\|v_{f,R,z}\|_\infty>CT R^{-\delta_B}
 \right)\le CR^{-\gamma_B}.
 \label{ren:amplitude-estimate}
\end{equation}
In particular, for $h_T=\min\{a^{-1},T\}$ one may take either $f=h_T$ or
$f=-h_T$; all starting scales are independent of $T$.
\end{corollary}

\begin{corollary}[Directional matrix cell problems]
\label{ren:direction-cell}
For $M\in\mathbb S^d$, let $w^B_{M,R,z}$ solve
\begin{equation}
 \begin{cases}
 -B:D^2w^B_{M,R,z}=(B-\overline B):M&\text{in }B_R(z),\\
 w^B_{M,R,z}=0&\text{on }\partial B_R(z).
 \end{cases}
 \label{ren:matrix-cell}
\end{equation}
With $|M|$ denoting the Frobenius norm,
\begin{equation}
 \mathbb P\!\left(
 R^{-2}\sup_{|M|\le1}\|w^B_{M,R,z}\|_\infty
 >CR^{-\delta_B}\right)\le CR^{-\gamma_B}.
 \label{ren:matrix-estimate}
\end{equation}
The estimate is simultaneous in $M$.
\end{corollary}

We use the following consequence of Theorem~\ref{thm:reg:harnack} and
Corollary~\ref{cor:reg:cube-harnack}. For $Q_m:=(-3^m/2,3^m/2)^d$, there
are deterministic $H<\infty$, $p_H>0$ and $m_H$ such that, whenever
$m\ge m_H$ and $z$ is deterministic, the event
\begin{equation}
 \mathcal H_m(z):=\left\{
 \begin{array}{l}
 \displaystyle\sup_{z+Q_m}h\le H\inf_{z+Q_m}h\text{ for every nonnegative}\\
 B\text{-harmonic function }h\text{ in }z+Q_{m+1}
 \end{array}\right\}
 \label{ren:harnack-event}
\end{equation}
satisfies $\mathbb P(\mathcal H_m(z))\ge p_H$.
Indeed, use finitely many chains of interior balls whose fourfold
expansions remain in the outer cube, and then a union bound in the
ball Harnack estimate. At a sufficiently large deterministic scale one
may take $p_H=1/2$. The identity of $A$- and $B$-harmonic functions transfers
this estimate to $B$. The event controls every nonnegative harmonic function,
including the environment-dependent functions selected below.

\subsection{The determinant weight and the monotone quantity}

Write
\begin{equation}
 \begin{aligned}
 \lambda(x)&:=\lambda_{\min}(A(x)),&
 b(x)&:=\lambda_{\min}(B(x)),\\
 D(x)&:=(\det B(x))^{-1},&q&:=\frac d{d-1}.
 \end{aligned}
 \label{ren:weights}
\end{equation}
Since $d\ge2$, one has $1<q\le2$.

\begin{lemma}[Higher integrability after normalization]
\label{ren:determinant-moment}
Pointwise,
\begin{equation}
 D\le d^d\lambda^{-(d-1)},\qquad
 \mathbb ED(0)^q\le d^{dq}\mathbb E\lambda_*(0)^{-d}<\infty.
 \label{ren:D-moment}
\end{equation}
If $H_\kappa:=D\mathbf1_{\{b<\kappa\}}$, then
\begin{equation}
 \mathbb EH_\kappa(0)\longrightarrow0\quad(\kappa\downarrow0).
 \label{ren:bad-weight-limit}
\end{equation}
Moreover,
\begin{equation}
 \mathbb EH_\kappa(0)
 \le C_d\kappa\,
 \mathbb E\!\left[\lambda(0)^{-d}
               \mathbf1_{\{\lambda(0)<d\kappa\}}\right]
 =o(\kappa).
 \label{ren:bad-weight-refined}
\end{equation}
\end{lemma}

\begin{proof}
Let $0<\alpha_1\le\cdots\le\alpha_d\le1$ be the eigenvalues of $A$.
Since $\alpha_1=\lambda$, $\det A\ge\alpha_d\lambda^{d-1}$ and
$a\le d\alpha_d$,
\[
 D=\frac{a^d}{\det A}
 \le d^d\alpha_d^{d-1}\lambda^{-(d-1)}
 \le d^d\lambda^{-(d-1)}.
\]
Now $(d-1)q=d$ and $\lambda(0)\ge\lambda_*(0)$ give
\eqref{ren:D-moment}. Dominated convergence gives
\eqref{ren:bad-weight-limit}. Finally, $b<\kappa$ implies
$\lambda=ab<d\kappa$, and hence
$\lambda^{-(d-1)}\le d\kappa\lambda^{-d}$ on this set.
This proves \eqref{ren:bad-weight-refined}.

\end{proof}

\begin{lemma}[$L^q$ bounds for averages of finite-range fields]
\label{ren:local-average}
Let $Z$ be a stationary real-valued field with deterministic finite range
of dependence $\ell_Z$, and suppose that $\mathbb E|Z(0)|^q<\infty$,
where $1<q\le2$. For every deterministic center $z$ and
$r\ge\ell_Z+1$,
\begin{equation}
 \mathbb E\left|
 \frac1{|B_r|}\int_{B_r(z)}Z(x)\,dx-\mathbb EZ(0)
 \right|^q
 \le C_{d,q,\ell_Z}r^{-d(q-1)}\mathbb E|Z(0)|^q.
 \label{ren:average-compression}
\end{equation}
\end{lemma}

\begin{proof}
Partition space into unit cubes $C_j$ and put
\[
 Y_j:=\int_{C_j\cap B_r(z)}(Z(x)-\mathbb EZ(0))\,dx.
\]
At most $C_dr^d$ terms are nonzero. A fixed finite coloring separates
same-color cubes by more than $\ell_Z$. Finite-range dependence, applied
to one cube and the union of the others, gives mutual independence within
each color by induction.
For independent centered $Y_j$, independent copies $Y_j'$ and independent
Rademacher signs $\varepsilon_j$, conditional Jensen gives
\begin{align*}
 \mathbb E\left|\sum_jY_j\right|^q
 &\le\mathbb E\left|\sum_j(Y_j-Y_j')\right|^q
 =\mathbb E\mathbb E_\varepsilon
       \left|\sum_j\varepsilon_j(Y_j-Y_j')\right|^q\\
 &\le\mathbb E\left(\sum_j|Y_j-Y_j'|^2\right)^{q/2}
 \le2^q\sum_j\mathbb E|Y_j|^q.
\end{align*}
Jensen and stationarity give $\mathbb E|Y_j|^q\le C_q\mathbb E|Z(0)|^q$.
Combine the finitely many colors using
$|\sum_{c=1}^Ju_c|^q\le J^{q-1}\sum_c|u_c|^q$, and divide by $|B_r|^q$.

\end{proof}

For $Z=H_\kappa$, the decay exponent in
\eqref{ren:average-compression} is $d(q-1)=d/(d-1)$.

\begin{proposition}[Probabilistic nondegeneracy of directional ball exits]
\label{ren:exit-covariance}
Let $Y$ be the diffusion with generator $B:D^2$, started at a deterministic
point $z$, and set
$\sigma=\inf\{t:Y_t\notin B_r(z)\}$ and $\Delta=Y_\sigma-z$.
In every environment,
\begin{equation}
 \mathbb E_z^\omega\sigma=r^2/2,\qquad
 \mathbb E_z^\omega\Delta=0,\qquad |\Delta|=r.
 \label{ren:exit-identities}
\end{equation}
There exists $\eta_d>0$ such that
\begin{equation}
 \frac1{|B_r|}\int_{B_r(z)}H_\kappa\le\eta_d
 \quad\Longrightarrow\quad
 \mathbb E_z^\omega[\Delta\otimes\Delta]
 \ge\frac\kappa2r^2I.
 \label{ren:exit-lower}
\end{equation}
Consequently, one can choose $\kappa>0,C<\infty$, depending on the
fixed coefficient law, such that, for all sufficiently large $r$,
\begin{equation}
 \mathbb P\!\left(
 \mathbb E_z^\omega[\Delta\otimes\Delta]
 \not\ge\frac\kappa2r^2I\right)
 \le Cr^{-d/(d-1)}.
 \label{ren:exit-probability}
\end{equation}
\end{proposition}

\begin{proof}
On a fixed ball, continuity and strict positivity give local uniform
ellipticity, so the diffusion and Dirichlet problem are well posed and
the exit time is integrable. Optional stopping in
$|Y_{t\wedge\sigma}-z|^2-2(t\wedge\sigma)$ gives
$\mathbb E_z^\omega\sigma=r^2/2$; the bounded stopped coordinate martingales
and path continuity give the other identities.

The determinant form of ABP yields
\[
 \mathbb E_z^\omega\int_0^\sigma\mathbf1_{\{b<\kappa\}}(Y_t)\,dt
 \le C_dr\left(\int_{B_r(z)}H_\kappa\right)^{1/d}
 \le C_dr^2\left(\fint_{B_r(z)}H_\kappa\right)^{1/d}.
\]
This follows first for continuous sources by the stopped It\^o formula,
and then for nonnegative Borel sources by approximation.
Choose $\eta_d$ so that the last expression is at most $r^2/4$.
For $|e|=1$,
\begin{align*}
 \mathbb E_z^\omega(e\cdot\Delta)^2
 &=2\mathbb E_z^\omega\int_0^\sigma e\cdot B(Y_t)e\,dt\\
 &\ge2\kappa\left(\mathbb E_z^\omega\sigma
 -\mathbb E_z^\omega\int_0^\sigma\mathbf1_{\{b<\kappa\}}(Y_t)\,dt\right)
 \ge\kappa r^2/2.
\end{align*}
Choose $\kappa$ with $\mathbb EH_\kappa<\eta_d/2$ and apply
\eqref{ren:average-compression} and Markov's inequality to obtain
\eqref{ren:exit-probability}.
\end{proof}

For a bounded convex domain $U$, let $\Gamma_Uu$ be the convex envelope
of $u\in C(\overline U)$. Write $\partial w(E)$ for the union, over $x\in E$,
of the subdifferentials of a convex function $w$:
\[
 \partial w(x):=\{p:w(y)\ge w(x)+p\cdot(y-x)\text{ for every }y\in U\}.
\]
For the pair of operators
\begin{equation}
 F(N,x)=-B(x):N+f(x),\qquad F^*(N,x)=-B(x):N-f(x),
 \label{ren:operator-dual}
\end{equation}
set
\begin{equation}
 \mu(U,F):=\frac1{|U|}\sup\left\{|\partial\Gamma_Uu(U)|:
 u\in C(\overline U),\ F(D^2u,x)\ge0\text{ in }U\right\}.
 \label{ren:mu-definition}
\end{equation}
The inequality is understood in the viscosity sense. Local H\"older
continuity and local uniform ellipticity imply that solutions
are $C^2$ in the interior.

\begin{lemma}[Replacement by a solution and the exact determinant bound]
\label{ren:contact-mass}
Let $Q$ be a bounded cube and let $f$ be locally H\"older continuous in
a neighborhood of this cube. Then
\begin{equation}
 \mu(Q,F)\le\frac1{d^d|Q|}\int_Q(f_+)^dD.
 \label{ren:mu-exact-bound}
\end{equation}
The supremum in \eqref{ren:mu-definition} can be restricted to solutions $F(D^2v,x)=0$ with continuous boundary data.
For these solutions, the Monge--Amp\`ere measure of the convex envelope
assigns zero mass to every Lebesgue null set, in particular to the grid
boundaries of the subcubes.
\end{lemma}

\begin{proof}
Given an admissible $u$, let $v$ solve $F(D^2v,x)=0$ with boundary data $u$.
Comparison gives $v\le u$. If $L$ supports $u$ from below at an interior
contact point $x_0$, then $v-L\ge0$ on the boundary and
$v(x_0)-L(x_0)\le0$. Thus $v-L$ has an interior minimum, and the slope of
$L$ belongs to $\partial\Gamma v(Q)$.
The Monge--Amp\`ere measure vanishes on the noncontact set by
\cite[Lemma~2.4]{ASquant}. That lemma concerns a continuous function and
its convex envelope and has no ellipticity hypothesis. Consequently
$|\partial\Gamma u(Q)|\le|\partial\Gamma v(Q)|$, proving the replacement
assertion.

For a solution, put $E=\{v=\Gamma v\}\cap Q$. At $x\in E$ the
supporting slope is $Dv(x)$, and
\[
 H_x:=D^2v(x)\ge0,\qquad B(x):H_x=f(x),\qquad
 \det H_x\le\frac{(B:H_x/d)^d}{\det B}
 \le d^{-d}(f_+)^dD.
\]
The determinant inequality is the arithmetic--geometric mean inequality
for $B^{1/2}H_xB^{1/2}$. Apply the area formula to the locally $C^1$ map
$Dv$ on a compact exhaustion of $Q$:
\[
 |\partial\Gamma v(Q)|=|Dv(E)|
 \le\int_E\det D^2v\le d^{-d}\int_Q(f_+)^dD.
\]
Taking the supremum proves \eqref{ren:mu-exact-bound}. The same area formula
maps every null subset of $E$ to a null set of slopes. Together with the
vanishing of the noncontact mass, this proves absolute continuity of the
envelope measure, including zero mass on grid boundaries.
\end{proof}

\begin{lemma}[Local measurability, subadditivity, and moment bounds]
\label{ren:mu-basic}
The random variable $\mu(Q,F)$ is measurable and depends only on $(B,f)$
inside $Q$. If $Q_{m+n}$ is partitioned into $N=3^{dn}$ triadic subcubes
$Q_i$ of scale $m$, then
\begin{equation}
 \mu(Q_{m+n},F)\le\frac1N\sum_{i=1}^N\mu(Q_i,F).
 \label{ren:subadditivity}
\end{equation}
Consequently, for each $p\ge1$, both $\mathbb E\mu(Q_m,F)$ and
$\mathbb E\mu(Q_m,F)^p$ are nonincreasing in $m$ whenever the corresponding
moments are finite. If $|f|\le K$, then
\begin{equation}
 \mathbb E\mu(Q_m,F)^q\le(K/d)^{dq}\mathbb ED(0)^q.
 \label{ren:mu-initial-moment}
\end{equation}
Moreover, $s\mapsto\mu(Q,F+s)$ is nondecreasing.
\end{lemma}

\begin{proof}
By Lemma~\ref{ren:contact-mass}, it suffices to use solutions.
Choose a countable dense family of continuous boundary data. Uniform
convergence of the data implies uniform convergence of solutions by
comparison, and hence of their convex envelopes. Weak convergence of
Monge--Amp\`ere measures and lower semicontinuity on open sets give
\[
 |\partial\Gamma v(Q)|\le\liminf_j|\partial\Gamma v_j(Q)|.
\]
Thus the countable family gives the same supremum. For fixed boundary
data, comparison and viscosity stability give measurable dependence on
$(B,f)|_Q$; the convex measure is measurable by local approximation.
Boundary coefficient values are determined by continuity from the interior.
This proves local measurability.

Replace a near-maximizer on $Q_{m+n}$ by a solution $v$.
Its envelope has zero mass on the grid boundaries, so
\[
 |\partial\Gamma_{Q_{m+n}}v(Q_{m+n})|
 =\sum_i|\partial\Gamma_{Q_{m+n}}v(Q_i)|
 \le\sum_i|Q_i|\mu(Q_i,F).
\]
Here a supporting plane on the large cube remains supporting on each
subcube. Letting the approximation error tend to zero gives
\eqref{ren:subadditivity}. Stationarity and Jensen then give monotonicity
of the first and $p$th moments.
Finally, \eqref{ren:mu-exact-bound}, Jensen and stationarity imply
\eqref{ren:mu-initial-moment}, while inclusion of the supersolution classes
gives monotonicity in $s$.
\end{proof}

\begin{lemma}[A geometric ABP estimate without ellipticity constants]
\label{ren:geometric-abp}
There exists $C_d$ such that every $u\in C(\overline Q_m)$ satisfying
$F(D^2u,x)\ge0$ obeys
\begin{equation}
 \inf_{\partial Q_m}u-\inf_{Q_m}u
 \le C_d3^{2m}\mu(Q_m,F)^{1/d}.
 \label{ren:geometric-abp-estimate}
\end{equation}
If the left-hand side is negative, the assertion is understood with its
positive part.
\end{lemma}

\begin{proof}
Suppose that the left-hand side is a positive number $h$, and choose an
interior minimum point $x_0$.
For $|p|<h/(2\operatorname{diam}Q_m)$, the boundary values of
$u(x)-p\cdot(x-x_0)$ are strictly greater than its value at $x_0$.
It therefore attains its minimum in the interior, and
$p$ belongs to $\partial\Gamma u(Q_m)$.
Consequently,
\[
 c_d\left(\frac h{\operatorname{diam}Q_m}\right)^d
 \le|\partial\Gamma u(Q_m)|
 \le |Q_m|\mu(Q_m,F).
\]
The conclusion follows from
$\operatorname{diam}Q_m=\sqrt d\,3^m$ and $|Q_m|=3^{dm}$.
\end{proof}

\subsection{Convex growth and the degenerate region}

The next lemma is uniform in the source amplitude, as required when the
contact mass is normalized below.

\begin{lemma}[Thin-layer exclusion]
\label{ren:thin-layer}
There exist $h_d,c_d,\eta_d>0$, depending only on the dimension, with the
following property.
Let $Q_0=(-1/2,1/2)^d$, and suppose that $B$ is locally H\"older continuous
and strictly positive definite in a neighborhood of this cube, with
$\operatorname{tr}B=1$.
Given $\kappa>0$, choose
\begin{equation}
 h\ge\max\{h_d,8/\kappa\},\qquad
 0<r<(4\sqrt h)^{-1},\qquad 3^n\le c_dr.
 \label{ren:thin-parameters}
\end{equation}
Suppose that $u\in C(\overline Q_0)\cap C^2(Q_0)$ is a supersolution of $F$
and that there exist $x_0\in Q_n$ and a unit vector $e$ such that
\begin{equation}
 \inf_{Q_0}u=u(x_0)=0,\qquad
 u\ge hr^2\quad\text{on }\{x\in Q_0:|e\cdot x|\ge r\}.
 \label{ren:thin-hypothesis}
\end{equation}
If
\begin{equation}
 \frac1{|B_{2\sqrt h r}|}
 \int_{B_{2\sqrt h r}}D\mathbf1_{\{b<\kappa\}}
 \le\eta_d,
 \label{ren:thin-good}
\end{equation}
then there is a triadic subcube $Q_n(x)\subset Q_0$ of scale $n$ such that
\begin{equation}
 \mu(Q_n(x),F)\ge2.
 \label{ren:thin-conclusion}
\end{equation}
\end{lemma}

\begin{proof}
Set
\[
 \varphi(x)=|x|^2-\frac h8(e\cdot x)^2,\qquad
 g(x)=\left(2-\frac h4e\cdot B(x)e\right)_+.
\]
Since $h\ge8/\kappa$, $0\le g\le2\mathbf1_{\{b<\kappa\}}$.
Let $z$ solve $-B:D^2z=g$ in $B_{2\sqrt h r}\Subset Q_0$, with zero
boundary data. Comparison and determinant ABP give
\[
 0\le z\le C_d\sqrt h\,r
   \left(\int_{B_{2\sqrt h r}}D\mathbf1_{\{b<\kappa\}}\right)^{1/d}
 \le C_d\eta_d^{1/d}hr^2\le hr^2/16.
\]
The last inequality fixes $\eta_d$ independently of $\kappa,h,r$.
Also $-B:D^2(\varphi+z)=\frac h4e\cdot Be-2+g\ge0$, so
$\widetilde u:=u+\varphi+z$ is a supersolution of $F$.

Put $S_1:=\{|e\cdot x|<r,\ |x|<\sqrt h\,r\}$ and $S_2:=2S_1$.
On $S_2\setminus S_1$, either $|e\cdot x|\ge r$ and $u\ge hr^2$, or
$|x|^2\ge hr^2$ and $|e\cdot x|<r$. In both cases,
\begin{equation}
 \widetilde u\ge hr^2/2.
 \label{ren:thin-annulus}
\end{equation}
Choose $c_d$ so that $|x_0|^2\le r^2$, and $h_d\ge16$. Then
$\widetilde u(x_0)\le r^2+hr^2/16\le hr^2/4$.
Every slope $|p|\le c_d\sqrt h\,r$ therefore makes
$\widetilde u-p\cdot x$ attain its minimum over $S_2$ in $S_1$, giving
\begin{equation}
 |\partial\Gamma_{S_2}\widetilde u(S_1)|\ge c_dh^{d/2}r^d.
 \label{ren:thin-slope-mass}
\end{equation}
The scale-$n$ cubes meeting $S_1$ lie in $S_2$ and have total volume at most
\[
 C_d(r+3^n)(\sqrt h\,r+3^n)^{d-1}\le C_dh^{(d-1)/2}r^d.
\]
Since $g$ is locally H\"older, $z$ and $\widetilde u$ are locally $C^2$.
Their contact slopes are gradients, so the area formula, as in
Lemma~\ref{ren:contact-mass}, gives zero mass to the grid boundaries.
Restricting each supporting plane to its contact cube now yields
\[
 |\partial\Gamma_{S_2}\widetilde u(S_1)|
 \le C_dh^{(d-1)/2}r^d
       \max_{Q_n(x)\cap S_1\ne\varnothing}\mu(Q_n(x),F).
\]
Comparison with \eqref{ren:thin-slope-mass} gives a lower bound
$c_d\sqrt h$ for the maximum. Increase $h_d$ to make it at least 2.
\end{proof}

The ball in \eqref{ren:thin-good} has deterministic center and is independent
of $e$, which may depend on the environment and on $u$. Stationarity gives,
for each deterministic $L>0$,
\begin{equation}
 \mathbb P\!\left(
 \fint_{B_{2\sqrt h rL}}H_\kappa>\eta_d\right)
 \le\eta_d^{-1}\mathbb EH_\kappa(0).
 \label{ren:geometry-probability}
\end{equation}

The remaining geometric input is purely convex. All subcubes below
belong to one compatible triadic grid. We invoke the following result
directly from \cite[Lemma~3.1]{ASquant}; its constants depend only on
dimension. The operator-dependent step is Lemma~\ref{ren:thin-layer},
which replaces the uniformly elliptic argument of
\cite[Lemma~3.2]{ASquant} by the averaged control of the bad region.

\begin{lemma}[Purely convex alternative, Lemma~3.1 of \cite{ASquant}]
\label{ren:convex-alternative}
There exist $c_0,h_0>0$, depending only on $d$, such that, if
$0<r<1$, $3^n\le c_0r$, and $w\in C(\overline Q_0)$ is convex with
\[
 \inf_{Q_0}w=\inf_{Q_n}w=0,\qquad
 |\partial w(Q_n(x))|\ge |Q_n|
 \quad\text{for every }Q_n(x)\subset Q_0,
\]
then at least one of the following holds:
\begin{align}
 w&\ge h_0r^{2-2/d}&&\text{on }\partial Q_0,
 \label{ren:bowl-alternative}\\
 w&\ge h_0r^{2-2/d}&&\text{on }
     \{x\in Q_0:|e\cdot x|\ge r\}
     \text{ for some }|e|=1.
 \label{ren:stripe-alternative}
\end{align}
\end{lemma}

Lemma~\ref{ren:thin-layer} excludes the second alternative under the
bad-region bound, as follows.

\begin{lemma}[Bowl-shaped growth under local control of the bad region]
\label{ren:bowl}
Fix $\kappa$, choose $h$ as in Lemma~\ref{ren:thin-layer}, and then take
\begin{equation}
 0<r<\min\{1,(4\sqrt h)^{-1},(h_0/h)^{d/2}\}.
 \label{ren:bowl-r-choice}
\end{equation}
There exist deterministic $n_g\ge1$ and $c_b>0$ such that the following
holds for integers $n\ge n_g$, $m\in\mathbb Z$, and $a_0>0$.
Suppose that $u$ solves $F(D^2u,x)=0$ on $Q_{m+n+1}$ and that
\begin{equation}
 a_0\le\frac{|\partial\Gamma_{Q_{m+n+1}}u(Q_m(x))|}{|Q_m|}
 \le\mu(Q_m(x),F)\le(1+3^{-dn})a_0
 \quad\text{for every }Q_m(x)\subset Q_{m+n}.
 \label{ren:bowl-flatness}
\end{equation}
If the averaged condition in \eqref{ren:thin-good} holds on the physical
ball $B_{2\sqrt h r\,3^{m+n}}(0)$, then there exist a central contact point
$x_0\in Q_m$ and a supporting slope $p_0$ such that
\begin{equation}
 u(x)-u(x_0)-p_0\cdot(x-x_0)
 \ge c_ba_0^{1/d}3^{2(m+n)}
 \quad\text{on }Q_{m+n+1}\setminus Q_{m+n}.
 \label{ren:bowl-bound}
\end{equation}
After subtracting this affine function, $u\ge0$ throughout its domain and
has a zero in the central cube $Q_m$.
\end{lemma}

\begin{proof}
For $t,L>0$, the rescaling
$G(N,x):=t^{-1}F(tN,Lx)$ and $v(x):=(tL^2)^{-1}u(Lx)$ satisfies
\begin{equation}
 \mu(L^{-1}U,G)=t^{-d}\mu(U,F).
 \label{ren:mu-scaling}
\end{equation}
Indeed, slopes scale by $(tL)^{-1}$ and volume by $L^{-d}$.
Take $L=3^{m+n}$ and $t=a_0^{1/d}$ to reduce to $m+n=0$ and $a_0=1$.
Only the source amplitude changes, which is allowed in
Lemma~\ref{ren:thin-layer}.

Set $w=\Gamma_{Q_1}u$. Positive contact mass in the central cube provides
a point $x_0\in Q_{-n}$ with $u(x_0)=w(x_0)$ and a supporting slope $p_0$.
Subtract its affine support. Then $u\ge w\ge0$ on $Q_1$ and
$u(x_0)=w(x_0)=0$. At interior points of $Q_0$, the subdifferentials of
$w$ relative to $Q_0$ and $Q_1$ agree. Thus \eqref{ren:bowl-flatness}
implies the mass assumption of Lemma~\ref{ren:convex-alternative}.

Choose $n_g$ so that $3^{-n}$ satisfies both geometric grid thresholds.
The thin-layer alternative would give
$u\ge w\ge h_0r^{2-2/d}\ge hr^2$ outside a layer of width $2r$.
Lemma~\ref{ren:thin-layer} would then give a subcube with $\mu\ge2$,
contradicting \eqref{ren:bowl-flatness}. Hence
$w\ge h_0r^{2-2/d}$ on $\partial Q_0$.
For $x\in Q_1\setminus Q_0$, write the intersection of $[x_0,x]$ with
$\partial Q_0$ as $x_0+\tau(x-x_0)$, where $0<\tau\le1$.
Convexity gives $h_0r^{2-2/d}\le\tau w(x)\le w(x)\le u(x)$.
Rescaling proves \eqref{ren:bowl-bound}. Affine subtraction leaves the
spatial origin unchanged, so the bad-region ball remains deterministically
centered.
\end{proof}

\subsection{The contraction estimate}

From now on, the base scale is large enough that $3^m\ge\ell_f+1$.
For a fixed source operator and a fixed shift of the source, write
\begin{equation}
 a_m:=(\mathbb E\mu(Q_m,F))^q,\qquad
 b_m:=\mathbb E\mu(Q_m,F)^q.
 \label{ren:ab-sequences}
\end{equation}

\begin{lemma}[Moment bound under block averaging]
\label{ren:q-compression}
There is a common deterministic constant $C$ such that, for all $m$ as
above and all integers $n\ge0$,
\begin{equation}
 b_{m+n}\le a_m+C3^{-n}b_m.
 \label{ren:q-compression-estimate}
\end{equation}
The constant is uniform over the specified source class.
\end{lemma}

\begin{proof}
Partition $Q_{m+n}$ into $N=3^{dn}$ scale-$m$ cubes.
Local measurability and $3^m\ge\ell_f+1$ permit a fixed finite coloring
with mutually independent block variables in each color.
Subtract the common mean in \eqref{ren:subadditivity} and apply the
independent-sum estimate from Lemma~\ref{ren:local-average} to get
\[
 \|\mu(Q_{m+n},F)\|_{L^q}
 \le a_m^{1/q}+CN^{-(q-1)/q}b_m^{1/q}.
\]
If $b_m>0$, divide by $b_m^{1/q}$ and use
$(x+\varepsilon)^q\le x^q+C_q\varepsilon$ for $x,\varepsilon\in[0,1]$.
Since $a_m\le b_m$ and $d(q-1)/q=1$, this gives
\eqref{ren:q-compression-estimate} for large $n$.
For the remaining $n$, use $b_{m+n}\le b_m$ and enlarge $C$.
The case $b_m=0$ is immediate.
\end{proof}

\begin{lemma}[Relative $L^q$ flatness]
\label{ren:relative-concentration}
Suppose that $X\ge0$, $\mathbb EX\ge a>0$, and
$\mathbb EX^q\le(1+\delta)a^q$.
Then, for $0<\varepsilon\le1$,
\begin{equation}
 \mathbb P(|X-a|>\varepsilon a)
 \le C_q\delta\varepsilon^{-2}.
 \label{ren:relative-concentration-estimate}
\end{equation}
\end{lemma}

\begin{proof}
Set $J_q(t)=t^q-1-q(t-1)$. Strict convexity gives $J_q\ge0$.
On $[0,2]$, the lower bound on its second derivative gives
$J_q(t)\ge c_q(t-1)^2$; on $[2,\infty)$, $J_q$ is increasing.
Hence $J_q(t)\ge c_q\varepsilon^2$ whenever $|t-1|>\varepsilon$.
On the other hand,
\[
 \mathbb EJ_q(X/a)
 =a^{-q}\mathbb EX^q-1-q(\mathbb EX/a-1)\le\delta.
\]
Markov's inequality proves the assertion.
\end{proof}

\begin{proposition}[Contraction from flatness]
\label{ren:flatness-contraction}
Assume the Harnack estimate \eqref{ren:harnack-event}. There exist a deterministic
integer $n_0\ge2$, $\delta_0\in(0,1)$, and $C<\infty$ with the following
property.
Let $m\ge m_H$ and $s>0$, and set
\[
 a:=\mathbb E\mu(Q_{m+n_0},F+s)>0,
 \qquad a_*:=\mathbb E\mu(Q_{m+n_0},F^*+s)>0.
\]
If
\begin{align}
 \mathbb E\mu(Q_m,F+s)^q&\le(1+\delta_0)a^q,
 \label{ren:flat-assumption}\\
 \mathbb E\mu(Q_m,F^*+s)^q&\le(1+\delta_0)a_*^q,
 \label{ren:flat-assumption-dual}
\end{align}
then
\begin{equation}
 \mathbb E\mu(Q_{m+n_0},F+s)^q
 +\mathbb E\mu(Q_{m+n_0},F^*+s)^q
 \le Cs^{dq}.
 \label{ren:flat-conclusion}
\end{equation}
These constants are uniform over the specified source class.
\end{proposition}

\begin{proof}
First choose $\kappa$ with
$\eta_d^{-1}\mathbb EH_\kappa<p_H/4$, and then $h,r,c_b,n_g$ from
Lemma~\ref{ren:bowl}. Choose $n_0$ with $n_0-1\ge n_g$ and large enough
for the absorption below. Set $N=3^{dn_0}$,
$0<\varepsilon\le(16N^2)^{-1}$, and choose $\delta_0$ last.

Lemma~\ref{ren:relative-concentration}, stationarity and moment
monotonicity imply that, except on a set of probability
$2C_q(N+1)\delta_0\varepsilon^{-2}$, all scale-$m$ subcubes $Q_i$ satisfy
\begin{align}
 \mu(Q_{m+n_0},F+s)&\ge(1-\varepsilon)a,&
 \mu(Q_i,F+s)&\le(1+\varepsilon)a,
 \label{ren:near-mean}\\
 \mu(Q_{m+n_0},F^*+s)&\ge(1-\varepsilon)a_*,&
 \mu(Q_i,F^*+s)&\le(1+\varepsilon)a_*.
 \label{ren:near-mean-dual}
\end{align}
Choose $\delta_0$ to make this failure probability less than $p_H/4$.
By \eqref{ren:geometry-probability} with $L=3^{m+n_0-1}$, the bad-region
condition fails with probability less than $p_H/4$. Their intersection
with $\mathcal H_m(0)$ therefore has probability at least $p_H/2$.
Fix an environment in this intersection.

Choose a $(1-\varepsilon)$ near-maximizer on $Q_{m+n_0}$ and replace it
by the solution $u$. For its envelope relative to this cube, set
$v_i=|\partial\Gamma u(Q_i)|/|Q_m|$. Zero grid-boundary mass gives
\[
 v_i\le(1+\varepsilon)a,\qquad
 \sum_{i=1}^Nv_i\ge N(1-\varepsilon)^2a,
\]
and consequently
\[
 v_i\ge\big[N(1-\varepsilon)^2-(N-1)(1+\varepsilon)\big]a
 \ge(1-3N\varepsilon)a=:a_-.
\]
Our choice of $\varepsilon$ gives $a_-\ge a/2$ and
\[
 \frac{(1+\varepsilon)a}{a_-}
 \le1+8N\varepsilon\le1+3^{-d(n_0-1)}.
\]
Lemma~\ref{ren:bowl}, applied to the inner cube $Q_{m+n_0-1}$ and the
outer domain $Q_{m+n_0}$, supplies an affine normalization such that
\begin{equation}
 u\ge0,\quad u(x_u)=0\text{ for some }x_u\in Q_m,\quad
 u\ge c3^{2(m+n_0)}a^{1/d}\text{ on }\partial Q_{m+n_0}.
 \label{ren:normalized-u}
\end{equation}
Construct $u_*$ similarly for the dual. In particular, the central contact
points are chosen after replacement by solutions.

Put $S=a^{1/d}+a_*^{1/d}$ and $R=3^{m+n_0}$. Since
$B:D^2u=f+s$ and $B:D^2u_*=-f+s$, comparison with
$cR^2S+s|x|^2-sdR^2/4$ gives
\begin{equation}
 u(0)+u_*(0)\ge cR^2S-C_dsR^2.
 \label{ren:sum-lower}
\end{equation}
Here the comparison function has $B$-trace $2s$ and lies below the sum
on the boundary by \eqref{ren:normalized-u}.

Let $v,v_*$ solve the respective equations on $Q_{m+1}$ with zero
boundary values. Geometric ABP bounds their negative parts by the
corresponding $\mu$; for the positive part of $v$, use $F^*-s\le F^*+s$,
and interchange the operators for $v_*$. Subadditivity and
\eqref{ren:near-mean}--\eqref{ren:near-mean-dual} give
\begin{equation}
 \|v\|_\infty+\|v_*\|_\infty\le C_d3^{2m}S.
 \label{ren:small-dirichlet-bound}
\end{equation}
The function $u-v$ is nonnegative and $B$-harmonic. On $\mathcal H_m(0)$,
\[
 u(0)-v(0)\le H\big(u(x_u)-v(x_u)\big)=-Hv(x_u),
\]
so $u(0)\le(H+1)\|v\|_\infty$, and similarly for $u_*$.
Together with \eqref{ren:sum-lower} and \eqref{ren:small-dirichlet-bound},
this yields
\[
 c3^{2n_0}S\le C_{d,H}S+C_ds3^{2n_0}.
\]
Our choice of $n_0$ allows us to absorb the first term and obtain $S\le Cs$.
Moment monotonicity and \eqref{ren:flat-assumption}--\eqref{ren:flat-assumption-dual}
then imply \eqref{ren:flat-conclusion}.
\end{proof}

\subsection{Centering and decay of the monotone quantity}

We now choose the center from the monotone quantity. Write
$F_0(N,x)=-B(x):N+f(x)$ for the uncentered operator.

\begin{lemma}[Balanced center]
\label{ren:balanced-center}
If $|f|\le1$, there is a deterministic $c_f\in[-1,1]$ such that, upon setting
$F_c=F_0-c_f$ and $F_c^*=F_0^*+c_f$, one has, for every $s>0$,
\begin{align}
 \lim_m\mathbb E\mu(Q_m,F_c+s)
 &\ge\lim_m\mathbb E\mu(Q_m,F_c^*-s),
 \label{ren:balance-positive}\\
 \lim_m\mathbb E\mu(Q_m,F_c-s)
 &\le\lim_m\mathbb E\mu(Q_m,F_c^*+s).
 \label{ren:balance-negative}
\end{align}
\end{lemma}

\begin{proof}
The limits $l(t):=\lim_m\mathbb E\mu(Q_m,F_0+t)$ and
$l_*(t):=\lim_m\mathbb E\mu(Q_m,F_0^*+t)$ exist by scale monotonicity
and the moment bound, and are nondecreasing in $t$.
For $t<-1$, both vanish by \eqref{ren:mu-exact-bound}. For $t>1$,
$(t-1)|x|^2/2$ is admissible for both operators, so
$l(t),l_*(t)\ge(t-1)^d$.
Thus $l(t)-l_*(-t)$ is nondecreasing, negative for $t<-1$ and positive for
$t>1$. Choose its crossing threshold $t_0\in[-1,1]$ and set $c_f=-t_0$.
The inequalities on either side of $t_0$ give
\eqref{ren:balance-positive}--\eqref{ren:balance-negative}, regardless of
whether the function is continuous at $t_0$.
\end{proof}

\begin{lemma}[Uniform lower bound for positive shifts after balancing]
\label{ren:positive-source}
If $F$ satisfies \eqref{ren:balance-positive}, there exists $c_0>0$,
depending only on $d$, such that, for every $m\in\mathbb Z$ and $s>0$,
\begin{equation}
 \mathbb E\mu(Q_m,F+s)\ge c_0s^d.
 \label{ren:positive-source-bound}
\end{equation}
If \eqref{ren:balance-negative} also holds, the same conclusion is valid
for $F^*$.
\end{lemma}

\begin{proof}
Let $l,l_*$ be the centered limits and put
$A_0:=\sup_{t>0}l_*(-t)<\infty$. Given $m\in\mathbb Z$ and $s,\eta>0$, take $M\ge m$
large enough that
$\mathbb E\mu(Q_M,F^*-s/2)\le A_0+\eta$.
Let $v_*$ solve $F^*(D^2v_*,x)=s/2$ with zero boundary data.
Geometric ABP and Markov give, with probability at least $1/2$,
\[
 \inf_{Q_M}v_*\ge-C_d(A_0+\eta)^{1/d}3^{2M}.
\]
For $L=3^M$, set $v(x)=\frac s8(|x|^2-L^2/4)-v_*(x)$.
Since $F(-D^2v_*)=-s/2$ and the paraboloid has $B$-trace $s/4$,
\begin{equation}
 F(D^2v,x)+s=s/4\ge0.
 \label{ren:positive-parabola-sign}
\end{equation}
Moreover, $v\ge0$ on $\partial Q_M$, while on the event just obtained,
$v(0)\le-sL^2/32+C_d(A_0+\eta)^{1/d}L^2$.
Geometric ABP and $(a-b)_+^d\ge2^{1-d}a^d-b^d$ imply
\[
 \mathbb E\mu(Q_M,F+s)\ge c_ds^d-C_d(A_0+\eta).
\]
Scale monotonicity transfers this to $m$.
For $0<t<s$, balancing gives $l(s)\ge l(t)\ge l_*(-t)$; taking the supremum
in $t$ gives the second lower bound $\mathbb E\mu(Q_m,F+s)\ge A_0$.
Let $\eta\downarrow0$ and combine:
\[
 \mathbb E\mu(Q_m,F+s)
 \ge\max\{A_0,c_ds^d-C_dA_0\}\ge c_0s^d.
\]
Interchanging $F,F^*$ proves the dual statement.
\end{proof}

The lower bound ensures positive means in each application of
Proposition~\ref{ren:flatness-contraction}.

From now on, $F$ denotes the operator centered using
Lemma~\ref{ren:balanced-center}.
Its source amplitude is at most 2.
Take $s_k=2^{-k}$ and define
\begin{align}
 a(m,k)&:=\big(\mathbb E\mu(Q_m,F+s_k)\big)^q,&
 b(m,k)&:=\mathbb E\mu(Q_m,F+s_k)^q,\nonumber\\
 a_*(m,k)&:=\big(\mathbb E\mu(Q_m,F^*+s_k)\big)^q,&
 b_*(m,k)&:=\mathbb E\mu(Q_m,F^*+s_k)^q.
 \label{ren:four-sequences}
\end{align}
All four quantities are nonincreasing in $m,k$.
The initial moment bound and the positive-shift lower bound provide common
deterministic constants $\underline c>0$ and $B_0<\infty$ such that
\begin{equation}
 \underline c2^{-dqk}\le a(m,k)\le b(m,k)\le B_0,
 \qquad
 \underline c2^{-dqk}\le a_*(m,k)\le b_*(m,k)\le B_0.
 \label{ren:four-bounds}
\end{equation}
For example, one may take $B_0=C_d\mathbb ED^q$, since the amplitude of
the centered source after adding $s_k$ is at most 3.

\begin{proposition}[Algebraic decay of the monotone quantity]
\label{ren:mu-decay}
There exist common deterministic $C<\infty$, $\zeta>0$, and an integer
$m_*$ such that
\begin{equation}
 \mathbb E\big[\mu(Q_m,F)^q+\mu(Q_m,F^*)^q\big]
 \le C3^{-\zeta m}\qquad(m\ge m_*).
 \label{ren:mu-decay-estimate}
\end{equation}
\end{proposition}

\begin{proof}
Take $m_*\ge m_H$ with $3^{m_*}\ge\ell_f+1$, and fix
$n_0,\delta_0$ from Proposition~\ref{ren:flatness-contraction}.
Choose $0<\delta_1\le\delta_0/4$ and $n_1>n_0$ so that
\begin{equation}
 C3^{-(n_1-n_0)}(1+\delta_1)
 \le\frac{\delta_0}{4(1+\delta_0)},
 \label{ren:iteration-parameters}
\end{equation}
where $C$ is from \eqref{ren:q-compression-estimate}.

Fix $k$ and suppose all four sequences are bounded by $B$ at scale
$m_0\ge m_*$. Along $m_j=m_0+jn_1$, the sum of the logarithmic decrease
ratios of any one sequence is at most
$\log(B/(\underline c2^{-dqk}))$, by \eqref{ren:four-bounds}.
If every step decreases at least one sequence by a factor greater than
$1+\delta_1$, the total decrease of the four sequences exceeds
$J\log(1+\delta_1)$. Thus, whenever
\begin{equation}
 J>\frac{4\log(B/(\underline c2^{-dqk}))}{\log(1+\delta_1)},
 \label{ren:log-budget}
\end{equation}
some deterministic endpoint $m=m_j$ satisfies
\begin{equation}
 a(m-n_1,k)\le(1+\delta_1)a(m,k),\qquad
 b(m-n_1,k)\le(1+\delta_1)b(m,k),
 \label{ren:common-flat-step}
\end{equation}
and the two starred inequalities.
Applying Lemma~\ref{ren:q-compression} over $n_1-n_0$ scales gives
\begin{align*}
 b(m-n_0,k)
 &\le a(m-n_1,k)+C3^{-(n_1-n_0)}b(m-n_1,k)\\
 &\le(1+\delta_1)a(m,k)
   +C3^{-(n_1-n_0)}(1+\delta_1)b(m-n_0,k).
\end{align*}
Absorbing by \eqref{ren:iteration-parameters} yields
$b(m-n_0,k)\le(1+\delta_0)a(m,k)$, and likewise for the dual.
The positive-shift bound makes both means positive, so
Proposition~\ref{ren:flatness-contraction} gives
\begin{equation}
 b(m,k)+b_*(m,k)\le C_12^{-dqk}.
 \label{ren:level-bound}
\end{equation}

Start with $k=0$ and $B=B_0$, at a uniformly bounded cost in scales.
After level $k-1$, monotonicity in $k$ permits the new upper bound
$B=C_12^{-dq(k-1)}$. Its ratio to the lower bound
$\underline c2^{-dqk}$ is $C_12^{dq}/\underline c$, independent of $k$.
Thus each level costs a fixed number of scales in \eqref{ren:log-budget}.
Inductively,
\[
 m_k\le m_*+C_2(k+1),\qquad
 b(m_k,k)+b_*(m_k,k)\le C_12^{-dqk}.
\]
The scales $m_k$ may depend on $f$, but their deterministic upper bound is
common to the class. For large $m$, take
$k=\lfloor(m-m_*)/C_2\rfloor-1$ and use monotonicity in $m$ and in the
source shift to obtain \eqref{ren:mu-decay-estimate}.
\end{proof}

\subsection{Cell estimates and identification of the center}

\begin{proof}[Proof of Theorem~\ref{ren:bounded-source}]
Retain the center $c_f$ from Lemma~\ref{ren:balanced-center}. Let $U_m$
solve $-B:D^2U_m+f-c_f=0$ in $Q_m$ with zero boundary data.
Geometric ABP applied to $\pm U_m$ gives
\[
 3^{-2m}\|U_m\|_\infty
 \le C_d\big[\mu(Q_m,F_c)^{1/d}+\mu(Q_m,F_c^*)^{1/d}\big].
\]
Proposition~\ref{ren:mu-decay} consequently yields
\begin{equation}
 \mathbb E\big[(3^{-2m}\|U_m\|_\infty)^{dq}\big]\le C3^{-\zeta m}.
 \label{ren:cube-potential-moment}
\end{equation}
Given $B_R(z)$, choose a cube centered at $z$ with side length between
$2R$ and $6R$. Solve on this cube and subtract the $B$-harmonic extension
of the solution's boundary values on the ball. The resulting zero-boundary
potential $U_{R,z}$ solves $-B:D^2U_{R,z}+f-c_f=0$ and has norm at most
twice the cube norm. Therefore
\begin{equation}
 \mathbb E\big[(R^{-2}\|U_{R,z}\|_\infty)^{dq}\big]\le CR^{-\zeta}.
 \label{ren:ball-potential-moment}
\end{equation}
Markov gives the claimed probability bound with, for example,
\begin{equation}
 \delta_B=\frac\zeta{2dq},\qquad \gamma_B=\frac\zeta2.
 \label{ren:exponent-choice}
\end{equation}

To identify $c_f$, use the ergodic invariant probability $\pi\sim\mathbb P$
from Lemma~\ref{down:density-normalized-law}. For the directional diffusion $Y$ started
at 0, put $A_t:=\int_0^t(f(Y_s)-\langle f\rangle_\pi)\,ds$.
The ergodic theorem gives $A_t/t\to0$ along almost every path for
$\mathbb P$-almost every environment.
For the exit time $\sigma_R$ from $B_R$, trace normalization gives
$\mathbb E_x^\omega\sigma_R=(R^2-|x|^2)/2$.
The strong Markov property then gives
\begin{equation}
 \mathbb E_0^\omega\sigma_R^2
 =2\mathbb E_0^\omega\int_0^{\sigma_R}
       \mathbb E_{Y_t}^\omega\sigma_R\,dt
 \le R^2\mathbb E_0^\omega\sigma_R\le R^4/2.
 \label{ren:exit-second-moment}
\end{equation}
Set $H_N:=\sup_{t\ge N}|A_t|/t$. Then
$0\le H_N\le2\|f\|_\infty$ and $H_N\to0$ pathwise. Since
$|A_{\sigma_R}|\le2\|f\|_\infty N+H_N\sigma_R$, Cauchy--Schwarz gives
\[
 R^{-2}\mathbb E_0^\omega|A_{\sigma_R}|
 \le2\|f\|_\infty N/R^2+C(\mathbb E_0^\omega H_N^2)^{1/2}.
\]
Let $R\to\infty$ and then $N\to\infty$, using dominated convergence:
\begin{equation}
 R^{-2}\mathbb E_0^\omega
       \int_0^{\sigma_R}(f(Y_s)-\langle f\rangle_\pi)\,ds\longrightarrow0.
 \label{ren:stopped-center-limit}
\end{equation}
Since
\[
 -U_{R,0}(0)=\mathbb E_0^\omega\int_0^{\sigma_R}(f(Y_s)-c_f)\,ds,
\]
we obtain $-R^{-2}U_{R,0}(0)\to(\langle f\rangle_\pi-c_f)/2$ almost surely.
Equation~\eqref{ren:ball-potential-moment} gives convergence to zero in
probability, so $c_f=\langle f\rangle_\pi$. Finally,
$v_{f,R,z}=-U_{R,z}$ has the sign convention of \eqref{ren:source-cell}.

All choices depend only on
$(d,q,\mathbb ED^q,\ell_f,H,p_H,m_H)$: first $\kappa,h,r,c_b$, then
$n_0,\varepsilon,\delta_0$, and finally $\delta_1,n_1$ and the number of
scales used in the iteration. The bounds use only the source amplitude and dependence range,
so the exponents and starting scale are common to the source class.
\end{proof}

\begin{proof}[Proof of Corollaries~\ref{ren:source-amplitude} and~\ref{ren:direction-cell}]
Apply Theorem~\ref{ren:bounded-source} to $f/T$. Linearity of the solution
and $\langle f/T\rangle_\pi=T^{-1}\langle f\rangle_\pi$ give
\eqref{ren:amplitude-estimate} with the same probability bound and starting
scale. For $h_T=\min\{a^{-1},T\}$, $h_T/T$ has amplitude at most 1 and
local dependence radius 0. It is locally H\"older since $a$ is locally
H\"older and bounded below on every compact set.

For the matrix estimate, $\operatorname{tr}B=1$ implies $|B|_F\le1$.
Choose a Frobenius orthonormal basis $E_1,\ldots,E_N$ of $\mathbb S^d$,
where $N=d(d+1)/2$, and apply the theorem to each $f_j=B:E_j$.
A union bound controls the $N$ potentials simultaneously.
For $M=\sum_jm_jE_j$ with $|M|\le1$, $\sum_j|m_j|\le\sqrt N$, so linearity
gives
\[
 \|w^B_{M,R,z}\|_\infty\le\sqrt N\max_j\|w^B_{E_j,R,z}\|_\infty.
\]
Absorb the dimensional factors into $C$ to obtain \eqref{ren:matrix-estimate}.
\end{proof}

\begin{remark}[Uniformity and subsequent source families]
\label{ren:uniformity-remark}
The common constants permit a deterministic diagonal choice such as
$T=R^t$ in \eqref{ren:amplitude-estimate}; each scale still concerns one
fixed source. For an unbounded source produced by division by $a$, this
estimate controls its bounded truncation. The remaining tail is treated
separately in Section~\ref{sec:green}.

\end{remark}

\section{Green estimates and the critical physical clock}
\label{sec:green}

We prove the cell estimate by combining the bounded-source theorem of
Section~\ref{sec:renormalization} with a truncation of the physical clock.
The main point is a coarse Green estimate with source exponent $p<d$.
This strict gain converts the critical $d$th moment into an algebraically
small truncation error. The Green estimate is pathwise and applies to
sources depending on the same environment; the signed principal part is
handled separately by the bounded-source theorem.

Fix the good-region threshold $\kappa$, the convex-order radius $a_c$ and
the fine-grid spacing $h_0$. All reverse H\"older constants and the exponent
$p$ will be chosen before the mesoscopic scale $L=R^v$. Constants below
may depend on these fixed parameters and the coefficient law, but not
on $R$, the pole, the source or the truncation height. The invariant laws
for the physical and normalized diffusions are those constructed in
Section~\ref{sec:density} and Lemma~\ref{down:density-normalized-law}.

\subsection{Critical moments of the coarse clock}
\label{clock:critical-moment}

Write
\begin{equation}
 a(x)=\operatorname{tr}A(x),\qquad B(x)=\frac{A(x)}{a(x)},
 \qquad h(x)=\frac1{a(x)},\qquad \operatorname{tr}B=1.
 \label{clock:normalization}
\end{equation}
Let $X$ be the normalized diffusion with generator $B:D^2$, and denote its quenched probability and expectation by $\mathbf P_x$ and $\mathbf E_x$. Physical time is recovered through $\int_0^t h(X_s)\,ds$. Section~\ref{sec:regularity} constructs the bad region, covariant under all translations,
\[
 W=\left\{x:\inf_{\overline B_2(x)}\lambda_{\min}(A)<\kappa\right\}.
\]
If $x\in W$, first exit the connected component $C_x$ containing $x$, arriving at a good point $Y$; if $x\notin W$, set $Y=x$ directly. Then exit $B_1(Y)$ starting from $Y$. Denote the resulting full stopping time in normalized time by $S$, and define the coarse kernel by
\begin{equation}
 K(x,E)=\mathbf P_x(X_S\in E),\qquad \mathcal L=K-I.
 \label{green:kernel}
\end{equation}
Proposition~\ref{prop:reg:coarse-process} and Lemma~\ref{lem:reg:convex-order} imply that every full coarse step is well defined, successive coarse stopping times do not accumulate, $K$ has zero drift, and, for every finite globally convex function $V$,
\begin{equation}
 \int (y-x)K(x,dy)=0,\qquad
 KV(x)\ge \frac1{|B_{a_c}|}\int_{B_{a_c}(x)}V(y)\,dy,
 \qquad a_c=\frac\kappa d.
 \label{green:convex-order-input}
\end{equation}
The moment bound~\eqref{reg:eq:coarse-moments} also gives the covariance estimate
\begin{equation}
 \Sigma(x)=\int(y-x)\otimes(y-x)K(x,dy)
 \quad\hbox{satisfies}\quad \Sigma(x)\ge a_c I.
 \label{green:covariance-lower}
\end{equation}
Convex order does not imply a minorization on arbitrary nonnegative tests.
For the clock estimate, use the cells
\[
 Q_k=k+[-1/2,1/2)^d,\qquad
 \widehat Q_k=k+[-3,3]^d.
\]
Set
\[
 I_k=\mathbf1_{\{\inf_{\widehat Q_k}\lambda_{\min}(A)<\kappa\}},
 \qquad H_k=1+|\mathcal C(k)|,
\]
where $\mathcal C(k)$ is the cluster of bad sites connected by $\ell^\infty$ nearest-neighbor adjacency, with the empty cluster used when $I_k=0$. Once $\kappa$ is fixed sufficiently small, there are deterministic $c,C>0$ such that
\begin{equation}
 \mathbb P(H_k\ge n)\le Ce^{-cn},\qquad n\ge1.
 \label{clock:geometric-tail}
\end{equation}
Every continuous bad component intersecting $Q_k$ has diameter at most $CH_k$. Thus, for all $x\in Q_k$, the full coarse path stays inside $B_{CH_k}(x)$, and
\begin{equation}
 \mathbf E_x S^j\le C_jH_k^{2j}\quad(j\ge1),\qquad
 |X_S-x|\le CH_k.
 \label{clock:whole-cell-step}
\end{equation}
Jensen's inequality gives noninteger orders. These bounds are uniform in $x\in Q_k$.

Let $\ell_{\rm dep}$ be the deterministic dependence distance of the original field, and define the local physical weights
\begin{equation}
 Z_j=\left(\inf_{j+[-4,4]^d}\lambda_{\min}(A)\right)^{-1},
 \qquad M_1=\sup_j\mathbb E Z_j^d<\infty.
 \label{clock:local-physical-weight}
\end{equation}
Indeed, finitely many unit balls cover the determining cube, and the $d$th power of their maximum inverse minimum is bounded by the sum of the $d$th powers.

\begin{lemma}[Weighted lattice-animal estimate]
\label{clock:weighted-animal}
The threshold $\kappa$ defining the bad region can be fixed sufficiently small once and for all so that, for arbitrary deterministic $j,k\in\mathbb Z^d$ and $n\ge1$,
\begin{equation}
 \mathbb E\left[Z_j^d\mathbf1_{\{H_k\ge n\}}\right]
 \le CM_1e^{-cn}.
 \label{clock:weighted-animal-bound}
\end{equation}
The constants are independent of $j,k,n$, and independence of $Z_j$ and $H_k$ is not required.
\end{lemma}

\begin{proof}
Choose $K_0$ so that the determining cubes of sites in each residue class modulo $K_0$ are separated by more than $\ell_{\rm dep}$. Finite-range dependence gives joint independence in each class, and Markov's inequality gives $p_\kappa:=\sup_i\mathbb P(I_i=1)\le C\kappa^d$.

For $n\ge2$, the event $H_k\ge n$ contains a bad animal of size $n-1$ anchored at $k$. Encoding a spanning-tree traversal gives at most $C_a^n$ such deterministic animals; $n=1$ is absorbed in the constant.

Fix an animal $\mathcal A$ and delete the sites whose determining cubes are within $\ell_{\rm dep}$ of the cube determining $Z_j$. At most $J_0$ sites are deleted, uniformly in $j,k$. One residue class among the remaining sites has at least
\[
 \frac{n-1-J_0}{K_0^d}
\]
sites. The selected cubes and the cube determining $Z_j$ are mutually separated. Applying finite-range dependence to each deterministic cube and the union of the others gives joint independence, whence
\[
 \mathbb E\left[Z_j^d\prod_{i\in\mathcal A}I_i\right]
 \le M_1p_\kappa^{(n-1-J_0)/K_0^d}.
\]
Absorb small $n$ into the constant and sum over the animals:
\[
 \mathbb E\left[Z_j^d\mathbf1_{\{H_k\ge n\}}\right]
 \le CM_1 C_{a}^{\,n}p_\kappa^{n/K_0^d}.
\]
Choose $\kappa$ so that $C_ap_\kappa^{1/K_0^d}<1$. Independence was used only after enumerating deterministic animals.
\end{proof}

\begin{proposition}[Critical moment and localization of a full physical coarse step]
\label{clock:critical-step}
Let
\begin{equation}
 \mathcal T=\int_0^S h(X_t)\,dt,\qquad \xi(x)=\mathbf E_x\mathcal T.
 \label{clock:step-time}
\end{equation}
Then, for every deterministic starting point $x$,
\begin{equation}
 \mathbb E\mathbf E_x\mathcal T^d\le C,
 \qquad \mathbb E\xi(x)^d\le C.
 \label{clock:critical-step-moment}
\end{equation}
Moreover, there is a local approximation $\xi^{(L)}(x)$ that depends only on the coefficients in $B_L(x)$ and satisfies
\begin{equation}
 \mathbb E|\xi(x)-\xi^{(L)}(x)|^d\le Ce^{-cL}.
 \label{clock:critical-localization}
\end{equation}
These conclusions neither use nor imply a moment of order $d+\eta$ for the physical coarse-step clock.
\end{proposition}

\begin{proof}
If $x\in Q_k$, the full-path control and $h\le\lambda_{\min}(A)^{-1}$ give
\[
 \mathcal T\le S\max_{j\in\mathcal J_{H_k}(k)}Z_j,
 \qquad
 \mathcal J_n(k)=\{j:|j-k|_\infty\le Cn\}.
\]
Combining this with~\eqref{clock:whole-cell-step},
\begin{equation}
 \mathbf E_x\mathcal T^d
 \le CH_k^{2d}\max_{j\in\mathcal J_{H_k}(k)}Z_j^d.
 \label{clock:whole-cell-physical-bound}
\end{equation}
Decompose according to $H_k$, bound the maximum by the sum and apply Lemma~\ref{clock:weighted-animal}:
\begin{align}
 \mathbb E\left[\mathbf E_x\mathcal T^d\,
       \mathbf1_{\{H_k\ge N\}}\right]
 &\le C\sum_{n\ge N}n^{2d}
      \sum_{j\in\mathcal J_n(k)}
       \mathbb E[Z_j^d\mathbf1_{\{H_k\ge n\}}]\notag\\
 &\le CM_1\sum_{n\ge N}n^{3d}e^{-cn}
 \le Ce^{-c'N}.                                         \label{clock:weighted-tail-sum}
\end{align}
Taking $N=1$ and then applying quenched Jensen's inequality yields~\eqref{clock:critical-step-moment}.

Inspect the component and its outer neighboring layer inside $B_L(x)$. If it is certified to lie in $B_{L-C}(x)$, set $\xi^{(L)}=\xi$ by computing the full coarse step; otherwise set it to zero. This construction depends only on the coefficients in $B_L(x)$ and differs from $\xi$ only when $H_k\ge cL$. Since $|\xi-\xi^{(L)}|^d\le\mathbf E_x\mathcal T^d$ on failure, \eqref{clock:weighted-tail-sum} proves~\eqref{clock:critical-localization}.

The weighted-animal estimate decouples the physical weight and the large cluster at order $d$; applying H\"older directly to their product would require an extra moment.
\end{proof}

\subsection{A coarse reverse H\"older inequality}
\label{green:grid-and-geometry}

Choose $h_0=1/N_0$, where $N_0$ is a fixed large integer, so that the fine-cell diameter satisfies
\[
 d_0=h_0\sqrt d\le a_c/8.
\]
Take a half-open fine-grid partition $\mathcal Q$ nested in the unit grid. For each fine cell $Q\subset Q_{k(Q)}$, set
\begin{equation}
 H_Q=H_{k(Q)},\qquad J_Q=C_{g}H_Q\ge1,
 \label{green:grid-weight}
\end{equation}
where the fixed constant $C_{g}$ is sufficiently large that, for all $x\in Q$,
\begin{equation}
 \int|y-x|^2K(x,dy)\le J_Q^2,
 \qquad \sup_{0\le t\le S}|X_t-x|\le J_Q
 \quad\mathbf P_x\text{-almost surely}.
 \label{green:grid-second-moment}
\end{equation}
The same $J_Q$ also controls the diameter of every entire bad component intersecting $Q$. The half-open partition avoids double counting when a Green atom lies on a cell boundary.

\begin{lemma}[Concentration of a fixed high-order geometric moment]
\label{green:geometry-moment}
There are deterministic $C_0,c,C,\vartheta_0>0$ such that, for every deterministic center $z$ and sufficiently large $r$,
\begin{equation}
 \mathbb P\left(
  \sum_{Q\cap B_{3r}(z)\ne\varnothing}J_Q^{4d}>C_0r^d
 \right)\le C\exp(-cr^{\vartheta_0}).
 \label{green:geometry-moment-event}
\end{equation}
Moreover, for every deterministic bounded region $E$ and $t\ge1$,
\begin{equation}
 \mathbb P\left(\max_{Q\cap E\ne\varnothing}J_Q>t\right)
 \le C(1+|E+B_1|)e^{-ct}
 \label{green:geometry-maximum}
\end{equation}
where the volume on the right is interpreted through the count of a fixed grid box containing the region. The constants are independent of the mesoscopic exponent $v$.
\end{lemma}

\begin{proof}
Use Proposition~\ref{prop:reg:geometry-concentration} with $m=4d$
and radius $3r/2$. Its weights are fixed multiples of finite-neighborhood
maxima of $1+H_k$, and therefore dominate the present weights
$J_Q=C_gH_{k(Q)}$, after a fixed change of constants. The grids can be
chosen identical. Thus~\eqref{reg:eq:geometry-concentration} gives
\eqref{green:geometry-moment-event}; in particular, its exponent is fixed
before $v$. The maximum bound follows by summing
$\mathbb P(J_Q>t)\le Ce^{-ct}$ over the deterministic cells meeting $E$.
\end{proof}

For a nonnegative locally finite measure $\nu$, define its fine-grid density by
\begin{equation}
 n_Q=\frac{\nu(Q)}{|Q|},\qquad
 n(x)=\sum_{Q\in\mathcal Q}n_Q\mathbf1_Q(x).
 \label{green:grid-density}
\end{equation}
Given an open domain $D$, the required adjoint condition is
\begin{equation}
 \int_{\mathbb R^d}\mathcal L\phi\,d\nu\le0
 \quad\text{for every nonnegative }\phi\in C_c^2(D).
 \label{green:adjoint-inequality}
\end{equation}
The integral must be well defined. Below, either $\mathcal L\phi$ has bounded support or $\nu$ is a finite killed Green measure.

\medskip\noindent\textit{Monge--Amp\`ere duality.}
\label{green:weighted-rh-subsection}

\begin{lemma}[Weighted annular estimate]
\label{green:weighted-rh}
Fix a center, which is suppressed in the notation below. Let $r\ge r_0$, $r\le s<t\le2r$, write $b=t-s$, and assume $b\ge\max\{8d_0,b_0\}$. Suppose that the geometry of the relevant bad components ensures
\begin{equation}
 K(x,B_{s+b/2})=0\qquad(x\notin B_t),
 \label{green:no-incoming}
\end{equation}
and that the nonnegative locally finite measure $\nu$ satisfies~\eqref{green:adjoint-inequality} for every nonnegative $\phi\in C_c^2(B_{s+b/2})$. With $q=d/(d-1)$, one has
\begin{equation}
 \left(\sum_{Q\cap B_s\ne\varnothing}|Q|n_Q^q\right)^{1/q}
 \le C\frac r{(t-s)^2}
       \sum_{Q\cap B_t\ne\varnothing}\nu(Q)J_Q^2.
 \label{green:weighted-rh-bound}
\end{equation}
For the actual coarse kernel, condition~\eqref{green:no-incoming} holds if all cell weights intersecting $B_{3r}$ are at most $cb$, where $c>0$ is fixed sufficiently small.
\end{lemma}

\begin{proof}
If a coarse endpoint $z'$ lies in $B_{s+b/2}$, its second-stage center satisfies $|Y-z'|=1$. For a good start the jump has length one. For a bad start, $Y\in\partial C_x$, so $C_x$ meets $B_{s+b/2+2}$ and its entire diameter is bounded by the nearby cell weights. Thus $|x-z'|\le cb+2<b/2$ for small $c$ and large $b_0$, which excludes $x\notin B_t$. This proves the incoming-jump assertion, including distant starting points.

Given a nonnegative array $(g_Q)$ supported on cells meeting $B_s$, sum smooth bumps equal to one on their closed cells, with support in fixed dilations. Finite overlap gives
\begin{equation}
 \widetilde g\ge g_Q\ \text{on }Q,\qquad
 \|\widetilde g\|_{L^d}\le CG,\qquad
 G=\left(\sum_Q|Q|g_Q^d\right)^{1/d}.
 \label{green:smooth-array}
\end{equation}

For $\epsilon>0$, solve the convex Monge--Amp\`ere problem in $B_{8r}$,
\begin{equation}
 \det D^2v_\epsilon=(\widetilde g+\epsilon)^d,\qquad
 v_\epsilon=0\quad\text{on }\partial B_{8r}.
 \label{green:ma-problem}
\end{equation}
The smooth positive right-hand side and strictly convex ball satisfy the hypotheses of Caffarelli--Nirenberg--Spruck~\cite[Theorem~1.1]{CNS}, yielding a smooth strictly convex solution. Set $m_0=-\min v_\epsilon$, attained at $x_0$. If $|p|<m_0/(32r)$, the minimum of $v_\epsilon-p\cdot(x-x_0)$ lies in the interior. Hence $B_{m_0/(32r)}\subset Dv_\epsilon(B_{8r})$. By strict convexity and the area formula,
\[
 c(m_0/r)^d\le |Dv_\epsilon(B_{8r})|
 =\int_{B_{8r}}\det D^2v_\epsilon
 =\int_{B_{8r}}(\widetilde g+\epsilon)^d
 \le C(G^d+\epsilon^dr^d).
\]
Since $v_\epsilon\le0$, this proves the first estimate below:
\begin{equation}
 \|v_\epsilon\|_{L^\infty(B_{8r})}\le Cr(G+\epsilon r),\qquad
 \|Dv_\epsilon\|_{L^\infty(B_{6r})}\le C(G+\epsilon r).
 \label{green:ma-amplitude}
\end{equation}
For the second, take $e=Dv_\epsilon(x)/|Dv_\epsilon(x)|$ when the gradient is nonzero. Convexity gives, for $x\in B_{6r}$,
\[
 r|Dv_\epsilon(x)|\le v_\epsilon(x+re)-v_\epsilon(x)
 \le2\|v_\epsilon\|_\infty.
\]
Extend $v_\epsilon$ by the supremum of its supporting affine functions at points of $B_{5r}$. The extension is globally convex, agrees with $v_\epsilon$ in $B_{5r}$ and is $C(G+\epsilon r)$-Lipschitz. Adding a constant gives a function $V$ satisfying
\begin{equation}
 V\ge0\ \text{on }B_{3r},\qquad
 \sup_{B_{3r}}|V|\le Cr(G+\epsilon r),\qquad
 \operatorname{Lip}V\le C(G+\epsilon r).
 \label{green:ma-extension}
\end{equation}
Adding a constant does not change $\mathcal LV$.

We need the lower bound to hold at every starting point in each cell. If $x\in Q$, then $Q\subset B_{a_c/4}(x)$. By the arithmetic--geometric mean inequality,
\[
 \Delta v_\epsilon\ge d(\det D^2v_\epsilon)^{1/d}
 \ge d\widetilde g\ge dg_Q\quad\text{on }Q.
\]
Let $M_x(\rho)$ denote the spherical mean. The divergence theorem gives
\[
 M_x'(\rho)=\frac1{d|B_1|\rho^{d-1}}
           \int_{B_\rho(x)}\Delta v_\epsilon.
\]
Consequently,
\begin{align}
 \frac1{|B_{a_c}|}\int_{B_{a_c}(x)}v_\epsilon-v_\epsilon(x)
 &=\int_0^{a_c}\left(1-\frac{\rho^d}{a_c^d}\right)M_x'(\rho)\,d\rho\notag\\
 &\ge c_d|Q|a_c^{2-d}g_Q\ge c g_Q.                 \label{green:all-points-ma}
\end{align}
Here we integrate only over $\rho\in[a_c/2,3a_c/4]$, for which the whole cell $Q$ lies in $B_\rho(x)$. By the convex-order comparison~\eqref{green:convex-order-input},
\begin{equation}
 \mathcal LV(x)\ge cg_Q\quad(x\in Q).
 \label{green:all-points-generator}
\end{equation}
At all other points, zero drift and Jensen's inequality still give $\mathcal LV\ge0$.

Choose $0\le\eta\le1$ such that $\eta=1$ on $B_{s+d_0}$, $\eta\in C_c^\infty(B_{s+b/2})$, and
\[
 |D\eta|\le Cb^{-1},\qquad |D^2\eta|\le Cb^{-2}.
\]
Since $V$ is smooth and nonnegative on the support of the cutoff, $\eta V$ is an admissible nonnegative adjoint test. The exact product identity is
\begin{align}
 \mathcal L(\eta V)(x)
 &=\eta(x)\mathcal LV(x)+V(x)\mathcal L\eta(x)\notag\\
 &\quad+\int(\eta(y)-\eta(x))(V(y)-V(x))K(x,dy).
 \label{green:nonlocal-product}
\end{align}
Zero drift cancels the first-order Taylor term of $\eta$, so, for $x\in Q$,
\[
 |\mathcal L\eta(x)|\le Cb^{-2}J_Q^2,\qquad
 \left|\int(\eta(y)-\eta(x))(V(y)-V(x))K(x,dy)\right|
 \le Cb^{-1}(G+\epsilon r)J_Q^2.
\]
Inside $B_t$, $|V|\le Cr(G+\epsilon r)$; since $b\le r$, the sum of the two errors is bounded by
\begin{equation}
 C\frac r{b^2}(G+\epsilon r)J_Q^2
 \label{green:cutoff-error}
\end{equation}
For $x\notin B_t$, one has $\eta(x)=0$, and the incoming-jump condition excludes entry into the cutoff support, so $\mathcal L(\eta V)(x)=0$. This ensures that integration against the locally finite measure is needed only over a bounded region.

Substituting~\eqref{green:all-points-generator}--\eqref{green:cutoff-error} into the adjoint inequality yields
\[
 c\sum_Q\nu(Q)g_Q
 \le\int\eta\mathcal LV\,d\nu
 \le C\frac r{b^2}(G+\epsilon r)
       \sum_{Q\cap B_t\ne\varnothing}\nu(Q)J_Q^2.
\]
Let $\epsilon\downarrow0$ in this inequality and use finite-array $L^d$ duality. This gives~\eqref{green:weighted-rh-bound}; the cutoff equals one on every cell on the left.
\end{proof}

\medskip\noindent\textit{Absorption and the common event.}
\label{green:hole-filling-subsection}

\begin{lemma}[Unweighted absorption under a fixed moment of order $4d$]
\label{green:hole-filling}
There are fixed $C_0,C_1,H<\infty$ and a deterministic initial scale for which the following holds. If
\begin{equation}
 \sum_{Q\cap B_{3r}(z)\ne\varnothing}J_Q^{4d}\le C_0r^d,
 \qquad
 \max_{Q\cap B_{3r}(z)\ne\varnothing}J_Q\le c r^{1/2},
 \label{green:local-geometric-good}
\end{equation}
and $\nu$ satisfies~\eqref{green:adjoint-inequality} for nonnegative tests in $B_{2r}(z)$, then
\begin{equation}
 \left(\sum_{Q\cap B_r(z)\ne\varnothing}|Q|n_Q^q\right)^{1/q}
 \le \frac Hr\sum_{Q\cap B_{2r}(z)\ne\varnothing}\nu(Q),
 \qquad q=\frac d{d-1}.
 \label{green:grid-critical-rh}
\end{equation}
The deterministic constant $H$ is independent of the mesoscopic exponent chosen later and of the exponent in the probability bound.
\end{lemma}

\begin{proof}
Translate $z$ to zero and absorb the fixed fine-cell volume. Write
\[
 F(s)=\left(\sum_{Q\cap B_s\ne\varnothing}\nu(Q)^q\right)^{1/q},
 \qquad M=\sum_{Q\cap B_{2r}\ne\varnothing}\nu(Q).
\]
Then $F(s)\le M$ for $s\le2r$. For $b=t-s\ge C_1r^{1/2}$ and sufficiently large fixed $C_1$, \eqref{green:local-geometric-good} supplies the incoming-jump condition and cell-width margins. Lemma~\ref{green:weighted-rh} gives
\begin{equation}
 F(s)\le C\frac r{b^2}
          \sum_{Q\cap B_t\ne\varnothing}\nu(Q)J_Q^2.
 \label{green:annular-weighted}
\end{equation}
Split at $J_Q=T\ge1$. The lower part is bounded by $T^2M$; H\"older's inequality bounds the upper part by
\begin{align*}
 \sum_{Q\cap B_t\ne\varnothing}\nu(Q)J_Q^2\mathbf1_{\{J_Q>T\}}
 &\le F(t)\left(\sum_{Q\cap B_t\ne\varnothing}
                  J_Q^{2d}\mathbf1_{\{J_Q>T\}}\right)^{1/d}\\
 &\le T^{-2}F(t)\left(\sum_{Q\cap B_{3r}\ne\varnothing}J_Q^{4d}\right)^{1/d}
 \le CrT^{-2}F(t).
\end{align*}
Thus
\[
 F(s)\le C\frac r{b^2}T^2M+C(r/b)^2T^{-2}F(t).
\]
Given a fixed $0<\theta<1/64$, take $T=C\theta^{-1/2}r/b$ to obtain
\begin{equation}
 F(s)\le\theta F(t)+C_\theta\frac Mr(r/b)^4.
 \label{green:hole-filling-step}
\end{equation}

Set $s_j=2r-r2^{-j}$. Let $N$ be the largest integer for which $s_{j+1}-s_j=r2^{-j-1}\ge C_1r^{1/2}$ holds for every $j<N$. After increasing the fixed initial scale,
\[
 N\ge\tfrac12\log_2r-C.
\]
Iterating~\eqref{green:hole-filling-step} finitely many times gives
\[
 F(r)\le\theta^NF(s_N)+C_\theta\frac Mr
                \sum_{j=0}^{N-1}(16\theta)^j.
\]
The geometric series bounds the second term by $CM/r$. For the first term, use $F(s_N)\le M$ and
\[
 \theta^N\le C r^{-\frac12\log_2(1/\theta)}\le Cr^{-1}
\]
Restoring the cell volume gives~\eqref{green:grid-critical-rh}. The terminal bound uses only finitely many cell masses, not a preexisting global $L^q$ estimate.
\end{proof}

\begin{proposition}[A common event for all mesoscopic subballs]
\label{green:common-event}
One may fix $\kappa_0=16$. Given a fixed sufficiently large $C_*>1$ and any $v\in(0,1/2]$, set $D_R=B_{C_*R}(z_0)$ and $L=R^v$. There is an event $\mathcal G_{R,v}(z_0)$ depending only on the coarse geometry near the fixed large domain such that, on this event, for every nonnegative measure satisfying~\eqref{green:adjoint-inequality} in $D_R$, simultaneously for all real centers $x$ and radii $r$ satisfying
\[
 r\ge L,\qquad B_{\kappa_0r}(x)\subset D_R
\]
one has
\begin{equation}
 \left(\frac1{|B_r|}\int_{B_r(x)}n^q\right)^{1/q}
 \le H_0\frac1{|B_{\kappa_0r}|}\int_{B_{\kappa_0r}(x)}n.
 \label{green:mesoscopic-rh}
\end{equation}
Here $H_0$ is independent of $v,R,\nu,x,r$. The event can also ensure that every full coarse path starting anywhere in $D_R$ stays within distance $L$ of its starting point. There is a fixed $\vartheta_0>0$ such that
\begin{align}
 \mathbb P(\mathcal G_{R,v}(z_0)^c)
 &\le CR^d(1+\log R)e^{-cR^{v\vartheta_0}}
       +CR^de^{-cR^{v/2}}                                  \label{green:two-probability-exponents}\\
 &\le C_v\exp\{-c_vR^{v\eta_0}\},
 \qquad \eta_0=\min\{\vartheta_0,1/2\},
 \quad R\ge R_0(v).                                       \label{green:common-event-probability}
\end{align}
Here $z_0$ is a fixed deterministic center; no assertion is made that a single random threshold works simultaneously for every $z_0$ in the whole space.
\end{proposition}

\begin{proof}
For the translated deterministic integer centers
\[
 z\in z_0+\mathbb Z^d,\qquad |z-z_0|\le2C_*R
\]
and the radii $s=2^jL$, $L\le s\le2C_*R$, require
\[
 \sum_{Q\cap B_{3s}(z)\ne\varnothing}J_Q^{4d}\le C_0s^d.
\]
Also require
\begin{equation}
 \max_{Q\cap B_{8C_*R}(z_0)\ne\varnothing}J_Q\le cL^{1/2},
 \label{green:global-maximum-event}
\end{equation}
where the fixed constant $c$ is sufficiently small. Each weight controls its entire intersecting component, including portions outside the killing domain. There are at most $CR^d(1+\log R)$ center--radius pairs. Lemma~\ref{green:geometry-moment} and \eqref{green:geometry-maximum} give the two terms in~\eqref{green:two-probability-exponents}. For fixed $v$, absorb the polynomial factors by increasing $R_0(v)$. The moment order, absorption exponent $1/2$, and constants $H_0,\eta_0$ are fixed independently of $v$.

For an admissible $B_r(x)$, take $L\ge2\sqrt d+4d_0$, a nearest grid center $z$, and the smallest dyadic radius $s$ satisfying
\[
 s\ge r+|z-x|+d_0.
\]
Since $|z-x|\le r/4$, we have $r\le s<3r$. Cells meeting $B_r(x)$ lie in $B_s(z)$; cells meeting $B_{2s}(z)$ lie in $B_{7r}(x)$; and $B_{3s}(z)$ with its cell margin lies in $B_{10r}(x)$. All are contained in $B_{16r}(x)\subset D_R$.

Since $J_Q\le cL^{1/2}\le cs^{1/2}$, Lemma~\ref{green:hole-filling} applies at $(z,s)$. Its inner sum controls $\int_{B_r(x)}n^q$ and its outer sum is bounded by $\int_{B_{16r}(x)}n$. Volume comparison yields~\eqref{green:mesoscopic-rh}.

Finally, \eqref{green:grid-second-moment} and \eqref{green:global-maximum-event} bound every full-path displacement starting in $D_R$ by $cL^{1/2}\le L$. The event was defined using a finite grid and no adjoint measure, so it applies to all real centers and all admissible measures simultaneously.
\end{proof}

\subsection{Self-improvement and the Green estimate}
\label{green:gehring-subsection}

The integrability gain uses the self-improvement principle of
Gehring~\cite{Gehring73}. Our reverse H\"older estimate is initially
available only for $r\ge L$. We first recover the smaller scales by
convolution, and then prove the local fixed-dilation version of the
self-improvement lemma. This keeps its gain independent of the
mesoscopic cutoff $L$.

Define the normalized tent kernel
\begin{equation}
 \varphi_L(x)=c_dL^{-d}(1-|x|/L)_+,\qquad
 \int_{\mathbb R^d}\varphi_L=1,\qquad
 f_L=\varphi_L*n.
 \label{green:tent-definition}
\end{equation}
All convolution estimates are used only on balls whose required neighborhoods lie inside $D_R$.

\begin{lemma}[Small-scale bounds for the mollified density]
\label{green:tent-rh}
For every nonnegative locally integrable $n$ and every $0<r\le L$,
\begin{equation}
 \sup_{B_r(x)}f_L\le C_d\frac1{|B_{2r}|}\int_{B_{2r}(x)}f_L.
 \label{green:tent-small-scale}
\end{equation}
If $n$ satisfies~\eqref{green:mesoscopic-rh}, then, for $r\ge L$ and $B_{\kappa r}(x)\subset D_R$, where
\begin{equation}
 \kappa=2\kappa_0+1=33,
 \label{green:dilation}
\end{equation}
one has
\begin{equation}
 \left(\frac1{|B_r|}\int_{B_r(x)}f_L^q\right)^{1/q}
 \le H_1\frac1{|B_{\kappa r}|}\int_{B_{\kappa r}(x)}f_L.
 \label{green:tent-large-scale}
\end{equation}
Thus $f_L$ satisfies a weak reverse H\"older inequality at every scale with the same fixed dilation factor $\kappa$. The constant $H_1$ is independent of $L$ and $v$.
\end{lemma}

\begin{proof}
For~\eqref{green:tent-small-scale}, rescale to $L=1$, fix $y$, and set $D_0=|y-x|$ and
\[
 t=(1-\max\{D_0-r,0\})_+.
\]
The kernel supremum is $c_dt$; assume $t>0$. If $D_0\le r\le1/4$, the kernel is at least $c_d/2$ on $B_r(x)$. If $D_0\le r$ and $r>1/4$, it is at least $3c_d/4$ on $B_{1/4}(y)\subset B_{2r}(x)$. In both cases this lower bound holds on a fixed volume fraction of $B_{2r}(x)$, since $r\le1$.

If $D_0>r$ and $t\ge4r$, then $D_0=1+r-t$, so, for $z\in B_{2r}(x)$,
\[
 1-|z-y|\ge1-D_0-2r=t-3r\ge t/4.
\]
Finally, if $D_0>r$ and $0<t<4r$, take $e=(y-x)/D_0$ and
\[
 B_{r/4}(x+3re/2)\subset B_{2r}(x).
\]
For every point $z$ in this ball,
\[
 |z-y|\le |D_0-3r/2|+r/4
 \le\max\{D_0-3r/2,r/2\}+r/4\le1-r/4.
\]
Here $D_0<1+r$ and $r\le1$. Thus the kernel is at least $c_dr/4\ge c_dt/16$ on a fixed fraction of $B_{2r}(x)$. In all cases,
\[
 \sup_{z\in B_r(x)}\varphi_1(z-y)
 \le C_d\frac1{|B_{2r}|}\int_{B_{2r}(x)}\varphi_1(z-y)\,dz.
\]
Integrating in $n(y)\,dy$ proves~\eqref{green:tent-small-scale}.

Now take $r\ge L$. Jensen's inequality for convolution, the support radius $L$, and~\eqref{green:mesoscopic-rh} give
\begin{align*}
 \left(\frac1{|B_r|}\int_{B_r(x)}f_L^q\right)^{1/q}
 &\le C\left(\frac1{|B_{r+L}|}\int_{B_{r+L}(x)}n^q\right)^{1/q}\\
 &\le CH_0\frac1{|B_{\kappa_0(r+L)}|}
                 \int_{B_{\kappa_0(r+L)}(x)}n.
\end{align*}
Since $\kappa_0(r+L)\le2\kappa_0r$, and every convolution kernel of radius $L\le r$ centered at a point of $B_{2\kappa_0r}(x)$ is entirely contained in $B_{(2\kappa_0+1)r}(x)$,
\[
 \int_{B_{(2\kappa_0+1)r}(x)}f_L
 \ge\int_{B_{2\kappa_0r}(x)}n.
\]
Volume comparison gives~\eqref{green:tent-large-scale}. For $r\le L$, \eqref{green:tent-small-scale} and the fixed enlargement $2r\mapsto\kappa r$ give the same reverse H\"older bound.
\end{proof}

\begin{lemma}[Local Gehring lemma with a fixed dilation factor]
\label{green:gehring}
Fix $q>1$, $\kappa\ge2$, and $H_1<\infty$, and let $f\ge0$ be bounded and integrable on $B_{3\kappa R}$. Suppose that, for every ball $B$ satisfying $\kappa B\subset B_{3\kappa R}$,
\begin{equation}
 \left(\frac1{|B|}\int_Bf^q\right)^{1/q}
 \le H_1\frac1{|\kappa B|}\int_{\kappa B}f.
 \label{green:weak-rh-input}
\end{equation}
Then there are $\epsilon_G>0$ and $C<\infty$, depending only on $d,q,\kappa,H_1$, such that
\begin{equation}
 \left(\frac1{|B_R|}\int_{B_R}f^{q+\epsilon_G}\right)^{1/(q+\epsilon_G)}
 \le C\frac1{|B_{3\kappa R}|}\int_{B_{3\kappa R}}f.
 \label{green:gehring-conclusion}
\end{equation}
\end{lemma}

\begin{proof}
Let $M=\int_{B_{3\kappa R}}f$. The case $M=0$ is immediate, so assume $M>0$. Take $R\le s<t\le2R$, set $g=t-s$, and choose
\[
 \lambda_0=C(d,\kappa)g^{-d}M,
\]
with the constant sufficiently large that every ball centered in $B_s$ with radius between $g/(100\kappa)$ and $g$ has average strictly less than $\lambda_0$.

For $\lambda\ge\lambda_0$ and a Lebesgue point $x\in B_s\cap\{f>\lambda\}$, continuity of ball averages gives a last radius $\rho_x\le g/(100\kappa)$ with average $\lambda$. Every larger radius up to $g$ has average at most $\lambda$; hence
\[
 \frac1{|5\kappa B_{\rho_x}|}\int_{5\kappa B_{\rho_x}(x)}f
 \le\lambda,\qquad 5\kappa B_{\rho_x}(x)\subset B_t.
\]

Vitali selection gives pairwise disjoint balls $B_i=B_{\rho_i}(x_i)$ whose fivefold dilations cover the above level set up to a null set. Applying~\eqref{green:weak-rh-input} to $5B_i$ yields
\[
 \int_{B_s\cap\{f>\lambda\}}f^q
 \le\sum_i\int_{5B_i}f^q
 \le C\lambda^q\sum_i|B_i|.
\]
Since $\int_{B_i}f=\lambda|B_i|$,
\[
 \frac\lambda2|B_i|
 \le\int_{B_i\cap\{f>\lambda/2\}}f.
\]
The balls $B_i$ are disjoint and contained in $B_t$, so
\begin{equation}
 \int_{B_s\cap\{f>\lambda\}}f^q
 \le C\lambda^{q-1}\int_{B_t\cap\{f>\lambda/2\}}f.
 \label{green:gehring-level-set}
\end{equation}

For $0<\epsilon\le1$ to be chosen, multiply~\eqref{green:gehring-level-set} by $\epsilon\lambda^{\epsilon-1}$ and integrate over $[\lambda_0,\infty)$. Since $f$ is bounded, all integrals are finite and Fubini's theorem applies directly. Write $H(s)=\int_{B_s}f^{q+\epsilon}$. Combining the left-hand side with the low-level contribution gives
\[
 H(s)\le\lambda_0^\epsilon\int_{B_s}f^q
  +C\frac{\epsilon\,2^{q+\epsilon-1}}{q+\epsilon-1}H(t)
 \le\lambda_0^\epsilon\int_{B_s}f^q+C\epsilon H(t).
\]
Apply~\eqref{green:weak-rh-input} on $B_{2R}$ to get
\[
 \int_{B_s}f^q\le CR^{-d(q-1)}M^q.
\]
Consequently,
\begin{equation}
 H(s)\le C\epsilon H(t)
       +C g^{-d\epsilon}R^{-d(q-1)}M^{q+\epsilon}.
 \label{green:gehring-hole-filling}
\end{equation}

Fix $\epsilon=\epsilon_G>0$ sufficiently small that $\theta=C\epsilon_G<1$ and $\theta2^{d\epsilon_G}<1$. Iterate~\eqref{green:gehring-hole-filling} along $s_j=2R-R2^{-j}$. Boundedness gives $\theta^NH(s_N)\to0$, and the remaining geometric series gives
\[
 H(R)\le C R^{-d(q+\epsilon_G-1)}M^{q+\epsilon_G}.
\]
Normalization proves~\eqref{green:gehring-conclusion}. For $f_L$, boundedness follows from finite mass in the convolution neighborhood; no higher integrability has been assumed.
\end{proof}

Lemmas~\ref{green:tent-rh} and~\ref{green:gehring} imply, provided the appropriate fixed dilation of the ball remains in $D_R$, that
\begin{equation}
 \left(\frac1{|B_\rho|}\int_{B_\rho(x)}f_L^{q'}\right)^{1/q'}
 \le C\rho^{-d}\int_{B_{C\rho}(x)}n,
 \qquad q'=q+\epsilon_G,\qquad \rho\ge L.
 \label{green:coarsened-high-integrability}
\end{equation}
All dilation factors here are fixed before $v$. Decreasing $\epsilon_G$ if necessary, we can arrange that its conjugate exponent
\begin{equation}
 p=\frac{q'}{q'-1}\quad\hbox{satisfies}\quad
 \max\{1,d/2\}<p<d.
 \label{green:source-exponent}
\end{equation}
The clock truncation below requires only $p<d$; retaining $p>d/2$ imposes no additional assumption.

\medskip\noindent\textit{The fine-grid Green bound.}
\label{green:internal-estimate-subsection}

\begin{lemma}[Fine-grid loss]
\label{green:fine-grid-loss}
Take mesoscopic cells $P$ nested with the fine grid and with side length comparable to $L/(4\sqrt d)$. Define their average densities by
\[
 N_P=\frac{\nu(P)}{|P|}.
\]
In a region where all relevant cells and their $L$-neighborhoods lie in the interior,
\begin{equation}
 N_P\le C f_L(x)\quad(x\in P),\qquad
 \|n\|_{L^{q'}(\cup P)}
 \le C L^{d/p}\left(\sum_P|P|N_P^{q'}\right)^{1/q'}.
 \label{green:fine-grid-loss-bound}
\end{equation}
Here $p$ is defined by~\eqref{green:source-exponent}.
\end{lemma}

\begin{proof}
Two points in the same $P$ are at distance at most $L/4$, so the tent kernel is at least $cL^{-d}$ at these displacements. Hence
\[
 f_L(x)\ge cL^{-d}\int_Pn\ge c'N_P\qquad(x\in P).
\]
For the second inequality, denote the common fine-cell volume by $|Q|$. Since $q'>1$,
\[
 \sum_{Q\subset P}\nu(Q)^{q'}\le\left(\sum_{Q\subset P}\nu(Q)\right)^{q'}
 =\nu(P)^{q'}.
\]
Thus
\begin{align*}
 \|n\|_{L^{q'}(\cup P)}^{q'}
 &=|Q|^{1-q'}\sum_P\sum_{Q\subset P}\nu(Q)^{q'}\\
 &\le (|P|/|Q|)^{q'-1}\sum_P|P|N_P^{q'}.
\end{align*}
Taking the root gives $L^{d(1-1/q')}=L^{d/p}$. This factor accounts for possible concentration in a single fine cell.
\end{proof}

\begin{proposition}[Interior Green bound for all poles in an enlarged killing domain]
\label{green:internal-green}
Fix $C_*$ sufficiently large; for example, $C_*=1024$ contains all fixed dilations used in this section. Let
\[
 D_R=B_{C_*R}(z_0),\qquad
 \tau_D=\inf\{n\ge0:Z_n\notin D_R\},
 \qquad Z_n=X_{S_n},
\]
where $S_n$ are the cumulative actual coarse stopping times. Define the Green measure killed at coarse endpoints by
\begin{equation}
 \nu^x(E)=\mathbf E_x\sum_{n<\tau_D}\mathbf1_E(Z_n).
 \label{green:killed-measure}
\end{equation}
For each fixed $v\in(0,1/2]$, on the same event as in Proposition~\ref{green:common-event}, simultaneously for all $x\in B_R(z_0)$,
\begin{equation}
 \nu^x(\mathbb R^d)\le CR^2,\qquad
 \|n^x\|_{L^{q'}(B_{2R}(z_0))}
 \le C L^{d/p}R^{2-d/p},\qquad L=R^v.
 \label{green:green-density-bound}
\end{equation}
Consequently, for every measurable source $F$ supported in $B_{2R}(z_0)$ and dominated on the fine cells by a nonnegative array $F_Q$, one has
\begin{equation}
 R^{-2}\sup_{x\in B_R(z_0)}\int |F|\,d\nu^x
 \le C L^{d/p}
 \left(R^{-d}\sum_{Q\cap B_{2R}(z_0)\ne\varnothing}|Q|F_Q^p\right)^{1/p}.
 \label{green:source-estimate}
\end{equation}
Enlarging the ball by a fixed cell width at the support boundary changes only fixed constants. This inequality holds simultaneously for all sources and all poles in the same environment; no independence between the Green measure and the source is required.
\end{proposition}

\begin{proof}
On the common geometric event, every coarse jump starting in $D_R$ has length at most $L\le R$, while~\eqref{green:covariance-lower} gives $\operatorname{tr}\Sigma\ge da_c$. For finite $N$, zero drift and the squared-martingale identity give
\[
 da_c\mathbf E_x(N\wedge\tau_D)
 \le\mathbf E_x|Z_{N\wedge\tau_D}-z_0|^2-|x-z_0|^2
 \le(C_*R+L)^2.
\]
The identity uses a bounded stopping step. Monotone convergence gives $\mathbf E_x\tau_D\le CR^2$, proving finiteness of $\nu^x$ and the first bound.

For every nonnegative $\phi\in C_c^2(D_R)$, finite-step telescoping, $\mathbf E_x\tau_D<\infty$, and boundedness of the test function give
\begin{equation}
 \int\mathcal L\phi\,d\nu^x
 =\mathbf E_x\phi(Z_{\tau_D})-\phi(x)
 =-\phi(x)\le0.
 \label{green:pole-sign}
\end{equation}
The pole has the required sign even inside the test ball; this is why the Monge--Amp\`ere test was chosen nonnegative. Apply Proposition~\ref{green:common-event} to every $\nu^x$ and then~\eqref{green:coarsened-high-integrability} on $B_{4R}(z_0)$. The outer ball $B_{396R}(z_0)$, including its convolution margin, lies in $D_R$. Since grid averaging preserves total mass,
\[
 \|f_L^x\|_{L^{q'}(B_{4R})}
 \le CR^{-d+d/q'}\nu^x(\mathbb R^d)
 \le CR^{2-d/p}.
\]
Apply Lemma~\ref{green:fine-grid-loss} on the mesoscopic cells covering $B_{2R}$, whose union lies in $B_{3R}$, to obtain the second bound.

Finally, integration cell by cell and H\"older's inequality with conjugate exponents $p,q'$ give
\[
 \int|F|\,d\nu^x
 \le\sum_QF_Q\nu^x(Q)
 \le\left(\sum_Q|Q|F_Q^p\right)^{1/p}
      \|n^x\|_{L^{q'}(B_{2R+O(h_0)})}.
\]
Equation~\eqref{green:green-density-bound} proves~\eqref{green:source-estimate}. All steps are deterministic on the common event, uniformly in $x$ and $F$.
\end{proof}

\subsection{Continuous sources and clock truncation}
\label{green:full-path-subsection}

To use the interior Green estimate, we first locate the support of the source accumulated over a full coarse path.

\begin{lemma}[Spatial support of a full-path source]
\label{green:full-path-support}
Let $E\subset B_R(z_0)$ be Borel and let $b\ge0$ be a Borel function. Define
\begin{equation}
 F_{E,b}(x)=\mathbf E_x\int_0^S\mathbf1_E(X_t)b(X_t)\,dt.
 \label{green:full-step-source}
\end{equation}
If all fine cells intersecting $E+B_3$ have geometric weight at most $L$, then
\begin{equation}
 F_{E,b}(x)=0\qquad\text{when }\operatorname{dist}(x,E)>CL.
 \label{green:source-support}
\end{equation}
The constant is independent of $b$. This assertion covers arbitrarily distant starting points simultaneously, without requiring a uniform global bound on geometric weights.
\end{lemma}

\begin{proof}
If a full path visits $E$ during its first stage, $C_x$ meets $E+B_1$, including a visit at its boundary exit. A cell in $E+B_3$ then controls the diameter of the entire component, giving $\operatorname{diam}C_x\le CL$ and $\operatorname{dist}(x,E)\le CL$.

For a second-stage visit, the unit-ball center $Y$ lies within distance one of $E$. If $x$ is good then $Y=x$. Otherwise $Y\in\partial C_x$, so $C_x$ meets $E+B_2$ and the same argument gives $\operatorname{dist}(x,E)\le CL+2$. Thus no full path from the stated distant points can visit $E$, proving~\eqref{green:source-support}.
\end{proof}

\begin{proposition}[Comparison with the coarse Green potential]
\label{green:continuous-comparison}
Let
\[
 \sigma_R=\inf\{t\ge0:X_t\notin B_R(z_0)\},\qquad
 D_R=B_{C_*R}(z_0),
\]
and retain $S_n,Z_n,\tau_D,\nu^x$ from Proposition~\ref{green:internal-green}. In every environment on the full-probability event where the coarse process is defined, for every nonnegative Borel function $b$ and every $x\in B_R(z_0)$,
\begin{equation}
 \mathbf E_x\int_0^{\sigma_R}b(X_t)\,dt
 \le \int F_{B_R(z_0),b}\,d\nu^x.
 \label{green:continuous-green-comparison}
\end{equation}
On the common geometric event, if $L=o(R)$, the source on the right is supported in $B_{R+CL}(z_0)\subset B_{2R}(z_0)$, so~\eqref{green:source-estimate} applies.
\end{proposition}

\begin{proof}
We first prove that the killed coarse chain has finite expected lifetime without imposing the geometric event of Proposition~\ref{green:internal-green}. Fix an environment in which all bad clusters are finite and the coarse process is defined. There are finitely many fine cells intersecting $D_R$, so
\[
 J_D:=\max\{J_Q:Q\cap D_R\ne\varnothing\}<\infty.
\]
By~\eqref{green:grid-second-moment}, every jump starting in $D_R$ has length at most $J_D$. For a finite integer $N$, zero drift and~\eqref{green:covariance-lower} therefore give
\begin{align*}
 d a_c\,\mathbf E_x(N\wedge\tau_D)
 &\le \mathbf E_x\sum_{n<N\wedge\tau_D}
                    \operatorname{tr}\Sigma(Z_n)\\
 &=\mathbf E_x|Z_{N\wedge\tau_D}-z_0|^2-|x-z_0|^2
 \le(C_*R+J_D)^2.
\end{align*}
Every stopped position belongs to $B_{C_*R+J_D}(z_0)$, which justifies the finite-step second-moment identity. Monotone convergence yields
\[
 \mathbf E_x\tau_D\le\frac{(C_*R+J_D)^2}{d a_c}<\infty.
\]
This is a pathwise finiteness statement; it requires no moment bound on $J_D$. In particular, $\nu^x$ is a finite measure and $\tau_D<\infty$ almost surely.

The coarse stopping times are finite and do not accumulate. By continuity, a path starting in $B_R(z_0)$ must exit this ball before a coarse endpoint can lie outside $D_R$. Hence
\[
 \sigma_R\le S_{\tau_D}.
\]
Before $\sigma_R$, one has $\mathbf1_{B_R(z_0)}(X_t)=1$ almost everywhere with respect to time. Nonnegativity therefore gives the pathwise inequality
\[
 \int_0^{\sigma_R}b(X_t)\,dt
 \le\int_0^{S_{\tau_D}}\mathbf1_{B_R(z_0)}(X_t)b(X_t)\,dt.
\]
Since $\{n<\tau_D\}$ is measurable at $S_n$, Tonelli's theorem and the strong Markov property give
\begin{align*}
 \mathbf E_x\int_0^{S_{\tau_D}}\mathbf1_{B_R(z_0)}(X_t)b(X_t)\,dt
 &=\mathbf E_x\sum_{n<\tau_D}
    \int_{S_n}^{S_{n+1}}\mathbf1_{B_R(z_0)}(X_t)b(X_t)\,dt\\
 &=\mathbf E_x\sum_{n<\tau_D}F_{B_R(z_0),b}(Z_n)
 =\int F_{B_R(z_0),b}\,d\nu^x.
\end{align*}
This proves~\eqref{green:continuous-green-comparison}, with both sides interpreted in $[0,\infty]$.

Excursions out of $D_R$ within a segment, and reentries into $B_R$ after first exit, only increase the nonnegative right-hand side. Thus endpoint killing need not agree with continuous killing. On the common event $J_Q\le cL^{1/2}\le L$ near $B_R$, so Lemma~\ref{green:full-path-support} gives the support assertion once $CL\le R$.
\end{proof}

\medskip\noindent\textit{Clock truncation.}
\label{clock:truncation-subsection}

For $T\ge1$, define
\begin{equation}
 h_T=\min\{h,T\},\qquad
 e_T(x)=\mathbf E_x\int_0^S(h-h_T)(X_t)\,dt.
 \label{clock:pointwise-truncation}
\end{equation}
We truncate the clock density in normalized time, rather than the duration of a full coarse step.

For each fine cell $Q\subset Q_{k(Q)}$, retain the fixed sets from~\eqref{clock:whole-cell-physical-bound} and write
\begin{equation}
 Z_{\max,Q}=\max_{j\in\mathcal J_{H_Q}(k(Q))}Z_j,
 \qquad U_Q^T=CH_Q^2 Z_{\max,Q}\mathbf1_{\{Z_{\max,Q}>T\}}.
 \label{clock:measurable-cell-tail}
\end{equation}
These cell variables are measurable, being finite maxima indexed by a measurable integer-valued weight.

\begin{lemma}[Algebraic decay of lower-order moments of the whole-cell clock tail]
\label{clock:tail-moment}
The constant in~\eqref{clock:measurable-cell-tail} can be fixed so that, for every environment and every $x\in Q$,
\begin{equation}
 0\le e_T(x)\le U_Q^T.
 \label{clock:cell-tail-dominance}
\end{equation}
For each fixed $1\le p<d$,
\begin{equation}
 \sup_Q\mathbb E(U_Q^T)^p\le C_pT^{p-d},\qquad T\ge1.
 \label{clock:cell-tail-moment}
\end{equation}
\end{lemma}

\begin{proof}
Along the full path $h\le Z_{\max,Q}$ and $\mathbf E_xS\le CH_Q^2$, uniformly for $x\in Q$. The integrand vanishes if $Z_{\max,Q}\le T$ and is otherwise at most $h$. This proves~\eqref{clock:cell-tail-dominance}.

Since $p-d<0$, for every $Z\ge0$,
\[
 Z^p\mathbf1_{\{Z>T\}}\le T^{p-d}Z^d.
\]
Decomposing according to $H_Q$, bounding the $d$th power of the maximum by the sum, and applying Lemma~\ref{clock:weighted-animal} gives
\begin{align*}
 \mathbb E(U_Q^T)^p
 &\le CT^{p-d}\sum_{n\ge1}n^{2p}
      \sum_{j\in\mathcal J_n(k(Q))}
       \mathbb E[Z_j^d\mathbf1_{\{H_Q\ge n\}}]\\
 &\le CT^{p-d}\sum_{n\ge1}n^{2p+d}e^{-cn}
 \le C_pT^{p-d}.
\end{align*}
The decay $T^{p-d}$ uses only the $d$th physical moment and $p<d$. At $p=d$ the calculation gives boundedness alone.
\end{proof}

Define the exit potentials for the normalized diffusion in the ball by
\begin{equation}
 U_R^T(x)=\mathbf E_x\int_0^{\sigma_R}h_T(X_t)\,dt,
 \qquad U_R(x)=\mathbf E_x\int_0^{\sigma_R}h(X_t)\,dt.
 \label{clock:torsions}
\end{equation}
For each fixed environment, local ellipticity implies that these exit potentials are finite. They solve, respectively,
\[
 -B:D^2U_R^T=h_T,\qquad -B:D^2U_R=h,\qquad
 U_R^T=U_R=0\quad\text{on }\partial B_R(z_0).
\]
In particular, $-A:D^2U_R=1$, so $U_R$ is the torsion function of the original physical operator.

\begin{proposition}[Algebraic probability bound for the clock-truncation error]
\label{clock:tail-probability}
Take the fixed exponent $p<d$ from Proposition~\ref{green:internal-green}. For every fixed sufficiently small $v>0$, every $\beta>0$, $T\ge1$, and sufficiently large $R$,
\begin{equation}
 \mathbb P\left(
 R^{-2}\sup_{B_R(z_0)}(U_R-U_R^T)>R^{-\beta}\right)
 \le C_v e^{-c_vR^{v\eta_0}}
      +C R^{vd+p\beta}T^{p-d}.
 \label{clock:tail-probability-bound}
\end{equation}
The deterministic constants and initial scale may depend on the already fixed $v$, but not on $T$. No independence of the truncated source from the environment, the Green measure, or the starting point is required.
\end{proposition}

\begin{proof}
Set
\[
 g_{R,T}(y)=\mathbf1_{B_R(z_0)}(y)(h(y)-h_T(y)),\qquad
 F_{R,T}(x)=\mathbf E_x\int_0^Sg_{R,T}(X_t)\,dt.
\]
Proposition~\ref{green:continuous-comparison} gives
\[
 (U_R-U_R^T)(x)\le\int F_{R,T}\,d\nu^x.
\]
On the common geometric event, Lemma~\ref{green:full-path-support} gives
\[
 \operatorname{supp}F_{R,T}\subset B_{R+CL}(z_0)\subset B_{2R}(z_0),
 \qquad L=R^v.
\]
Moreover, $F_{R,T}\le e_T$. We may therefore dominate $F_{R,T}$ by $U_Q^T$ on its support and set the dominating source to zero elsewhere. Applying~\eqref{green:source-estimate} yields
\begin{equation}
 R^{-2}\sup_{B_R(z_0)}(U_R-U_R^T)
 \le C R^{vd/p}
 \left(R^{-d}\sum_{Q\cap B_{2R}(z_0)\ne\varnothing}|Q|(U_Q^T)^p\right)^{1/p}.
 \label{clock:deterministic-truncation}
\end{equation}
The right-hand side is a measurable finite sum controlling every starting point. Raise it to the $p$th power and use Markov's inequality, \eqref{clock:cell-tail-moment}, and the $O(R^d)$ total cell volume:
\begin{align*}
 &\mathbb P\left(
 R^{-2}\sup_{B_R}(U_R-U_R^T)>R^{-\beta},\,
 \mathcal G_{R,v}(z_0)\right)\\
 &\hspace{2cm}\le
 C R^{vd+p\beta}\,
 R^{-d}\sum_{Q\cap B_{2R}\ne\varnothing}|Q|\mathbb E(U_Q^T)^p
 \le C R^{vd+p\beta}T^{p-d}.
\end{align*}
Add the failure probability of $\mathcal G_{R,v}$. Continuity of the potentials makes the supremum measurable through a countable dense set. The spatial cutoff must precede the bound by $e_T$: without it the source would reach the killing boundary, outside the scope of the interior Green estimate.
\end{proof}

\subsection{The clock mean}
\label{clock:bounded-source-interface}

The invariant density $m$ from Section~\ref{sec:density} satisfies
\[
 \mathbb Em(0)=1,\qquad \overline A=\mathbb E[m(0)A(0)].
\]
By Lemma~\ref{down:density-normalized-law}, the ergodic invariant probability for the environment seen by the normalized diffusion is
\begin{equation}
 \pi(d\omega)=\frac{a(0,\omega)m(0,\omega)}{\operatorname{tr}\overline A}
               \mathbb P(d\omega).
 \label{clock:direction-invariant-law}
\end{equation}
Hence
\begin{equation}
 \mathbb E_\pi B=\overline B:=\frac{\overline A}{\operatorname{tr}\overline A},
 \qquad
 \mu:=\mathbb E_\pi h=\frac1{\operatorname{tr}\overline A}<\infty.
 \label{clock:actual-means}
\end{equation}
The identities use $aB=A$ and $ah=1$; in particular $\mu$ is an occupation mean, not a spatial mean under $\mathbb P$.

By Theorem~\ref{ren:bounded-source} and Corollary~\ref{ren:source-amplitude}, there are fixed $\delta_B,\gamma_B>0$ and $C,R_B<\infty$ such that the following holds. Suppose $f$ is stationary jointly with $B$, is determined by the original coefficients within a common bounded radius, is quenched locally H\"older continuous, and satisfies $|f|\le1$. The zero-boundary solution
\[
 -B:D^2v=f-\mathbb E_\pi f\quad\text{in }B_R(z_0)
\]
satisfies
\begin{equation}
 \mathbb P(R^{-2}\|v\|_\infty>CR^{-\delta_B})
 \le CR^{-\gamma_B},\qquad R\ge R_B.
 \label{clock:bounded-source-input}
\end{equation}
The constants and initial scale are uniform over this normalized class, although the event may depend on the deterministic choice of $f$. Their dependence includes the range of the original field and the determining radius of $f$, even when $B$ is constant.

The sources $f_T=h_T/T$ share this determining radius and satisfy $|f_T|\le1$. They remain locally H\"older continuous in each environment; the theorem requires no uniform deterministic H\"older constant. Set
\begin{equation}
 \mu_T=\mathbb E_\pi h_T,\qquad
 \Psi_R(x)=\frac{R^2-|x-z_0|^2}{2}.
 \label{clock:truncated-mean}
\end{equation}
Since $\operatorname{tr}B=1$, one has $-B:D^2\Psi_R=1$. Consequently,
\[
 T^{-1}(U_R^T-\mu_T\Psi_R)
\]
is precisely the zero-boundary solution in~\eqref{clock:bounded-source-input} corresponding to $f_T$. We obtain
\begin{equation}
 \mathbb P\left(
 \sup_{B_R(z_0)}|R^{-2}U_R^T-\mu_T R^{-2}\Psi_R|
 >CT R^{-\delta_B}\right)
 \le CR^{-\gamma_B},\qquad R\ge R_B,\quad T\ge1.
 \label{clock:uniform-truncated-torsion}
\end{equation}
The error bound is proportional to $T$, while the probability bound and the initial scale $R_B$ are independent of $T$. This uniformity is needed for $R=T^b$ and $T=R^t$. In Section~\ref{sec:renormalization}, the common initial moments and dependence range fix the contraction scale before the deterministic centering constant is identified; the identification introduces no source-dependent threshold.

\medskip\noindent\textit{Convergence of the truncated means.}
\label{clock:mean-tail-subsection}

Integrability gives $\mu_T\uparrow\mu$, but no power rate. We obtain one by comparing the Green truncation bound with~\eqref{clock:uniform-truncated-torsion}.

\begin{proposition}[Algebraic convergence of the truncated means]
\label{clock:mean-tail}
There are fixed $s>0$ and $C<\infty$ such that
\begin{equation}
 0\le\mu-\mu_T\le CT^{-s},\qquad T\ge1.
 \label{clock:mean-tail-bound}
\end{equation}
One may first choose
\begin{equation}
 b>\frac4{\min\{\delta_B,\gamma_B\}},\qquad
 s=\min\left\{\frac{d-p}{4p},b\delta_B-1\right\}>0.
 \label{clock:mean-parameters}
\end{equation}
\end{proposition}

\begin{proof}
Decrease $\delta_B,\gamma_B$ to at most one and fix $b$ as in~\eqref{clock:mean-parameters}. For each deterministic $T$, set
\[
 R=T^b.
\]
With the Green exponent $p<d$ fixed, choose the mesoscopic exponent $v_1>0$ small enough that
\[
 bv_1d<\frac{d-p}{4},\qquad v_1\le\frac12,
\]
and set
\[
 \beta_1=\frac{d-p}{4bp},\qquad s_1=b\beta_1=\frac{d-p}{4p}.
\]
By Proposition~\ref{clock:tail-probability}, the clock-tail event
\[
 R^{-2}\sup_{B_R}(U_R-U_R^T)\le R^{-\beta_1}=T^{-s_1}
\]
has failure probability at most
\[
 C e^{-cT^{bv_1\eta_0}}+
 C T^{bv_1d+bp\beta_1-(d-p)}
 \le C e^{-cT^{bv_1\eta_0}}+CT^{-(d-p)/2}.
\]

Apply~\eqref{clock:uniform-truncated-torsion} at $T$ and $2T$. The errors are bounded by $CT^{1-b\delta_B}$, outside an event of probability $CT^{-b\gamma_B}$. Together with the clock-tail estimate, a union bound gives a positive-probability intersection for large $T$.

In any environment in this intersection, evaluate at the center $z_0$ and use $R^{-2}\Psi_R(z_0)=1/2$ and $U_R^{2T}-U_R^T\le U_R-U_R^T$ to obtain
\begin{align*}
 0\le\frac{\mu_{2T}-\mu_T}{2}
 &\le R^{-2}(U_R^{2T}-U_R^T)(z_0)+CT^{1-b\delta_B}\\
 &\le T^{-s_1}+CT^{1-b\delta_B}
 \le CT^{-s}.
\end{align*}
Since the left side is deterministic, this proves
\[
 \mu_{2T}-\mu_T\le CT^{-s}.
\]
Using $\mu_T\uparrow\mu<\infty$ from~\eqref{clock:actual-means}, sum over $2^jT$:
\[
 \mu-\mu_T
 =\sum_{j=0}^{\infty}(\mu_{2^{j+1}T}-\mu_{2^jT})
 \le CT^{-s}\sum_{j=0}^{\infty}2^{-js}
 \le C'T^{-s}.
\]
For bounded $T\ge1$, increase the constant using $\mu-\mu_T\le\mu$.
\end{proof}

\subsection{The cell estimate}
\label{clock:cell-subsection}

\begin{proposition}[Algebraic error for the untruncated physical torsion function]
\label{clock:physical-torsion}
There are fixed positive exponents $\delta_2,\gamma_2$, finite constants, and an initial scale such that, for every fixed deterministic center $z_0$,
\begin{equation}
 \mathbb P\left(
 \sup_{B_R(z_0)}\left|R^{-2}U_R-\mu\frac{1-|x-z_0|^2/R^2}{2}\right|
 >CR^{-\delta_2}\right)
 \le CR^{-\gamma_2}.
 \label{clock:physical-torsion-bound}
\end{equation}
One explicit admissible choice of parameters is the following: first fix $s>0$ from Proposition~\ref{clock:mean-tail}, then choose
\begin{equation}
 \begin{gathered}
 0<t<\frac14\min\{\delta_B,\gamma_B,1\},\qquad
 v=\frac{t(d-p)}{8d},\\
 0<\delta_2<\min\left\{\frac12,\delta_B-t,ts,
                         \frac{t(d-p)}{8p}\right\},\qquad
 \gamma_2=\frac12\min\{\gamma_B-t,t(d-p)\}>0.
 \end{gathered}
 \label{clock:final-parameters}
\end{equation}
\end{proposition}

\begin{proof}
Set $T=R^t$ and use the exact decomposition
\[
 U_R-\mu\Psi_R
 =(U_R-U_R^T)+(U_R^T-\mu_T\Psi_R)+(\mu_T-\mu)\Psi_R.
\]
By~\eqref{clock:uniform-truncated-torsion}, with probability at least $1-CR^{-\gamma_B}$ the second term satisfies
\[
 R^{-2}\|U_R^T-\mu_T\Psi_R\|_\infty
 \le CT R^{-\delta_B}=CR^{-(\delta_B-t)}.
\]
The third term is controlled deterministically by~\eqref{clock:mean-tail-bound}:
\[
 R^{-2}\|(\mu_T-\mu)\Psi_R\|_\infty
 \le CT^{-s}=CR^{-ts}.
\]
Both decay exponents are strictly larger than the chosen $\delta_2$.

For the first term, use~\eqref{clock:tail-probability-bound}. The algebraic power in its failure probability is
\[
 R^{vd+p\delta_2}T^{p-d}
 =R^{vd+p\delta_2-t(d-p)}.
\]
By~\eqref{clock:final-parameters},
\[
 vd=\frac{t(d-p)}8,\qquad
 p\delta_2<\frac{t(d-p)}8,\qquad
 vd+p\delta_2-t(d-p)<-\frac34t(d-p).
\]
Thus the tail probability, its geometric error and the bounded-source failure probability are all at most $CR^{-\gamma_2}$ after increasing the initial scale. This proves~\eqref{clock:physical-torsion-bound}.

The choices are ordered: the geometric $4d$th moment fixes the reverse H\"older constant and $p$; the bounded-source and mean-tail exponents are then fixed; finally choose $t,v,\delta_2,\gamma_2$. The calculation includes the fine-grid loss $R^{vd/p}$.
\end{proof}

\begin{theorem}[Algebraic cell-problem error under the original critical moment]
\label{thm:cell}
Assume \textup{(H1)--(H5)}, and let $\overline A$ be the homogenized matrix defined in Section~\ref{sec:introduction}. For $M\in\mathbb S^d$, let $w_{M,R}^{z_0}$ solve
\begin{equation}
 \left\{
 \begin{aligned}
 -A:D^2w_{M,R}^{z_0}&=(A-\overline A):M&&\text{in }B_R(z_0),\\
 w_{M,R}^{z_0}&=0&&\text{on }\partial B_R(z_0).
 \end{aligned}\right.
 \label{clock:original-cell-equation}
\end{equation}
There are $\delta_2,\gamma_2>0$ and $C,R_2<\infty$, depending on the fixed coefficient law, such that, for every fixed deterministic center $z_0$ and every $R\ge R_2$,
\begin{equation}
 \mathbb P\left(
 R^{-2}\sup_{|M|\le1}\|w_{M,R}^{z_0}\|_{L^\infty(B_R(z_0))}
 >CR^{-\delta_2}\right)
 \le CR^{-\gamma_2}.
 \label{clock:cell-rate}
\end{equation}
One may take the positive exponents from Proposition~\ref{clock:physical-torsion}, decreased by a fixed factor if necessary. The estimate requires only the critical inverse-ellipticity moment in \textup{(H5)}.
\end{theorem}

\begin{proof}
Let
\[
 q_M(x)=\frac12(x-z_0)\cdot M(x-z_0),
\]
and let $H_{M,R}$ be the $A$-harmonic extension of the boundary data $q_M$. Since $A=aB$ and $a>0$, it is also $B$-harmonic. Direct substitution into the equation gives the exact decomposition
\begin{equation}
 w_{M,R}^{z_0}=H_{M,R}-q_M-(\overline A:M)U_R.
 \label{clock:cell-exact-decomposition}
\end{equation}
The harmonic extension for $\overline B:D^2$ is
\begin{equation}
 \overline H_{M,R}=q_M+(\overline B:M)\Psi_R.
 \label{clock:homogeneous-extension}
\end{equation}
Indeed, $D^2\overline H_{M,R}=M-(\overline B:M)I$ and $\operatorname{tr}\overline B=1$.

Since $\operatorname{tr}B=1$, the difference solves
\[
 -B:D^2(H_{M,R}-\overline H_{M,R})
 =B:D^2\overline H_{M,R}
 =(B-\overline B):M,
\]
with zero boundary values. The source is local, bounded and has zero occupation mean by~\eqref{clock:actual-means}. Apply~\eqref{clock:bounded-source-input} to a fixed matrix basis and use linearity:
\begin{equation}
 \mathbb P\left(
 R^{-2}\sup_{|M|\le1}\|H_{M,R}-\overline H_{M,R}\|_\infty
 >CR^{-\delta_B}\right)
 \le CR^{-\gamma_B}.
 \label{clock:direction-harmonic-error}
\end{equation}
Only a finite union is required. Substitute~\eqref{clock:homogeneous-extension} into~\eqref{clock:cell-exact-decomposition} and use
\[
 \mu\overline A=\frac{\overline A}{\operatorname{tr}\overline A}
 =\overline B,
\]
to obtain
\begin{equation}
 w_{M,R}^{z_0}
 =(H_{M,R}-\overline H_{M,R})
 -(\overline A:M)(U_R-\mu\Psi_R).
 \label{clock:cell-cancellation}
\end{equation}
Combine~\eqref{clock:direction-harmonic-error} and Proposition~\ref{clock:physical-torsion}. Since $\delta_2<\delta_B$ and $\gamma_2<\gamma_B$, a union bound proves~\eqref{clock:cell-rate}. The cancellation uses the identities for $\overline A$, $\overline B$ and $\mu$ established in Lemma~\ref{down:density-normalized-law}.
\end{proof}

\medskip\noindent\textit{Quantifiers.}
\label{green:quantifiers}

The geometry, fine grid, reverse H\"older constants and $p<d$ are fixed
before $v$. On $\mathcal G_{R,v}(z_0)$ the estimates hold simultaneously
for all admissible measures, poles and interior sources, including sources
correlated with the Green measure. Only the explicit cell bounds are
averaged in the clock estimate. The common bounded-source threshold
permits both $R=T^b$ and $T=R^t$.

All probability tails are for a fixed deterministic center. A finite set
of centers is handled by a union bound, including its cardinality when
it grows with scale. Full-step sources must first be localized as in
Lemma~\ref{green:full-path-support}; the Green estimate does not extend
to the killing boundary.

\section{Quadratic correctors and second-order regularity}
\label{down:sec-corrector}\label{sec:second-order}

We construct correctors by combining the cell estimate with affine excess
decay. The strict inequality between their exponents makes the differences
of successive cell solutions summable modulo affine functions. Constructing
correctors on expanding domains and controlling each new contribution by
large-scale regularity is also the strategy of Fischer and Otto
\cite[Section~3]{FO}. Here the cell equation is linear and affine functions
are exactly $A$-harmonic. We can therefore use a fixed affine projection;
the convergence condition is the explicit inequality
$\alpha+\delta_2>1$ established below. Write
\[
 [v]_{r,z}=\inf_{\ell\in\mathcal P_1}\|v-\ell\|_{L^\infty(B_r(z))},
 \qquad \mathcal E(v;r,z)=r^{-1}[v]_{r,z}.
\]
All equations below use the effective matrix $\bar A$ of
Theorem~\ref{down:density-main}.

\subsection{The two estimates used in the construction}

By Theorem~\ref{thm:cell}, after decreasing $\delta_2$ we have
$0<\delta_2<1$ and
\begin{equation}
 \mathbb P\left(\sup_{|M|\le1}\|w_{M,R}\|_{L^\infty(B_R)}
                  >C_2R^{2-\delta_2}\right)
 \le C_2R^{-\gamma_2},\qquad R\ge R_0,
 \label{down:cell-input}
\end{equation}
where $w_{M,R}$ solves the zero-boundary problem
$-A:D^2w_{M,R}=(A-\bar A):M$ in $B_R$. Fix
\begin{equation}
 1-\delta_2<\alpha<1.
 \label{down:alpha-choice}
\end{equation}
Theorem~\ref{thm:reg:q1} gives a translation-covariant scale $r_*(z)\ge1$ with
\begin{equation}
 \mathbb P(r_*(z)>t)\le C_1t^{-\gamma_1}
 \label{down:regularity-tail}
\end{equation}
and, at each fixed center almost surely, the estimate
\begin{equation}
 [u]_{r,z}\le C_1(r/R)^{1+\alpha}[u]_{R,z},
 \qquad R\ge2r_*(z),\quad r_*(z)\le r\le R/2,
 \label{down:regularity-input}
\end{equation}
simultaneously for all $A$-harmonic $u$ in $B_R(z)$.

\begin{lemma}[Affine Liouville property]
\label{down:liouville}
In an environment satisfying \eqref{down:regularity-input}, an entire
$A$-harmonic function satisfying
\begin{equation}
 \liminf_{R\to\infty}R^{-\alpha}\mathcal E(u;R,0)=0
 \label{down:liouville-growth}
\end{equation}
is affine. In particular, this holds if $[u]_{R,0}=O(R^\rho)$ for
some $\rho<1+\alpha$.
\end{lemma}

\begin{proof}
For $r\ge r_*(0)$, apply \eqref{down:regularity-input} along radii $R_k$
realizing the liminf:
\[
 \mathcal E(u;r,0)
 \le C_1r^\alpha R_k^{-\alpha}\mathcal E(u;R_k,0)\longrightarrow0.
\]
The affine subspace of $C(\overline{B_r})$ is closed, so $u$ is affine on
$B_r$. The affine expressions agree on nested balls and hence give one
whole-space affine function. The last assertion follows from
$R^{-\alpha}\mathcal E(u;R,0)=O(R^{\rho-1-\alpha})$.
\end{proof}

Intersecting the events for a countable sequence $\alpha\uparrow1$ gives
this conclusion for every fixed-power growth bound $O(R^{2-\varepsilon})$,
$\varepsilon>0$, on one event. It does not cover every $o(R^2)$ bound.

\subsection{Construction and growth}

\begin{theorem}[Whole-space correctors]
\label{down:corrector-main}
There is a measurable family $\phi_M\in C^{2,\theta}_{\rm loc}(\mathbb R^d)$,
linear in $M\in\mathbb S^d$, such that almost surely, simultaneously for all $M$,
\begin{equation}
 -A:D^2\phi_M=(A-\bar A):M\quad\text{in }\mathbb R^d.
 \label{down:corrector-equation}
\end{equation}
For $Y_R=R^{-2}\sup_{|M|\le1}[\phi_M]_{R,0}$ there are deterministic
$C,R_1<\infty$ such that
\begin{equation}
 \mathbb P(Y_R>CR^{-\delta_2})
 \le CR^{-\min\{\gamma_1,\gamma_2\}},\qquad R\ge R_1.
 \label{down:corrector-tail}
\end{equation}
For each fixed deterministic $z$, almost surely for all $M$,
\begin{equation}
 \phi_M(\cdot+z,\omega)-\phi_M(\cdot,\tau_z\omega)\in\mathcal P_1.
 \label{down:corrector-covariance}
\end{equation}
The solution is unique modulo affine functions in the class
$[\phi_M]_{R,0}=O(R^{2-\delta_2})$.
\end{theorem}

\begin{proof}
\emph{Step 1. Convergence after a fixed affine normalization.}
Set $w_{M,n}=w_{M,2^n}$ and $h_{M,n}=w_{M,n+1}-w_{M,n}$.
The difference is $A$-harmonic in $B_{2^n}$, by subtraction of the cell
equations and Lemma~\ref{down:local-regularity}. The cell estimate and
Borel--Cantelli give, almost surely for all sufficiently large $n$,
\begin{equation}
 \sup_{|M|\le1}\|w_{M,n}\|_{L^\infty(B_{2^n})}
 \le C_2\,2^{n(2-\delta_2)}.
 \label{down:eventual-cell-bound}
\end{equation}
Thus \eqref{down:regularity-input} implies
\begin{equation}
 \sup_{|M|\le1}[h_{M,n}]_{r,0}
 \le Cr^{1+\alpha}2^{n(1-\alpha-\delta_2)},
 \qquad r_*(0)\le r\le2^{n-1}.
 \label{down:increment-affine-bound}
\end{equation}

Define the linear projection onto affine functions by
\begin{equation}
 \Pi v(x)=\fint_{B_1}v(y)\,dy
 +(d+2)\sum_{i=1}^d\left(\fint_{B_1}y_iv(y)\,dy\right)x_i.
 \label{down:fixed-projection}
\end{equation}
Indeed, $\fint_{B_1}y_i=0$ and
$\fint_{B_1}y_iy_j=\delta_{ij}/(d+2)$ give $\Pi\ell=\ell$ on
$\mathcal P_1$. Since
$\|\Pi v\|_{B_s}\le C_d(1+s)\|v\|_{B_1}$ for $s\ge1$, subtracting
an affine function and taking the infimum yields
\begin{equation}
 \|(I-\Pi)v\|_{B_s}\le C_s[v]_{s,0}.
 \label{down:projection-quotient}
\end{equation}
Let $q_{M,n}=(I-\Pi)w_{M,n}$. For each fixed compact ball $B_r$, choose
$s\ge\max\{r,r_*(0),1\}$. For sufficiently large $n$,
\[
 \sup_{|M|\le1}\|q_{M,n+1}-q_{M,n}\|_{B_r}
 \le C_ss^{1+\alpha}2^{n(1-\alpha-\delta_2)}.
\]
The exponent is negative by \eqref{down:alpha-choice}. Hence $q_{M,n}$
converges locally uniformly, uniformly over $|M|\le1$, to a function
$\phi_M$. The constant involving $s$ is used only for convergence in a
fixed environment.

Linearity follows from uniqueness of the cell problems and linearity of
$\Pi$. Viscosity stability~\cite[Section~6]{CIL92} gives
\eqref{down:corrector-equation}, and
Lemma~\ref{down:local-regularity} gives the stated local regularity.
For measurability, first fix $R$ and $M$ and consider admissible coefficient
restrictions $A_j\to A$ uniformly on $\overline{B_R}$. Let $w_j,w$ be
their zero-boundary cell solutions, all defined using the same fixed
matrix $\bar A$. Continuity and strict positivity give
\[
 \kappa=\min_{\overline{B_R}}\lambda_{\min}(A)>0,
 \qquad A_j\ge\tfrac12\kappa I
 \quad\text{for all sufficiently large }j.
\]
The boundary Schauder theorem \cite[Theorem~6.14]{GT}, applied to the
fixed coefficient $A$, gives $D^2w\in L^\infty(B_R)$. Subtracting the
classical equations yields
\[
 -A_j:D^2(w_j-w)=(A_j-A):(M+D^2w)\quad\text{in }B_R,
 \qquad w_j-w=0\quad\text{on }\partial B_R.
\]
Comparison with both signs of the quadratic barrier gives
\[
 \|w_j-w\|_\infty
 \le\frac{R^2}{d\kappa}\|A_j-A\|_\infty
       \bigl(|M|+\|D^2w\|_\infty\bigr)\longrightarrow0.
\]
Thus the cell-solution map is continuous in the relative uniform
topology on admissible coefficient restrictions. The coefficient
restriction is a measurable map into $C(\overline{B_R};\mathbb S^d)$,
since its evaluations on a fixed countable dense set are measurable.
Thus the cell solution is measurable. Construct these solutions for a
fixed finite basis of $\mathbb S^d$, apply the continuous integral
projection $I-\Pi$, and take the locally uniform limit on a countable
compact exhaustion. Extending linearly in $M$ and setting the family
equal to zero on the exceptional null set gives a measurable family
$\phi_M$. Matrix suprema are measurable by reduction to a countable
dense set. The constants in the continuity argument may depend on $A$
and $R$; no probabilistic bound on them is used.

\emph{Step 2. Growth and its probability bound.}
Take $r\ge r_*(0)$ and $n$ with $2r\le2^n<4r$, and suppose
\eqref{down:eventual-cell-bound} holds for every $k\ge n$. In the quotient
$C(\overline{B_r})/\mathcal P_1$, the classes of $q_{M,k}$ and $w_{M,k}$
agree. The locally uniform limit and \eqref{down:increment-affine-bound}
therefore give
\begin{align}
 \sup_{|M|\le1}[\phi_M-w_{M,n}]_{r,0}
 &\le Cr^{1+\alpha}\sum_{k=n}^\infty2^{k(1-\alpha-\delta_2)}\notag\\
 &\le Cr^{1+\alpha}2^{n(1-\alpha-\delta_2)}
 \le Cr^{2-\delta_2}.
 \label{down:quotient-tail-sum}
\end{align}
Adding the cell bound gives, with a deterministic constant,
\begin{equation}
 \sup_{|M|\le1}[\phi_M]_{r,0}\le Cr^{2-\delta_2}.
 \label{down:corrector-deterministic-growth}
\end{equation}
In particular, it holds almost surely for every sufficiently large real $r$.
For a deterministic $r$, define
\[
 G_r=\{r_*(0)\le r\}\cap
 \bigcap_{k\ge\lceil\log_2(2r)\rceil}
 \left\{\sup_{|M|\le1}\|w_{M,k}\|_\infty
                     \le C_2\,2^{k(2-\delta_2)}\right\}.
\]
On the construction event, $G_r$ implies
\eqref{down:corrector-deterministic-growth}, while
\[
 \mathbb P(G_r^c)
 \le C_1r^{-\gamma_1}
    +\sum_{k\ge\lceil\log_2(2r)\rceil}C_2\,2^{-k\gamma_2}
 \le Cr^{-\min\{\gamma_1,\gamma_2\}}.
\]
This proves \eqref{down:corrector-tail}.

\emph{Step 3. Uniqueness and translation covariance.}
The difference $v$ of two solutions in the asserted class is entire
$A$-harmonic and satisfies
$\mathcal E(v;R,0)=O(R^{1-\delta_2})=o(R^\alpha)$.
Lemma~\ref{down:liouville} makes it affine.
For fixed $z$, stationarity ensures that both $\omega$ and $\tau_z\omega$
lie in the construction event almost surely. The function
\[
 v_M(x)=\phi_M(x+z,\omega)-\phi_M(x,\tau_z\omega)
\]
is harmonic in the shifted environment. Translation preserves affine
functions, and $B_R+z\subset B_{R+|z|}$ preserves the growth class.
Uniqueness proves \eqref{down:corrector-covariance}; a finite basis and
linearity make the statement simultaneous in $M$.
\end{proof}

\begin{remark}[Covariance of derivatives]
\label{down:corrector-quantifier-remark}
For each fixed $z$, differentiating \eqref{down:corrector-covariance}
twice shows that $D^2\phi_M$ is translation covariant. The function and
its gradient need not be stationary. We do not assert integrability of
$D^2\phi_M(0)$ or use a mean-zero Hessian identity. If first-moment
integrability is established separately, the ergodic theorem and
\[
 \left|\int D^2\phi_M\,\psi_R\right|
 \le C_\psi R^{-2}[\phi_M]_{cR,0}\longrightarrow0
\]
identify that mean as zero. The covariance statement remains an
almost-sure statement for each fixed translation. For comparison,
Armstrong and Lin \cite[Section~7]{AL} construct correctors with stationary
values in dimensions $d>4$, under uniform ellipticity and their local
product-space assumption (P3). Their fluctuation argument uses that
additional probabilistic structure. The present construction gives
covariance modulo affine functions under H1--H5.
\end{remark}


We first obtain a harmonic approximation on an event common to all
$A$-harmonic functions with uniform norm at most one. Adding a small
corrector to the quadratic Taylor polynomial of the harmonic
approximation then improves the
second-order excess. The corrected-polynomial excess and its Liouville
consequence follow the framework of Fischer and Otto
\cite[Section~3]{FO}, whose divergence-form theory assumes uniform
ellipticity and quantified sublinearity of the first-order corrector and
its flux potential. Here affine functions already solve the equation,
and the quadratic correction is supplied by
Theorem~\ref{down:corrector-main}; the excess is measured in the uniform
norm. The approximation-and-iteration argument is also related to the
nondivergence-form $C^{1,1}$ theory of Armstrong and Lin
\cite[Section~3]{AL}. Their regularity theorem assumes uniform ellipticity,
stationarity and finite range of dependence. The proof below uses the
corrector estimate together with the coarse H\"older estimate and the
fixed-tolerance events of Proposition~\ref{prop:reg:fixed-tolerance}.
Each event is common to all the functions under
consideration. Set
\[
 \mathcal S_0=\{M\in\mathbb S^d:\bar A:M=0\},\qquad
 \gamma_H=\frac{d}{4d-2},\qquad
 \gamma_\phi=\min\{\gamma_H,\gamma_2\},
\]
with the parameter choices in Theorem~\ref{thm:reg:q1}.

\subsection{Uniform approximation of harmonic functions}

Lemmas~\ref{lem:reg:ellipticity-averages} and~\ref{lem:reg:coarse-holder}
give $\nu,a_0\in(0,1)$ and coefficient events $\mathcal V_R(z)$ with
$\mathbb P(\mathcal V_R(z)^c)\le CR^{-\gamma_H}$. On these events,
every $A(z+R\cdot)$-harmonic $u$ in $B_1$ with $\|u\|_{B_1}\le1$ satisfies
\begin{equation}
 |u(x)-u(y)|\le C\bigl(|x-y|^\nu+R^{-(1-a_0)\nu}\bigr),
 \qquad x,y\in B_{3/4}.
 \label{eq:second-coarse-holder-input}
\end{equation}
Proposition~\ref{prop:reg:fixed-tolerance} gives, for each $t>0$ and
$R\ge R_t$, coefficient events $\mathcal E_R(t,z)$ with
$\mathbb P(\mathcal E_R(t,z)^c)\le C_t\exp(-c_tR^{b_t})$ on which
\begin{equation}
 \|u_F-F((\cdot-z)/R)\|_{B_R(z)}
 \le t\|F\|_{C^3(\overline B_1)}.
 \label{eq:second-fixed-tolerance-input}
\end{equation}
Here $F$ is any $\bar A$-harmonic function in $C^3(\overline B_1)$ and
$u_F$ is the $A$-harmonic extension of $F((\cdot-z)/R)$ to $B_R(z)$.
Both event families are measurable, translation covariant and uniform in
the indicated functions.

\begin{lemma}[Uniform homogeneous approximation]
\label{lem:second-uniform-approximation}
For each $\eta>0$ there are deterministic $R_\eta,C_\eta<\infty$ and
translation-covariant measurable events $\mathcal A_{\eta,R}(z)$ with
$\mathbb P(\mathcal A_{\eta,R}(z)^c)\le C_\eta R^{-\gamma_H}$ for
$R\ge R_\eta$, such that every $A(z+R\cdot)$-harmonic $u$ in $B_1$,
$\|u\|_{B_1}\le1$, admits a $\bar A$-harmonic $v$ in $B_{1/2}$ satisfying
\begin{equation}
 \|v\|_{B_{1/2}}\le1,\qquad \|u-v\|_{B_{1/2}}\le\eta.
 \label{eq:second-uniform-approximation}
\end{equation}
\end{lemma}

\begin{proof}
For $0<h<1/12$, mollify $u$ with a nonnegative smooth kernel of integral
one and radius $h$, obtaining $g$ in $B_{2/3}$. On $\mathcal V_R(z)$,
\[
 \|u-g\|_{B_{2/3}}
 \le C\bigl(h^\nu+R^{-(1-a_0)\nu}\bigr),\qquad
 \|g\|_{C^4(B_{2/3})}\le C_h,\qquad |g|\le1.
\]
Let $v$ and $u_g$ be the $\bar A$-harmonic and $A(z+R\cdot)$-harmonic
extensions of $g$ to $B_{1/2}$. Comparison gives $|v|\le1$ and
\[
 \|u-u_g\|_{B_{1/2}}
 \le C\bigl(h^\nu+R^{-(1-a_0)\nu}\bigr).
\]
The constant-coefficient boundary Schauder estimate gives
$\|v\|_{C^3(\overline B_{1/2})}\le C_h$. Applying
\eqref{eq:second-fixed-tolerance-input} at radius $R/2$ to $F(y)=v(y/2)$
gives $\|u_g-v\|_{B_{1/2}}\le tC_h$ on $\mathcal E_{R/2}(t,z)$.
Choose $h$, then $t$, then $R_\eta$ so that these three error terms are
at most $\eta/3$, and take
\[
 \mathcal A_{\eta,R}(z)=\mathcal V_R(z)\cap\mathcal E_{R/2}(t,z).
\]
The union bound gives the probability estimate. Since both input events
are uniform in their data, the choices of $g$ and $v$ introduce no new
exceptional event.
\end{proof}

\subsection{Corrected quadratic polynomials}

The correctors of Theorem~\ref{down:corrector-main} satisfy
\begin{equation}
 -A:D^2\phi_M=(A-\bar A):M
 \label{eq:second-corrector-equation}
\end{equation}
and, for each fixed $z$, almost surely simultaneously in $M$,
\begin{equation}
 \phi_M(\cdot+z,\omega)-\phi_M(\cdot,\tau_z\omega)\in\mathcal P_1.
 \label{eq:second-corrector-covariance}
\end{equation}
With $\delta_\phi=\delta_2\in(0,1)$, define
$Y_R(z)=R^{-2}\sup_{|M|\le1}[\phi_M]_{R,z}$. Covariance and
\eqref{down:corrector-tail} give
\begin{equation}
 \mathbb P(Y_R(z)>CR^{-\delta_\phi})\le CR^{-\gamma_\phi},
 \qquad R\ge R_0.
 \label{eq:second-corrector-tail-input}
\end{equation}
Define the finite-dimensional space of exact $A$-harmonic functions
\[
 \mathcal H_2^A=\left\{a+b\cdot x+\frac12x\cdot Mx+\phi_M(x):
             a\in\mathbb R,\ b\in\mathbb R^d,\ M\in\mathcal S_0\right\}
\]
and its excess
\begin{equation}
 D_2(u;r,z)=\inf_{H\in\mathcal H_2^A}\|u-H\|_{B_r(z)},\qquad
 E_2(u;r,z)=r^{-2}D_2(u;r,z).
 \label{eq:second-excess-definition}
\end{equation}
The space is independent of the affine representatives of $\phi_M$ and
of the center chosen in the quadratic polynomial. Equation
\eqref{eq:second-corrector-covariance} therefore transports the whole
space under each fixed translation. A finite matrix basis makes this
statement simultaneous in $M$.

\begin{lemma}[Closure and coercivity]
\label{lem:second-coercivity}
For bounded $u$, the infimum in \eqref{eq:second-excess-definition} and
the affine infima defining $Y_R$ are attained. If
$c_d=(4\sqrt d)^{-1}$ and $q_M(x)=\tfrac12(x-z)\cdot M(x-z)$, then
\begin{align}
 [q_M]_{R,z}&\ge c_dR^2|M|,
 \label{eq:second-quadratic-coercivity}\\
 [q_M+\phi_M]_{R,z}&\ge(c_d-Y_R(z))R^2|M|.
 \label{eq:second-corrected-coercivity}
\end{align}
If $Y_R(z)<c_d$, restriction to $B_R(z)$ is injective on $\mathcal H_2^A$.
\end{lemma}

\begin{proof}
The restrictions of the two approximation spaces are finite-dimensional
subspaces of the uniform normed space and hence closed. A minimizing
sequence has bounded norm; finite-dimensional compactness gives
attainment. This argument also applies to bounded continuous functions
on an open ball.

Translate $z$ to zero and suppose $\|q_M-a-b\cdot x\|_{B_R}\le e$.
Evaluation at zero gives $|a|\le e$. Averaging the values at $sv$ and
$-sv$, for $|v|=1$ and $s<R$, gives
$|s^2v\cdot Mv/2-a|\le e$. Letting $s\uparrow R$ yields
$\|M\|_{\rm op}\le4e/R^2$, which implies
\eqref{eq:second-quadratic-coercivity} since
$|M|\le\sqrt d\|M\|_{\rm op}$. The triangle inequality in the quotient
by affine functions and linearity in $M$ give
\eqref{eq:second-corrected-coercivity}. If a corrected polynomial
vanishes on the ball and $Y_R<c_d$, then $M=0$, and its remaining affine
part vanishes globally.
\end{proof}

\subsection{Excess decay}

Interior estimates for the constant-coefficient equation give a
$K=K(d,\bar A)\ge1$ such that every $\bar A$-harmonic $v$ in $B_{1/2}$,
$\|v\|_{B_{1/2}}\le2$, has a Taylor polynomial
$P(x)=v(0)+Dv(0)\cdot x+\tfrac12x\cdot Mx$ satisfying
\begin{equation}
 M=D^2v(0)\in\mathcal S_0,\qquad |M|\le K,\qquad
 \|v-P\|_{B_\theta}\le K\theta^3\quad(0<\theta\le1/8).
 \label{eq:second-taylor}
\end{equation}
Indeed, the change of variables $x=\bar A^{1/2}y$ gives a harmonic
function on an ellipsoid. On balls of a fixed radius about the image of
$B_{1/4}$, differentiation of the mean-value formula bounds derivatives
through order three by $\|v\|_\infty$. Taylor's formula gives the estimate,
and the equation gives $\bar A:M=0$.

\begin{proposition}[One-step improvement]
\label{prop:second-one-step}
Fix $\alpha\in(0,1)$ and choose a dyadic $\theta=2^{-m}\le1/8$ with
$K\theta^3\le\theta^{2+\alpha}/4$. Set
\[
 \eta=\tfrac18\theta^{2+\alpha},\qquad
 \kappa=\min\{c_d/2,\theta^{2+\alpha}/(8K)\}.
\]
On $\mathcal A_{\eta,R}(z)\cap\{Y_R(z)\le\kappa\}$, for $R\ge R_\eta$,
every $A$-harmonic function in $B_R(z)$ satisfies
\begin{equation}
 E_2(u;\theta R,z)\le\tfrac12\theta^\alpha E_2(u;R,z).
 \label{eq:second-one-step}
\end{equation}
\end{proposition}

\begin{proof}
The cases of infinite or zero outer excess are immediate, the latter by
attainment. Otherwise choose a minimizer $H_0$, set $D=D_2(u;R,z)>0$,
and apply Lemma~\ref{lem:second-uniform-approximation} to
\[
 U(x)=D^{-1}(u-H_0)(z+Rx),\qquad \|U\|_{B_1}=1.
\]
Let $v,P,M$ be the resulting harmonic approximation and its Taylor data.
Choose an affine $\ell_M$ with
$\|\phi_M-\ell_M\|_{B_R(z)}\le\kappa R^2|M|$. The function
\[
 h(x)=P(x)+R^{-2}\bigl(\phi_M(z+Rx)-\ell_M(z+Rx)\bigr)
\]
is exactly $A(z+R\cdot)$-harmonic, and
\[
 \|U-h\|_{B_\theta}\le\eta+K\theta^3+K\kappa
 \le\tfrac12\theta^{2+\alpha}.
\]
Moreover $H_0(y)+Dh((y-z)/R)\in\mathcal H_2^A$: its added matrix is
$DR^{-2}M$, its added corrector is $\phi_{DR^{-2}M}$, and the other terms
are affine. Rescaling the approximation error proves the result.
\end{proof}

\begin{theorem}[Corrected large-scale $C^{2,\alpha}$ estimate]
\label{thm:second-regularity}
Under H1--H5, for every $\alpha\in(0,1)$ there are $C_\alpha<\infty$ and
a measurable translation-covariant scale $r_{2,\alpha}(z)\ge1$ with
\begin{equation}
 \mathbb P(r_{2,\alpha}(z)>t)\le C_\alpha t^{-\gamma},\qquad
 \gamma=\min\{\gamma_H,\gamma_\phi\}>0,\quad t\ge1.
 \label{eq:second-random-tail}
\end{equation}
For each fixed $z$, almost surely for all $R\ge2r_{2,\alpha}(z)$, all
$A$-harmonic $u$ in $B_R(z)$ and $r_{2,\alpha}(z)\le r\le R/2$,
\begin{equation}
 E_2(u;r,z)\le C_\alpha(r/R)^\alpha E_2(u;R,z).
 \label{eq:second-large-scale}
\end{equation}
\end{theorem}

\begin{proof}
Choose $n_{\rm det}\ge0$ so that all inputs apply and
$C2^{-n\delta_\phi}\le\kappa$ for $n\ge n_{\rm det}$. Then
\[
 G_n(z)=\mathcal A_{\eta,2^n}(z)\cap\{Y_{2^n}(z)\le\kappa\},\qquad
 \mathbb P(G_n(z)^c)\le C_\alpha2^{-n\gamma}.
\]
Define at the origin
\[
 n_* =\max\{n_{\rm det},1+\sup\{n\ge n_{\rm det}:G_n(0)^c\}\},
 \qquad r_{2,\alpha}(0)=2^{n_*},
\]
with the empty supremum equal to $n_{\rm det}-1$. The union bound gives
\[
 \mathbb P(r_{2,\alpha}(0)>t)
 \le C_\alpha\sum_{2^n>t/2}2^{-n\gamma}
 \le C_\alpha t^{-\gamma}
\]
for large $t$, and an enlarged constant covers $t\ge1$. Borel--Cantelli
gives finiteness. Matrix, affine and spatial extrema defining $Y_R$ can
be taken over fixed countable dense sets, so the scale is measurable.
Set $r_{2,\alpha}(z,\omega)=r_{2,\alpha}(0,\tau_z\omega)$.
Covariance of the corrector classes and events gives the asserted
fixed-center properties.

For an admissible $R$, let $S\le R$ be the largest dyadic radius, and
put $S_j=\theta^jS$. Choose the largest $j$ with $S_j\ge r$. All the
radii in the iteration are dyadic and at least $r_{2,\alpha}(z)$, so
\eqref{eq:second-one-step} yields
\[
 E_2(u;S_j,z)\le(S_j/S)^\alpha E_2(u;S,z).
\]
Since $D_2$ is nondecreasing in the radius, $R/2<S\le R$ and
$r\le S_j<r/\theta$ give
\begin{align*}
 E_2(u;r,z)&\le(S_j/r)^2E_2(u;S_j,z)\\
 &\le\theta^{-(2+\alpha)}(r/S)^\alpha E_2(u;S,z)\\
 &\le2^{2+\alpha}\theta^{-(2+\alpha)}(r/R)^\alpha E_2(u;R,z).
\end{align*}
This deterministic iteration on countably many coefficient events
covers all solutions and all admissible real radii simultaneously.
\end{proof}

\subsection{Liouville classification}

\begin{lemma}[Growth of affine representatives]
\label{lem:second-growth-representatives}
If $f\in C(\mathbb R^d)$, $0<\delta_0<1$, and
$[f]_{2^n,0}\le K2^{n(2-\delta_0)}$ for every sufficiently large $n$,
then $\|f\|_{B_R}=O(R^{2-\delta_0})$.
\end{lemma}

\begin{proof}
Let $\ell_n=a_n+b_n\cdot x$ attain the affine infimum. On $B_{2^n}$,
$\|\ell_{n+1}-\ell_n\|\le CK2^{n(2-\delta_0)}$. Evaluation at zero
and at opposite points gives
\[
 |a_{n+1}-a_n|\le CK2^{n(2-\delta_0)},\qquad
 |b_{n+1}-b_n|\le CK2^{n(1-\delta_0)}.
\]
Both geometric sums increase with $n$, and hence
$|a_n|=O(2^{n(2-\delta_0)})$, $|b_n|=O(2^{n(1-\delta_0)})$.
Adding the approximation error proves the dyadic bound; monotonicity
covers other radii.
\end{proof}

By \eqref{eq:second-corrector-tail-input} and Borel--Cantelli, the lemma
applies simultaneously to all correctors, by linearity and a finite
matrix basis, for any fixed $0<\delta_0<\delta_\phi$. Thus every fixed
affine representative is $O(R^{2-\delta_0})$, and every member of
$\mathcal H_2^A$ has at most quadratic growth.

\begin{corollary}[Second-order Liouville theorem]
\label{cor:second-liouville}
Almost surely, every entire $A$-harmonic function with
$\|u\|_{B_R}=O(R^{3-\varepsilon})$ for some $\varepsilon>0$ belongs to
$\mathcal H_2^A$. For each $0<\varepsilon\le1$, the space of solutions
with this growth is exactly $\mathcal H_2^A$ and has dimension
$d(d+3)/2$. One event suffices for all such $\varepsilon$ and solutions.
\end{corollary}

\begin{proof}
Intersect the events of Theorem~\ref{thm:second-regularity} over rational
$\alpha\in(0,1)$ and include the corrector growth event. Given
$\varepsilon>0$, choose rational $\alpha>\max\{0,1-\varepsilon\}$.
For fixed $r\ge r_{2,\alpha}(0)$, using zero as the outer approximant,
\[
 E_2(u;r,0)\le C(r/R)^\alpha E_2(u;R,0)
 \le C_uC r^\alpha R^{1-\varepsilon-\alpha}\longrightarrow0.
\]
Closure gives $H_r\in\mathcal H_2^A$ equal to $u$ on $B_r$.
Fix a dyadic $r_0\ge r_{2,\alpha}(0)$; then $Y_{r_0}\le\kappa<c_d$.
For every $r\ge r_0$, $H_r-H_{r_0}$ vanishes on $B_{r_0}$, so
Lemma~\ref{lem:second-coercivity} makes it zero globally. Hence
$u=H_{r_0}$ on the whole space.

The quadratic growth of members of $\mathcal H_2^A$ gives the reverse
inclusion for $\varepsilon\le1$. Coercivity makes the parametrization by
$(a,b,M)$ injective, and
\[
 \dim\mathcal H_2^A
 =(d+1)+\dim\ker(M\mapsto\bar A:M)
 =(d+1)+d(d+1)/2-1=d(d+3)/2.
\]
\end{proof}

The proof requires a fixed-power improvement over cubic growth and does
not classify all $o(R^3)$ solutions. The scale tail exponent is
$\min\{d/(4d-2),\gamma_2\}$; constants may depend on the coefficient law
and on the prescribed regularity exponent.

\section{Quantitative averages of the invariant density}
\label{down:sec-density-quantitative}

The cell estimate controls averages of $m(A-\bar A)$ through the adjoint
equation. We recover averages of $m-1$ by evolving the test function
under the homogenized heat semigroup and integrating its time
derivative. The construction of Section~\ref{sec:density} gives
\begin{equation}
 m>0,\qquad \mathbb E[m(0)]=1,\qquad
 \int mA:D^2\zeta=0\quad(\zeta\in C_c^\infty),\qquad
 \mathbb E[mA]=\bar A.
 \label{down:density-input}
\end{equation}
Throughout this section only the first probabilistic moment of $m$ is
used.

Armstrong, Fehrman and Lin \cite[Theorem~1.4 and Section~5.2]{AFL}
obtain algebraic rates for averages of $m$
and $mA$ under deterministic uniform ellipticity, uniform H{\"o}lder
continuity, and finite-range dependence. Their proof uses quantitative
parabolic Green-function estimates. Under H1--H5, the cited theorem
does not apply. We derive these estimates from the finite-cell
bound~\eqref{down:cell-input}, as outlined above.

\begin{theorem}[Algebraic rates for density and flux averages]
\label{down:density-quantitative-main}
Suppose~\eqref{down:cell-input} holds. For each fixed
$\psi\in C_c^\infty(\mathbb R^d)$, set $\psi_R(x)=R^{-d}\psi(x/R)$ and
\[
 Z_R(\psi)=\left|\int\psi_R(m-1)\right|
             +\left|\int\psi_R(mA-\bar A)\right|,
\]
with the Frobenius norm in the second term. There are deterministic
$C_\psi,R_\psi<\infty$ such that, for every deterministic $R\ge R_\psi$,
\begin{equation}
 \mathbb P\left(Z_R(\psi)>C_\psi R^{-\beta_3}\right)
 \le C_\psi R^{-\gamma_3},\qquad
 \beta_3=\frac{\delta_2}{4},\qquad
 \gamma_3=\min\left\{\gamma_2,\frac{\delta_2}{4}\right\}.
 \label{down:density-quantitative-tail}
\end{equation}
\end{theorem}

\subsection{From cell estimates to weighted averages}

\begin{lemma}[Adjoint energy estimate]\label{down:adjoint-energy}
Let $r\ge1$ and $w\in C^2(B_{2r})$ satisfy
$-A:D^2w=f$, $\|f\|_{B_{2r}}\le C_f$, and $\|w\|_{B_{2r}}\le W$.
Then
\begin{equation}
 \int_{B_r}m\,Dw\cdot A Dw
 \le C_d(C_fW+W^2r^{-2})\int_{B_{2r}}m.
 \label{down:energy-estimate}
\end{equation}
\end{lemma}
\begin{proof}
Choose $\chi\in C_c^\infty(B_{2r})$ equal to $1$ on $B_r$, with
$0\le\chi\le1$, $|D\chi|\le C/r$, and $|D^2\chi|\le C/r^2$.
Since $mA\in L^1_{\rm loc}$, approximation in $C^2$ on a fixed compact
support extends the adjoint identity to $C_c^2$ tests, in particular
$\chi^2w^2$. Writing $\mathcal G=A:D^2$, its product rule is
\begin{align*}
 0={}&-2\int m\chi^2wf+2\int m\chi^2Dw\cdot A Dw\\
 &+\int mw^2(2\chi\mathcal G\chi+2D\chi\cdot A D\chi)
       +8\int m\chi wD\chi\cdot A Dw.
\end{align*}
Matrix Cauchy--Schwarz gives
\[
 8|\chi wD\chi\cdot A Dw|
 \le\chi^2Dw\cdot A Dw+16w^2D\chi\cdot A D\chi.
\]
Absorb the energy term and use $A\le I$ to bound all remaining cutoff
terms by $C_dW^2r^{-2}\int_{B_{2r}}m$. The source term contributes
$2C_fW\int_{B_{2r}}m$, proving the claim.
\end{proof}

\begin{lemma}[Weighted coefficient error]\label{down:cell-to-weighted-error}
Set $J=m(A-\bar A)$. Suppose that, for $r\ge1$ and all $|M|\le1$,
\begin{equation}
 \|w_{M,2r}\|_\infty\le Kr^{2-\delta_2}.
 \label{down:local-cell-good}
\end{equation}
Then for every $\varphi\in C_c^2(B_r)$ and $M\in\mathbb S^d$ with
$|M|\le1$,
\begin{equation}
 \left|\int\varphi J:M\right|
 \le Cr^{-\delta_2/2}\left(\fint_{B_{2r}}m\right)
       \max_{j=1,2}\bigl(r^{d+j}\|D^j\varphi\|_\infty\bigr),
 \label{down:weighted-error}
\end{equation}
where $C$ depends on $d$ and $K$.
\end{lemma}
\begin{proof}
For $w=w_{M,2r}$, the source $f=(A-\bar A):M$ is bounded by $C_d$.
Local regularity and Lemma~\ref{down:adjoint-energy} yield
\begin{equation}
 \int_{B_r}mDw\cdot A Dw
 \le C(r^{2-\delta_2}+r^{2-2\delta_2})\int_{B_{2r}}m
 \le Cr^{2-\delta_2}\int_{B_{2r}}m.
 \label{down:cell-energy}
\end{equation}
Test the adjoint equation with $\varphi w$ to obtain
\begin{equation}
 \int\varphi J:M=\int mwA:D^2\varphi
                  +2\int mD\varphi\cdot A Dw.
 \label{down:weighted-error-identity}
\end{equation}
The first term is bounded by
$Cr^{2-\delta_2}\|D^2\varphi\|_\infty\int_{B_{2r}}m$.
Weighted Cauchy--Schwarz and~\eqref{down:cell-energy} bound the second
by $Cr^{1-\delta_2/2}\|D\varphi\|_\infty\int_{B_{2r}}m$.
These two bounds imply~\eqref{down:weighted-error}.
\end{proof}

We first prove the theorem for $\operatorname{supp}\psi\subset B_1$.
Fix
\[
 0<\varsigma<\delta_2/2,\qquad
 \kappa=\delta_2/2-\varsigma,\qquad r_n=2^nR\quad(n\ge0).
\]
For a sufficiently large deterministic $K$, let $\mathcal E_R$ be the
event on which, for every $n\ge0$,
\begin{equation}
 \sup_{|M|\le1}\|w_{M,2r_n}\|_\infty\le Kr_n^{2-\delta_2},
 \qquad \fint_{B_{2r_n}}m\le r_n^{\varsigma}.
 \label{down:all-scales-event}
\end{equation}
Stationarity and $\mathbb E m(0)=1$ give
$\mathbb E\fint_{B_{2r_n}}m=1$. Markov's inequality, the cell estimate,
and a union bound therefore imply
\begin{equation}
 \mathbb P(\mathcal E_R^c)
 \le C\sum_{n\ge0}(r_n^{-\gamma_2}+r_n^{-\varsigma})
 \le C(R^{-\gamma_2}+R^{-\varsigma}).
 \label{down:all-scales-probability}
\end{equation}
This step requires no independence between $m$, the cell solutions, or
the events at different radii. On $\mathcal E_R$,
\begin{equation}
 \left|\int\varphi J:M\right|
 \le Cr_n^{-\kappa}
       \max_{j=1,2}(r_n^{d+j}\|D^j\varphi\|_\infty),
 \quad \varphi\in C_c^2(B_{r_n}),\quad |M|\le1,
 \label{down:weighted-error-good-event}
\end{equation}
and comparison with the dyadic radii gives
\begin{equation}
 \int_{B_L}m\le CL^{d+\varsigma}\qquad(L\ge R).
 \label{down:polynomial-mass}
\end{equation}
The latter bound justifies the noncompact heat-kernel tests below.

\subsection{The heat-semigroup argument}

Let $\bar p_t$ be the heat kernel of $\bar A:D^2$:
\[
 \bar p_t(x)=(4\pi t)^{-d/2}(\det\bar A)^{-1/2}
 \exp\left(-\frac{x\cdot\bar A^{-1}x}{4t}\right).
\]
Set $\varphi_t=\bar p_t*\psi_R$ for $t>0$, $\varphi_0=\psi_R$, and
$S(t)=\sqrt{R^2+t}$.

\begin{lemma}[Heat-kernel tests]\label{down:heat-test-bounds}
For $0\le j\le4$ and $t\ge0$,
\begin{equation}
 |D^j\varphi_t(x)|\le C_\psi S(t)^{-d-j}
                  \exp\left(-c\frac{|x|^2}{S(t)^2}\right).
 \label{down:heat-test-gaussian}
\end{equation}
The constants depend on $\psi$, $d$, and the ellipticity of $\bar A$.
\end{lemma}
\begin{proof}
For $t\ge R^2$, differentiate the Gaussian and use $|y|\le R\le\sqrt t$
on the support of $\psi_R$, together with
$|x-y|^2\ge|x|^2/2-|y|^2$ and $\|\psi_R\|_1=\|\psi\|_1$.
This gives $Ct^{-(d+j)/2}e^{-c|x|^2/t}$.
For $0<t\le R^2$, place derivatives on $\psi_R$. When $|x|\le2R$,
the bound follows from
$\|D^j\varphi_t\|_\infty\le C_\psi R^{-d-j}$. When $|x|>2R$,
\[
 |D^j\varphi_t(x)|\le C_\psi R^{-j}t^{-d/2}e^{-c|x|^2/t}
 \le C_\psi R^{-d-j}e^{-c'|x|^2/R^2}.
\]
The last inequality follows by putting $u=t/R^2\le1$,
$K=|x|^2/R^2>4$, and using
$u^{-d/2}e^{-cK/u}\le Ce^{-c'K}$.
At $t=0$, compact support gives the result directly.
\end{proof}

On $\mathcal E_R$, the function $F(t)=\int m\varphi_t$ is absolutely
convergent and differentiable on bounded time intervals, by
\eqref{down:polynomial-mass} and~\eqref{down:heat-test-gaussian}.
To use the adjoint equation, take a cutoff $\eta_L$ equal to $1$ on
$B_L$ and supported in $B_{2L}$. The product-rule error in testing
with $\eta_L\varphi_t$ is bounded by
\[
 C\int_{B_{2L}\setminus B_L}m
       (L^{-1}|D\varphi_t|+L^{-2}|\varphi_t|).
\]
For fixed $t$ it tends to zero by Gaussian decay and the polynomial mass
bound. The main term converges absolutely, so
\begin{equation}
 F'(t)=\int m\bar A:D^2\varphi_t
      =\int m(\bar A-A):D^2\varphi_t
      =-\int J:D^2\varphi_t.
 \label{down:heat-derivative-identity}
\end{equation}

\begin{lemma}[Integrable time derivative]\label{down:heat-derivative-estimate}
On $\mathcal E_R$, for every $t\ge0$,
\begin{equation}
 |F'(t)|\le C_\psi S(t)^{-2-\kappa},\qquad
 \int_0^\infty|F'(t)|dt\le\frac{2C_\psi}{\kappa}R^{-\kappa}.
 \label{down:heat-derivative-bound}
\end{equation}
\end{lemma}
\begin{proof}
Choose $n$ so that $r_n\le S(t)<2r_n$ and a radially nonincreasing
$\xi\in C_c^\infty(B_1)$ equal to $1$ on $B_{1/2}$, with $0\le\xi\le1$.
The partition
\[
 \chi_0=\xi(\cdot/r_{n+2}),\qquad
 \chi_k=\xi(\cdot/r_{n+k+2})-\xi(\cdot/r_{n+k+1})\quad(k\ge1)
\]
satisfies $\sum\chi_k=1$, $\chi_k\ge0$, and
$\operatorname{supp}\chi_k\subset B_{r_{n+k+2}}$. For $k\ge1$ its
support lies in $|x|\ge2^{k-1}S(t)$, and
$|D^j\chi_k|\le C_jr_{n+k+2}^{-j}$.
For a Frobenius orthonormal basis $(E^a)$ of $\mathbb S^d$, set
$\Phi_{a,k}=\chi_k(D^2\varphi_t:E^a)$ and $r=r_{n+k+2}$. The product
rule and~\eqref{down:heat-test-gaussian} give
\begin{equation}
 \max_{j=1,2}r^{d+j}\|D^j\Phi_{a,k}\|_\infty
 \le C_\psi S(t)^{-2}2^{k(d+2)}e^{-c4^k}.
 \label{down:annular-test-estimate}
\end{equation}
Indeed, a term $D^b\chi_kD^{j-b+2}\varphi_t$ contributes at most
$S(t)^{-2}(r/S(t))^{d+j-b}$ after multiplication by $r^{d+j}$;
$j-b\le2$, and the support gives the Gaussian factor. For $k=0$,
absorb that factor into the constant.
Applying~\eqref{down:weighted-error-good-event} and summing over $a$
yields
\[
 \left|\int\chi_kJ:D^2\varphi_t\right|
 \le C_\psi S(t)^{-2-\kappa}
                 2^{k(d+2-\kappa)}e^{-c4^k}.
\]
This is summable in $k$. Since $|J|\le C_dm$, the mass bound also
justifies summing the decomposed integrals. Equation
\eqref{down:heat-derivative-identity} proves the first estimate, and
\[
 \int_0^\infty(R^2+t)^{-1-\kappa/2}dt=\frac2\kappa R^{-\kappa}
\]
proves the second.
\end{proof}

\begin{lemma}[The infinite-time limit]\label{down:heat-endpoint}
On one probability-one event, for every fixed $R$ and
$\psi\in C_c^\infty$,
\begin{equation}
 \lim_{t\to\infty}\int m\bar p_t=1,\qquad
 \lim_{t\to\infty}\int m(\bar p_t*\psi_R)=\int\psi.
 \label{down:heat-endpoint-limit}
\end{equation}
\end{lemma}
\begin{proof}
Lemma~\ref{down:density-ergodic} gives the weighted spatial limit for
all compactly supported continuous tests on one event, and
\begin{equation}
 \int_{B_L}m\le C(\omega)L^d\qquad(L\ge L_0(\omega)).
 \label{down:ergodic-mass-bound}
\end{equation}
Put $L=\sqrt t$. For fixed $K$, apply that limit to a continuous cutoff
of $\bar p_1(x/L)$. Its tail is controlled by
\[
 \int_{|x|>KL}m(x)L^{-d}e^{-c|x|^2/L^2}dx
 \le C(\omega)\sum_{j\ge0}(2^{j+1}K)^d e^{-c4^jK^2}.
\]
Letting first $L\to\infty$ and then $K\to\infty$ proves
$\int m\bar p_t\to\int\bar p_1=1$.
For $\operatorname{supp}\psi\subset B_1$ and $t\ge R^2$, the
Gaussian gradient estimate gives
\[
 \left|\varphi_t(x)-\left(\int\psi\right)\bar p_t(x)\right|
 \le C_\psi\frac R{\sqrt t}\,t^{-d/2}e^{-c|x|^2/t}.
\]
Integrating against $m$ and using~\eqref{down:ergodic-mass-bound} on
annuli bounds the error by $C_\psi C(\omega)R/\sqrt t\to0$.
A larger fixed support only changes the constant multiplying $R$.
\end{proof}

\begin{proof}[Proof of Theorem~\ref{down:density-quantitative-main}]
Intersect $\mathcal E_R$ with the probability-one event in
Lemma~\ref{down:heat-endpoint}. Integrating
\eqref{down:heat-derivative-bound} over $[0,\infty)$ and using
\eqref{down:heat-endpoint-limit}, we obtain
\begin{equation}
 \left|\int\psi_R(m-1)\right|
 =|F(0)-F(\infty)|\le C_\psi R^{-\kappa}.
 \label{down:mass-average-bound}
\end{equation}
Apply~\eqref{down:weighted-error-good-event} to $\psi_R$ on $B_{2R}$
and to each member of a Frobenius basis to get
$|\int\psi_RJ|\le C_\psi R^{-\kappa}$. The identity
$mA-\bar A=J+\bar A(m-1)$ then gives $Z_R(\psi)\le C_\psi R^{-\kappa}$.
Take $\varsigma=\delta_2/4$, so $\kappa=\delta_2/4$, and use
\eqref{down:all-scales-probability}.

For arbitrary compact support, choose $c_\psi\ge1$ with
$\operatorname{supp}\psi\subset B_{c_\psi}$ and put
$\widetilde\psi(x)=c_\psi^d\psi(c_\psi x)$. Then
$\widetilde\psi\in C_c^\infty(B_1)$ and
$\psi_R=\widetilde\psi_{c_\psi R}$. This changes only the constants
and the initial scale.
\end{proof}

\subsection{First-moment rates}

\begin{lemma}\label{down:positive-variable-lone}
If $X\ge0$, $\mathbb EX=a<\infty$, and
$\mathbb P(|X-a|>t)\le p$, then
\begin{equation}
 \mathbb E|X-a|\le2t+2ap.
 \label{down:positive-variable-bound}
\end{equation}
\end{lemma}
\begin{proof}
Since $\mathbb E(X-a)=0$, we have
$\mathbb E|X-a|=2\mathbb E(a-X)_+$. The negative part is bounded by
$t$ on $\{|X-a|\le t\}$ and by $a$ on its complement.
\end{proof}

\begin{corollary}[$L^1$ rates for averages]\label{down:density-lone}
For every fixed $\psi\in C_c^\infty(\mathbb R^d)$,
\begin{equation}
 \mathbb E Z_R(\psi)\le C_\psi R^{-\kappa_3},\qquad
 \kappa_3=\min\{\beta_3,\gamma_3\}>0.
 \label{down:density-lone-rate}
\end{equation}
\end{corollary}
\begin{proof}
For $\psi\ge0$, take $X_R=\int\psi_Rm\ge0$ with
$\mathbb EX_R=\int\psi$. Theorem~\ref{down:density-quantitative-main}
and Lemma~\ref{down:positive-variable-lone} give the density estimate.
For a fixed vector $\xi$, apply the same argument to
\[
 X_{R,\xi}=\int\psi_Rm\,\xi\cdot A\xi\ge0,\qquad
 \mathbb EX_{R,\xi}=(\xi\cdot\bar A\xi)\int\psi.
\]
Its deviation is bounded by $|\xi|^2$ times the matrix deviation in
$Z_R$. Taking $\xi=e_i$ and $e_i+e_j$ and using
\[
 2B_{ij}=(e_i+e_j)\cdot B(e_i+e_j)-e_i\cdot Be_i-e_j\cdot Be_j
\]
recovers the matrix estimate. For signed $\psi$, choose a smooth
nonnegative compactly supported $\eta$ equal to $1$ near its support
and $K\ge\|\psi\|_\infty$. Apply the nonnegative result to
$\psi+K\eta$ and $K\eta$, and subtract.
\end{proof}

\section{Dirichlet problems and applications}
\label{down:sec-dirichlet}\label{sec:applications}

We transfer the cell estimate to a bounded $C^{2,\eta}$ domain $U$,
$0<\eta<1$, with deterministic data $f\in C^{0,\eta}(\overline U)$ and
$g\in C^{2,\eta}(\overline U)$. Write
\[
 N(f,g)=\|f\|_{C^{0,\eta}(\overline U)}+\|g\|_{C^{2,\eta}(\overline U)},
 \qquad A_\varepsilon(x)=A(x/\varepsilon),
\]
and let
\begin{equation}
 \begin{cases}
 -A_\varepsilon:D^2u_\varepsilon=f&\text{in }U,\\
 u_\varepsilon=g&\text{on }\partial U,
 \end{cases}
 \qquad
 \begin{cases}
 -\bar A:D^2u_0=f&\text{in }U,\\
 u_0=g&\text{on }\partial U.
 \end{cases}
 \label{down:dirichlet-problems}
\end{equation}
If $g$ is specified on $\partial U$, use a fixed bounded extension in its
norm. Lemma~\ref{lem:local-regularity} gives pathwise existence,
uniqueness and interior $C^{2,\theta}$ regularity for
$0<\theta<\min\{\sigma,\eta,1\}$. Only the homogenized solution requires
a deterministic Schauder bound:
\begin{equation}
 \|u_0\|_{C^{2,\eta}(\overline U)}\le C_0N(f,g),
 \label{down:homogenized-schauder}
\end{equation}
where $C_0$ depends on $U,\eta,\bar A$; see \cite{GT}.

\begin{theorem}[Dirichlet error estimate]
\label{down:dirichlet-main}
Let $\delta_2,\gamma_2$ be the exponents of Theorem~\ref{thm:cell}, and set
\begin{equation}
 a=\frac{\gamma_2}{2(d+\gamma_2)},\qquad
 q=\min\{a\eta,(1-a)\delta_2\},\qquad
 \beta_4=q/2,\qquad \gamma_4=\min\{\gamma_2/2,d\beta_4\}.
 \label{down:dirichlet-exponents}
\end{equation}
For every fixed deterministic $U,\eta,f,g$, there are deterministic
$C,\varepsilon_0>0$ such that
\begin{equation}
 \mathbb P\left(\|u_\varepsilon-u_0\|_{L^\infty(U)}
                 >CN(f,g)\varepsilon^{\beta_4}\right)
 \le C\varepsilon^{\gamma_4},\qquad 0<\varepsilon\le\varepsilon_0.
 \label{down:dirichlet-probability}
\end{equation}
\end{theorem}

The proof uses local cells and a torsion function. The local quadratic
comparison is the mechanism used in Armstrong and Smart
\cite[Section~6, Proposition~6.2]{ASquant} to pass from a local
homogenization quantity to a Dirichlet error. Their proposition treats
uniformly elliptic fully nonlinear equations and Lipschitz data. We use
the global Schauder bound for the linear homogenized equation and prove
the comparison below for the stated H\"older source. The torsion function
absorbs its residual, including near the boundary, using precisely H5.

\subsection{A torsion estimate at the critical moment}

Let $T_\varepsilon\ge0$ solve
\begin{equation}
 -A_\varepsilon:D^2T_\varepsilon=1\quad\text{in }U,
 \qquad T_\varepsilon=0\quad\text{on }\partial U.
 \label{down:torsion-equation}
\end{equation}

\begin{lemma}[Torsion moment]
\label{down:torsion-moment}
We have
\begin{equation}
 \|T_\varepsilon\|_\infty
 \le C_d\operatorname{diam}(U)
 \left(\int_U\lambda_{\min}(A(x/\varepsilon))^{-d}\,dx\right)^{1/d},
 \label{down:torsion-abp}
\end{equation}
and consequently
\begin{equation}
 \sup_{0<\varepsilon\le1}\mathbb E\|T_\varepsilon\|_\infty^d
 \le C_U\mathbb E[\lambda_*(0)^{-d}],\qquad
 \mathbb P(\|T_\varepsilon\|_\infty>\varepsilon^{-s})
 \le C_U\varepsilon^{sd}\quad(s>0).
 \label{down:torsion-tail}
\end{equation}
\end{lemma}

\begin{proof}
We give the contact-set proof of the Alexandrov maximum estimate; see
\cite[Chapter~9]{GT} for the classical form. Keeping its determinant
weight explicit gives the dependence on degeneracy needed here.
Write $S=\max T_\varepsilon$ and $D=\operatorname{diam}(U)$, and suppose
$S>0$. If $x_0$ is a maximum point, then for $|p|<S/(2D)$ the maximum
of $T_\varepsilon(x)-p\cdot(x-x_0)$ is attained in the interior: its value
at $x_0$ is $S$, whereas its boundary values are at most $S/2$.
Each such contact point lies in $\{T_\varepsilon\ge S/2\}\Subset U$ and
satisfies
\[
 DT_\varepsilon=p,\qquad B=-D^2T_\varepsilon\ge0,\qquad
 A_\varepsilon:B=1.
\]
The arithmetic--geometric mean inequality gives
\[
 \det B
 =\frac{\det(A_\varepsilon^{1/2}BA_\varepsilon^{1/2})}
        {\det A_\varepsilon}
 \le d^{-d}\lambda_{\min}(A_\varepsilon)^{-d}.
\]
The contact set is measurable by testing the upper supporting inequality
on a countable dense set. The gradient is $C^1$ on a neighborhood of the
compact set $\{T_\varepsilon\ge S/2\}$ containing the contact set.
The image of the contact set contains $B_{S/(2D)}$, so the area formula
yields
\[
 |B_{S/(2D)}|
 \le\int_{\text{contact set}}|\det D^2T_\varepsilon|
 \le d^{-d}\int_U\lambda_{\min}(A_\varepsilon)^{-d},
\]
proving \eqref{down:torsion-abp}. Tonelli and stationarity give
\[
 \mathbb E\int_U\lambda_{\min}(A(x/\varepsilon))^{-d}\,dx
 =|U|\mathbb E[\lambda_{\min}(A(0))^{-d}]
 \le|U|\mathbb E[\lambda_*(0)^{-d}].
\]
Raising the ABP bound to the power $d$ and applying Markov proves
\eqref{down:torsion-tail}.
\end{proof}

\subsection{Comparison with local cells}

For $0<h\le1$, choose deterministic centers $\{x_i\}_{i\in I_h}$ with
\begin{equation}
 \overline U\subset\bigcup_{i\in I_h}B_{h/2}(x_i),\qquad
 |I_h|\le C_Uh^{-d}.
 \label{down:dirichlet-cover}
\end{equation}
Their balls may cross $\partial U$, since the coefficients are defined
on all of $\mathbb R^d$.

\begin{lemma}[Cell--Dirichlet comparison]
\label{down:cell-dirichlet-comparison}
Let $c_{i,M}$ solve the zero-boundary problem
$-A_\varepsilon:D^2c_{i,M}=(A_\varepsilon-\bar A):M$ in $B_{4h}(x_i)$.
If, for some $K,\vartheta>0$, simultaneously for all $i,M$,
\begin{equation}
 \|c_{i,M}\|_\infty\le Kh^2\vartheta|M|,
 \label{down:local-cell-dirichlet-bound}
\end{equation}
then
\begin{equation}
 \|u_\varepsilon-u_0\|_\infty
 \le CN(f,g)\left[h^2\vartheta
             +(h^\eta+\vartheta)\|T_\varepsilon\|_\infty\right],
 \label{down:deterministic-dirichlet-error}
\end{equation}
where $C$ depends on $K,U,\eta,\bar A$.
\end{lemma}

\begin{proof}
Put $N=N(f,g)>0$ and $e=u_\varepsilon-u_0$; if $N=0$, uniqueness
makes the assertion immediate. Set
\[
 S=C_0Kh^2\vartheta N,\qquad
 \kappa=C_1Nh^\eta+8dS/h^2,
\]
where $C_1$ will exceed the residual constant below. Suppose
$E=\max_{\overline U}(e-\kappa T_\varepsilon)>2S$.
The maximum is attained at some $x_0\in U$, since both functions vanish
on $\partial U$. Choose $i$ with $x_0\in B_{h/2}(x_i)$ and put
$M=D^2u_0(x_0)$, $c=c_{i,M}$. Then $\|c\|_\infty\le S$, and $c$ is
defined on $B_h(x_0)\subset B_{4h}(x_i)$.

Write $L_\varepsilon=-A_\varepsilon:D^2$. Subtracting the equations
and freezing at $x_0$ gives
\begin{equation}
 \begin{split}
 L_\varepsilon(e-c)
 &=f(x)-f(x_0)+A_\varepsilon:
                  (D^2u_0(x)-D^2u_0(x_0)),\\
 |L_\varepsilon(e-c)|&\le C_{\rm res}Nh^\eta
 \quad\text{in }U\cap B_h(x_0).
 \end{split}
 \label{down:dirichlet-local-residual}
\end{equation}
Consider there the function
\[
 v=e-\kappa T_\varepsilon-c-4Sh^{-2}|x-x_0|^2.
\]
At $x_0$, $v(x_0)\ge E-S>S$. On the actual boundary
$\partial U\cap\overline{B_h(x_0)}$, $v\le S$; on the artificial boundary
$\partial B_h(x_0)\cap\overline U$, $v\le E-3S<v(x_0)$.
Thus $v$ has an interior maximum. But $\operatorname{tr}A_\varepsilon\le d$
and \eqref{down:dirichlet-local-residual} give
\[
 L_\varepsilon v
 \le C_{\rm res}Nh^\eta-\kappa+8Sh^{-2}\operatorname{tr}A_\varepsilon<0
\]
if $C_1>C_{\rm res}$, contradicting $D^2v\le0$ at that maximum.
Therefore $e\le\kappa T_\varepsilon+2S$. Applying the same argument to
$(-u_\varepsilon,-u_0)$ gives the other inequality.
The local estimate is uniform in $M$ and extends by homogeneity, so it
applies to $D^2u_0(x_0)$ without any new probability estimate at the
random maximum point.
\end{proof}

\begin{proof}[Proof of Theorem~\ref{down:dirichlet-main}]
Put $y_i=x_i/\varepsilon$ and $R=4h/\varepsilon$. The rescaled cell
\[
 c_{i,M}(x)=\varepsilon^2w_{M,R}
             ((x-x_i)/\varepsilon,\tau_{y_i}\omega)
\]
solves the problem of Lemma~\ref{down:cell-dirichlet-comparison}. By
stationarity and Theorem~\ref{thm:cell}, outside an event of probability
\begin{equation}
 Ch^{-d}(\varepsilon/h)^{\gamma_2},
 \label{down:cell-union-dirichlet}
\end{equation}
all these cells satisfy
\[
 \|c_{i,M}\|_\infty
 \le C_2\,4^{2-\delta_2}h^2(\varepsilon/h)^{\delta_2}|M|.
\]
Thus we may take $K=C_2\,4^{2-\delta_2}$ and
$\vartheta=(\varepsilon/h)^{\delta_2}$. Only a union bound over the
deterministic covering centers is used.

Choose $h=\varepsilon^a$ with $a,q$ as in
\eqref{down:dirichlet-exponents}. The cell radius tends to infinity, and
the probability in \eqref{down:cell-union-dirichlet} is at most
$C\varepsilon^{\gamma_2-a(d+\gamma_2)}=C\varepsilon^{\gamma_2/2}$.
On this event and $\{\|T_\varepsilon\|_\infty\le\varepsilon^{-q/2}\}$,
\eqref{down:deterministic-dirichlet-error} gives
\begin{align*}
 \|u_\varepsilon-u_0\|_\infty
 &\le CN\left[\varepsilon^{2a+(1-a)\delta_2}
       +(\varepsilon^{a\eta}+\varepsilon^{(1-a)\delta_2})
                           \varepsilon^{-q/2}\right]\\
 &\le C'N\varepsilon^{q/2}.
\end{align*}
The torsion exception has probability at most $C\varepsilon^{dq/2}$
by \eqref{down:torsion-tail}. A final union bound proves the theorem.
\end{proof}

\subsection{The reverse reduction}

\begin{proposition}[Equivalence of cell and Dirichlet rates]
\label{down:cell-dirichlet-equivalence}
Under H1--H5, a cell probability estimate with positive algebraic
exponents exists if and only if the actual Dirichlet problem has such
an error estimate for every fixed deterministic smooth domain and data.
The latter estimates need not hold for all data on one event.
\end{proposition}

\begin{proof}
The forward implication is Theorem~\ref{down:dirichlet-main}. For the
reverse implication take $U=B_1$, $q_M(x)=\tfrac12x\cdot Mx$,
$g_M=q_M$ and $f_M=-\bar A:M$. The homogenized solution is $q_M$.
Dirichlet uniqueness and scaling, with $R=\varepsilon^{-1}$, give
\begin{equation}
 u_{\varepsilon,M}(x)-q_M(x)
 =\varepsilon^2w_{M,R}(x/\varepsilon),\qquad x\in B_1.
 \label{down:reverse-cell-scaling}
\end{equation}
Apply the assumed estimates to these data for a Frobenius orthonormal
basis $E^1,\ldots,E^N$ of $\mathbb S^d$, $N=d(d+1)/2$. Take the minimum
of their finitely many positive exponents and the maximum of their
constants and initial scales. By linearity, for
$M=\sum_am_aE^a$, $|M|\le1$,
\[
 \|w_{M,R}\|_\infty
 \le\sum_a|m_a|\|w_{E^a,R}\|_\infty
 \le\sqrt N\max_a\|w_{E^a,R}\|_\infty.
\]
The finite union bound proves the uniform cell estimate.
\end{proof}

The proof treats nonconstant sources and the boundary through the
torsion estimate~\eqref{down:torsion-tail}, using only H5 to control
degeneracy.


We give two applications: convergence in the energy defined by the
invariant density, and an approximation of the effective matrix using
only finite-domain problems with known data.

\subsection{A boundary energy identity}

\begin{proposition}[Energy identity]
\label{em:energy}
Let $U$ be bounded and $C^2$, let $A$ be measurable and symmetric with
$0<A\le I$, and let $m\ge0$ be locally integrable on a neighborhood of
$\overline U$. Suppose
\begin{equation}
 \int_UmA:D^2\zeta=0\qquad(\zeta\in C_c^\infty(U)).
 \label{em:eq:2-1}
\end{equation}
If $e\in C^2(\overline U)$, $e=0$ on $\partial U$, and $-A:D^2e=F$ in
$U$, then both integrals below are absolutely convergent and
\begin{equation}
 \int_Um\nabla e\cdot A\nabla e=\int_UmeF.
 \label{em:eq:2-2}
\end{equation}
\end{proposition}

\begin{proof}
By compactly supported mollification, \eqref{em:eq:2-1} extends to
$C_c^2(U)$: the integral difference is bounded by
$\sqrt d\|D^2(\zeta_n-\zeta)\|_\infty\int_Km$ on a fixed compact set
$K\Subset U$.
Let $d_U=\operatorname{dist}(\cdot,\partial U)$, which is $C^2$ in an
interior tubular neighborhood. Choose a smooth $\chi$ equal to zero on
$(-\infty,1]$ and one on $[2,\infty)$, with $0\le\chi\le1$, and set
$\chi_t=\chi(d_U/t)$ near the boundary, extending it by one inside.
For small $t$,
\begin{equation}
 \chi_t\in C_c^2(U),\qquad
 |\nabla\chi_t|\le C_Ut^{-1},\qquad |D^2\chi_t|\le C_Ut^{-2},
 \label{em:eq:2-3}
\end{equation}
with derivatives supported in $\{t<d_U<2t\}$. Writing
$K_e=\|e\|_{C^2(\overline U)}$, the zero boundary values give
\begin{equation}
 |e(x)|\le K_ed_U(x)
 \label{em:eq:2-4}
\end{equation}
in this neighborhood. Test with $\chi_te^2$:
\begin{equation}
 0=\int_Um\chi_t A:D^2(e^2)
   +\int_Ume^2A:D^2\chi_t
   +4\int_Ume\nabla\chi_t\cdot A\nabla e.
 \label{em:eq:2-5}
\end{equation}
The last two terms are bounded in absolute value by
\begin{equation}
 C_UK_e^2\int_{\{d_U<2t\}}m\longrightarrow0.
 \label{em:eq:2-6}
\end{equation}
Dominated convergence in the first term gives
$0=\int_UmA:D^2(e^2)=2\int_Um(eA:D^2e+\nabla e\cdot A\nabla e)$,
which proves the identity. Boundedness of $A,e,D^2e,\nabla e$ and
$m\in L^1(U)$ gives absolute convergence; no trace of $m$ is needed.
\end{proof}

\subsection{Weighted gradient convergence}

Retain \eqref{down:dirichlet-problems}, and put
$e_\varepsilon=u_\varepsilon-u_0$, $m_\varepsilon(x)=m(x/\varepsilon)$
and $N=N(f,g)$. For each realization and $\varepsilon>0$, uniform
ellipticity on the fixed compact set $\varepsilon^{-1}\overline U$
and H\"older continuity from Lemma~\ref{lem:local-regularity},
together with the global Schauder theorem \cite[Theorem~6.14]{GT}, give
\begin{equation}
 u_\varepsilon\in C^{2,\theta}(\overline U),\qquad
 0<\theta<\min\{\eta,\sigma,1\}.
 \label{em:classical}
\end{equation}
Uniqueness identifies this solution with the viscosity solution. We use
this pathwise regularity only to justify the energy identity. The
ellipticity constant required by the cited Schauder theorem is the
positive minimum on the fixed compact set $\varepsilon^{-1}\overline U$;
it may depend on the realization and on $\varepsilon$. No Schauder
constant for $u_\varepsilon$ enters the error estimate.

\begin{corollary}[Weighted gradient error]
\label{em:weighted}
Let $\beta_4,\gamma_4$ be as in Theorem~\ref{down:dirichlet-main} and
let $\gamma_3$ be the probability exponent of
Theorem~\ref{down:density-quantitative-main}. For every fixed
deterministic $U,\eta,f,g$ there are $C,\varepsilon_0>0$ such that
\begin{equation}
 \mathbb P\left(\|m_\varepsilon^{1/2}A_\varepsilon^{1/2}
                    \nabla e_\varepsilon\|_{L^2(U)}
                    >CN\varepsilon^{\beta_4/2}\right)
 \le C\varepsilon^{\min\{\gamma_4,\gamma_3\}},
 \qquad 0<\varepsilon\le\varepsilon_0.
 \label{em:eq:3-4}
\end{equation}
\end{corollary}

\begin{proof}
Scaling the adjoint equation gives
\[
 \int m_\varepsilon A_\varepsilon:D^2\zeta
 =\varepsilon^{d-2}\int m(y)A(y):D_y^2[\zeta(\varepsilon y)]\,dy=0
\]
for $\zeta\in C_c^\infty(\mathbb R^d)$. Subtracting the Dirichlet equations,
\begin{equation}
 -A_\varepsilon:D^2e_\varepsilon
 =(A_\varepsilon-\bar A):D^2u_0=:F_\varepsilon\quad\text{in }U,
 \qquad e_\varepsilon=0\quad\text{on }\partial U,
 \label{em:eq:3-5}
\end{equation}
where \eqref{down:homogenized-schauder} gives
\begin{equation}
 \|F_\varepsilon\|_\infty
 \le(\sqrt d+|\bar A|)\|D^2u_0\|_\infty\le CN.
 \label{em:eq:3-6}
\end{equation}
By \eqref{em:classical}, Proposition~\ref{em:energy} applies:
\begin{equation}
 \int_Um_\varepsilon\nabla e_\varepsilon\cdot A_\varepsilon\nabla e_\varepsilon
 =\int_Um_\varepsilon e_\varepsilon(A_\varepsilon-\bar A):D^2u_0.
 \label{em:eq:3-7}
\end{equation}
Thus, with deterministic $C$,
\begin{equation}
 \|m_\varepsilon^{1/2}A_\varepsilon^{1/2}
                  \nabla e_\varepsilon\|_{L^2(U)}^2
 \le CN\|e_\varepsilon\|_\infty\int_Um_\varepsilon.
 \label{em:eq:3-8}
\end{equation}
Choose deterministic $\psi\in C_c^\infty(\mathbb R^d)$ with $\psi\ge0$
and $\psi\ge1$ on $\overline U$. For $R=\varepsilon^{-1}$,
\begin{equation}
 \int_Um_\varepsilon\le\int\psi(x)m(x/\varepsilon)\,dx
 =\int\psi_R(y)m(y)\,dy.
 \label{em:eq:3-9}
\end{equation}
The last quantity is at most $\int\psi+C\varepsilon^{\beta_3}\le\int\psi+1$
outside an event of probability $C\varepsilon^{\gamma_3}$, by
Theorem~\ref{down:density-quantitative-main}. On its intersection with
the Dirichlet good event, \eqref{em:eq:3-8} is bounded by
$CN^2\varepsilon^{\beta_4}$. Taking square roots and a union bound proves
the claim.
\end{proof}

\begin{remark}[A rate using only the first moment]
\label{em:firstmoment}
Stationarity, normalization and Tonelli give
$\mathbb E\int_Um_\varepsilon=|U|$. Thus, for $0<a<\beta_4$,
$\mathbb P(\int_Um_\varepsilon>\varepsilon^{-a})\le|U|\varepsilon^a$.
Combining this with \eqref{em:eq:3-8} and the Dirichlet estimate gives
\begin{equation}
 \mathbb P\left(\|m_\varepsilon^{1/2}A_\varepsilon^{1/2}
                  \nabla e_\varepsilon\|_{L^2(U)}
             >CN\varepsilon^{(\beta_4-a)/2}\right)
 \le C\varepsilon^{\min\{\gamma_4,a\}}.
 \label{em:eq:3-10}
\end{equation}
A positive weighted energy rate therefore requires no quantitative
average estimate or higher moment of $m$. These bounds do not give an
unweighted $H^1$ estimate.
\end{remark}

\subsection{A finite-domain effective matrix}

Let $T_R,V_{ij,R}$ solve
\begin{equation}
 \begin{cases}
 -A:D^2T_R=1&\text{in }B_R,\\
 T_R=0&\text{on }\partial B_R,
 \end{cases}
 \qquad
 \begin{cases}
 -A:D^2V_{ij,R}=A_{ij}&\text{in }B_R,\\
 V_{ij,R}=0&\text{on }\partial B_R.
 \end{cases}
 \label{em:eq:4-1}
\end{equation}
The same pathwise Schauder theory gives unique solutions in
$C^{2,\theta}(\overline{B_R})$. Define
\begin{equation}
 \widehat A_{R,ij}=V_{ij,R}(0)/T_R(0).
 \label{em:eq:4-2}
\end{equation}
These problems involve only the original coefficients.

\begin{proposition}[Finite-domain matrix estimate]
\label{em:matrix}
Set $\kappa_R=\min_{\overline{B_R}}\lambda_{\min}A>0$.
Then $\widehat A_R$ is well defined, symmetric and measurable, with
\begin{equation}
 \kappa_RI\le\widehat A_R\le I.
 \label{em:eq:4-3}
\end{equation}
For $X_R=R^{-2}\sup_{|M|\le1}\|w_{M,R}\|_{L^\infty(B_R)}$,
\begin{equation}
 |\widehat A_R-\bar A|\le2dX_R.
 \label{em:eq:4-4}
\end{equation}
\end{proposition}

\begin{proof}
The barriers $q_-=(R^2-|x|^2)/(2d)$ and $q_+=q_-/\kappa_R$ satisfy
$-A:D^2q_-\le1\le-A:D^2q_+$. Comparison gives
\begin{equation}
 q_-\le T_R\le q_+,\qquad T_R(0)\ge R^2/(2d)>0.
 \label{em:eq:4-5}
\end{equation}
Symmetry follows from uniqueness and $A_{ij}=A_{ji}$. For
$V_{\xi,R}=\sum_{i,j}\xi_i\xi_jV_{ij,R}$, comparison of its source
$\xi\cdot A\xi$ with $\kappa_R|\xi|^2$ and $|\xi|^2$ gives
$\kappa_R|\xi|^2T_R\le V_{\xi,R}\le|\xi|^2T_R$. Evaluation at zero
proves \eqref{em:eq:4-3}.

For $M\in\mathbb S^d$, let $V_{M,R}=\sum_{i,j}M_{ij}V_{ij,R}$.
Linearity and zero-boundary uniqueness give
\begin{equation}
 V_{M,R}-(\bar A:M)T_R=w_{M,R}.
 \label{em:eq:4-6}
\end{equation}
Therefore
\[
 |(\widehat A_R-\bar A):M|
 =|w_{M,R}(0)|/T_R(0)
 \le2dR^{-2}\|w_{M,R}\|_\infty.
\]
Taking the supremum over $|M|\le1$ proves \eqref{em:eq:4-4}.

For measurability, suppose admissible coefficient restrictions $A_n$
converge uniformly to $A$ on $\overline{B_R}$. Eventually
$A_n\ge\kappa_RI/2$. If $u_n,u$ are the defining zero-boundary solutions
with $F_n\to F$ uniformly, then
$-A_n:D^2(u_n-u)=F_n-F+(A_n-A):D^2u$.
Using the quadratic upper barrier for both signs gives
\[
 \|u_n-u\|_\infty\le\frac{R^2}{d\kappa_R}
 \bigl(\|F_n-F\|_\infty+\|A_n-A\|_\infty\|D^2u\|_\infty\bigr)
 \longrightarrow0.
\]
This applies to $F=1$ and $F=A_{ij}$, proving continuity of the solution
maps, their evaluations and their ratio. The canonical restriction map
is measurable, hence so is $\widehat A_R$.
\end{proof}

\begin{corollary}[Accuracy of the finite-domain matrix]
\label{em:matrixrate}
With the constants of Theorem~\ref{thm:cell},
\begin{equation}
 \mathbb P(|\widehat A_R-\bar A|>2dC_2R^{-\delta_2})
 \le C_2R^{-\gamma_2},\qquad R\ge R_2.
 \label{em:eq:4-8}
\end{equation}
\end{corollary}

\begin{proof}
By \eqref{em:eq:4-4}, the exceptional event is contained in
$\{X_R>C_2R^{-\delta_2}\}$; apply Theorem~\ref{thm:cell}.
\end{proof}

At most $1+d(d+1)/2$ scalar problems are required. Stationarity gives
the same conclusion at each prescribed center, with evaluation at that
center. This is an error estimate for the continuum construction;
numerical discretization introduces a separate error.

\section*{Use of generative artificial intelligence}
OpenAI ChatGPT was used to assist with drafting mathematical arguments,
checking calculations, locating references, organizing the exposition,
and preparing the \LaTeX\ source.



\begin{thebibliography}{99}

\bibitem{AFL}
Armstrong, S., Fehrman, B., Lin, J.: Green function and invariant measure estimates for nondivergence form elliptic homogenization.
Preprint, arXiv:2211.13279v2 (2022; revised 2025).
\url{https://arxiv.org/abs/2211.13279v2}

\bibitem{AKM19}
Armstrong, S., Kuusi, T., Mourrat, J.-C.: Quantitative Stochastic Homogenization and Large-Scale Regularity.
Grundlehren der mathematischen Wissenschaften, vol. 352.
Springer, Cham (2019).
\url{https://doi.org/10.1007/978-3-030-15545-2}

\bibitem{AL}
Armstrong, S., Lin, J.: Optimal quantitative estimates in stochastic homogenization for elliptic equations in nondivergence form. Arch. Ration. Mech. Anal. \textbf{225}, 937--991 (2017).
\url{https://doi.org/10.1007/s00205-017-1118-z}

\bibitem{ASconvex16}
Armstrong, S.N., Smart, C.K.: Quantitative stochastic homogenization of convex integral functionals.
Ann. Sci. \'Ec. Norm. Sup\'er. (4) \textbf{49}, 423--481 (2016).
\url{https://doi.org/10.24033/asens.2287}

\bibitem{ASquant}
Armstrong, S.N., Smart, C.K.: Quantitative stochastic homogenization of elliptic equations in nondivergence form. Arch. Ration. Mech. Anal. \textbf{214}, 867--911 (2014).
\url{https://doi.org/10.1007/s00205-014-0765-6}.
Corrected version: \url{https://arxiv.org/abs/1306.5340v5} (2019).

\bibitem{ASdeg}
Armstrong, S.N., Smart, C.K.: Regularity and stochastic homogenization of fully nonlinear equations without uniform ellipticity. Ann. Probab. \textbf{42}, 2558--2594 (2014).
\url{https://doi.org/10.1214/13-AOP833}

\bibitem{AvLin89}
Avellaneda, M., Lin, F.-H.: Compactness methods in the theory of homogenization. II. Equations in non-divergence form.
Comm. Pure Appl. Math. \textbf{42}, 139--172 (1989).
\url{https://doi.org/10.1002/cpa.3160420203}

\bibitem{Bauman84}
Bauman, P.: Positive solutions of elliptic equations in nondivergence form and their adjoints. Ark. Mat. \textbf{22}, 153--173 (1984).
\url{https://doi.org/10.1007/BF02384378}

\bibitem{BCDG}
Berger, N., Cohen, M., Deuschel, J.-D., Guo, X.: An elliptic Harnack inequality for difference equations with random balanced coefficients. Ann. Probab. \textbf{50}, 835--873 (2022).
\url{https://doi.org/10.1214/21-AOP1544}

\bibitem{BD14}
Berger, N., Deuschel, J.-D.: A quenched invariance principle for non-elliptic random walk in i.i.d. balanced random environment. Probab. Theory Related Fields \textbf{158}, 91--126 (2014).
\url{https://doi.org/10.1007/s00440-012-0478-4}

\bibitem{Bowman26}
Bowman, D.: Harnack inequality for non-uniformly elliptic equations in non-divergence form.
Preprint, arXiv:2604.13303v1 (2026).
\url{https://arxiv.org/abs/2604.13303v1}

\bibitem{CC95}
Caffarelli, L.A., Cabr\'e, X.: Fully Nonlinear Elliptic Equations. American Mathematical Society Colloquium Publications, vol. 43. American Mathematical Society, Providence, RI (1995).
\url{https://doi.org/10.1090/coll/043}

\bibitem{CNS}
Caffarelli, L., Nirenberg, L., Spruck, J.: The Dirichlet problem for nonlinear second-order elliptic equations. I. Monge--Amp\`ere equation. Comm. Pure Appl. Math. \textbf{37}, 369--402 (1984).
\url{https://doi.org/10.1002/cpa.3160370306}

\bibitem{CS10}
Caffarelli, L.A., Souganidis, P.E.: Rates of convergence for the homogenization of fully nonlinear uniformly elliptic PDE in random media.
Invent. Math. \textbf{180}, 301--360 (2010).
\url{https://doi.org/10.1007/s00222-009-0230-6}

\bibitem{CSW05}
Caffarelli, L.A., Souganidis, P.E., Wang, L.: Homogenization of fully nonlinear, uniformly elliptic and parabolic partial differential equations in stationary ergodic media.
Comm. Pure Appl. Math. \textbf{58}, 319--361 (2005).
\url{https://doi.org/10.1002/cpa.20069}

\bibitem{CIL92}
Crandall, M.G., Ishii, H., Lions, P.-L.: User's guide to viscosity solutions of second order partial differential equations. Bull. Amer. Math. Soc. (N.S.) \textbf{27}, 1--67 (1992).
\url{https://doi.org/10.1090/S0273-0979-1992-00266-5}

\bibitem{FO}
Fischer, J., Otto, F.: A higher-order large-scale regularity theory for random elliptic operators. Comm. Partial Differential Equations \textbf{41}, 1108--1148 (2016).
\url{https://doi.org/10.1080/03605302.2016.1179318}

\bibitem{Gehring73}
Gehring, F.W.: The $L^p$-integrability of the partial derivatives of a quasiconformal mapping. Acta Math. \textbf{130}, 265--277 (1973).
\url{https://doi.org/10.1007/BF02392268}

\bibitem{GT}
Gilbarg, D., Trudinger, N.S.: Elliptic Partial Differential Equations of Second Order. Classics in Mathematics. Springer, Berlin (2001).
\url{https://doi.org/10.1007/978-3-642-61798-0}

\bibitem{GPT22}
Guo, X., Peterson, J., Tran, H.V.: Quantitative homogenization in a balanced random environment. Electron. J. Probab. \textbf{27}, Paper No.~132, 1--31 (2022).
\url{https://doi.org/10.1214/22-EJP851}

\bibitem{GST25}
Guo, X., Sprekeler, T., Tran, H.V.: Homogenization of non-divergence form operators in i.i.d. random environments.
Preprint, arXiv:2512.04410v2 (2025).
\url{https://arxiv.org/abs/2512.04410v2}

\bibitem{GuoTran25}
Guo, X., Tran, H.V.: Optimal convergence rates in stochastic homogenization in a balanced random environment. Probab. Theory Related Fields \textbf{193}, 821--880 (2025).
\url{https://doi.org/10.1007/s00440-025-01409-1}

\bibitem{GZ12}
Guo, X., Zeitouni, O.: Quenched invariance principle for random walks in balanced random environment. Probab. Theory Related Fields \textbf{152}, 207--230 (2012).
\url{https://doi.org/10.1007/s00440-010-0320-9}

\bibitem{PV82}
Papanicolaou, G.C., Varadhan, S.R.S.: Diffusions with random coefficients.
In: Kallianpur, G., Krishnaiah, P.R., Ghosh, J.K. (eds.)
Statistics and Probability: Essays in Honor of C.R. Rao, pp. 547--552.
North-Holland, Amsterdam (1982).

\bibitem{SV}
Stroock, D.W., Varadhan, S.R.S.: Multidimensional Diffusion Processes. Classics in Mathematics. Springer, Berlin (2006). Reprint of the 1979 edition.
\url{https://doi.org/10.1007/3-540-28999-2}

\bibitem{Timar}
Tim\'ar, \'{A}.: Boundary-connectivity via graph theory. Proc. Amer. Math. Soc. \textbf{141}, 475--480 (2013).
\url{https://doi.org/10.1090/S0002-9939-2012-11333-4}

\bibitem{Whitt}
Whitt, W.: Proofs of the martingale FCLT. Probab. Surv. \textbf{4}, 268--302 (2007).
\url{https://doi.org/10.1214/07-PS122}

\bibitem{Yur82}
Yurinskii, V.V.: Averaging of second-order nondivergent equations with random coefficients.
Siberian Math. J. \textbf{23}, 276--287 (1982).
\url{https://doi.org/10.1007/BF00971701}

\end{thebibliography}
\end{document}